\documentclass[12pt,a4paper]{amsart}
\usepackage{amsfonts}
\usepackage{amscd}
\usepackage{amsmath,amssymb}
\usepackage{latexsym}
\usepackage{CJK}
\usepackage{amsthm}
\usepackage{bbold}
\usepackage{xcolor}
\usepackage[margin=2.9cm]{geometry}

\usepackage[utf8]{inputenc}
\usepackage[english]{babel}
\usepackage{calrsfs}
\usepackage[mathscr]{euscript}
\usepackage{young}
\usepackage{tikz}
\usepackage{graphicx}
\usepackage[vcentermath]{youngtab}

\usepackage{hyperref}
\hypersetup{
	colorlinks,
	citecolor=black,
	filecolor=black,
	linkcolor=black,
	urlcolor=black
}

\newtheorem{thm}{Theorem}[section]
\newtheorem{definition}[thm]{Definition}
\newtheorem{eg}[thm]{Example}
\newtheorem{prop}[thm]{Proposition}
\newtheorem{rmk}[thm]{Remark}
\newtheorem{lemma}[thm]{Lemma}
\newtheorem{cor}[thm]{Corollary}

\newtheorem{claim}[thm]{Claim}
\newtheorem{property}{Property}

\newcommand{\1}{\mathbb{1}}
\newcommand{\C}{\mathbb{C}}
\newcommand{\D}{\mathcal{D}}
\newcommand{\F}{\mathcal{F}}

\newcommand{\N}{\mathbb{N}}
\newcommand{\Y}{\mathcal{Y}}
\newcommand{\Z}{\mathbb{Z}}

\newcommand{\ci}{\mathfrak{i}}
\newcommand{\cj}{\mathfrak{j}}
\newcommand{\gl}{\mathfrak{gl}}
\newcommand{\s}{\mathfrak{s}}
\newcommand{\co}{\mathcal{M}}
\newcommand{\h}{\mathfrak{h}}
\newcommand{\ft}{\mathfrak{t}}
\newcommand{\m}{\mathfrak{m}}
\newcommand{\0}{\mathfrak{0}}
\newcommand{\tab}{Tab(\varphi^{\xi}_{n,p,\mu})}

\begin{document}
	
	\title{Representations of Degenerate Affine Hecke Algebra of Type $C_n$ Under the Etingof-Freund-Ma Functor}
	\title[Representations of degenerate AHA of type $C_n$]{Representations of Degenerate Affine Hecke Algebra of Type $C_n$ Under the Etingof-Freund-Ma Functor}
	\date{\today}
	\author{Yue Zhao}
	\maketitle
	
	\textbf{Abstract.} We compute the image of a polynomial $GL_N$-module under the Etingof-Freund-Ma functor \cite{EFM}. We give a combinatorial description of the image in terms of standard tableaux on a collection of skew shapes and analyze weights of the image in terms of contents.
	
	\tableofcontents
	
	\section{Introduction}
	Schur-Weyl duality connects polynomial representations of $GL_N$ and representations of the symmetric group $S_n$. Let $V = {\C}^N$ denote the vector representation of $GL_N$. Then $V^{\otimes n}$ has a $GL_N$-action. Let $S_n$ be the symmetric group on $n$ indices. The tensor $V^{\otimes n}$ has a natural right $S_n$-action. By the Schur-Weyl duality, we have the decomposition
	$$V^{\otimes n} = \bigoplus_{|\lambda| = n} V^{\lambda} \boxtimes S_{\lambda},$$ 
	where $n \geq N$, $\lambda$ is a partition of $n$ with at most $N$ rows, $S_{\lambda}$ runs through all irreducible representations $S_n$ and $V^{\lambda}$ is the irreducible $GL_N$-module with highest weight $\lambda$. Moreover, the actions of Jucys-Murphy elements are diagonalizable.In \cite{AS}, Arakawa and Suzuki constructed a functor from the category of $U(\gl_N)$-modules to the category of representations of the degenerate affine Hecke algebra of type $A_n$. In \cite{CEE}, Calaque, Enriquez and Etingof generalized this functor to the category of representations of degenerate double affine Hecke algebra of type $A_n$.  Etingof, Freund and Ma \cite{EFM} extended the construction to the category of representations of degenerate affine and double affine Hecke algebra of type $BC_n$ by considering the classical symmetric pair $(\gl_N, \gl_p \times \gl_{N - p})$. As a quantization of the functors by Etingof-Freund-Ma, Jordan and Ma in \cite{JM} constructed functors from the category of $U_q(\gl_N)$-modules to the category of representations of affine Hecke algebra of type $C_n$ and from the category of quantum $\D$-modules to the category of representations of the double affine Hecke algebra of type $C^{\vee}C_n$. The construction in \cite{JM} used the theory of quantum symmetric pair $(U_q(\gl_N), B_{\sigma})$ where $B_{\sigma}$ is a coideal subalgebra. This is a quantum analogue of the classical symmetric pair.\\
	On the other hand, in \cite{Ree}, Reeder did the classification of irreducible representations of affine Hecke algebra of type $C_2$ with equal parameters. In \cite{K}, Kato indexed and analyzed the weights of representations of affine Hecke algebra of type $C_n$. In \cite{M}, Ma analyzed the image of principal series modules under the Etingof-Freund-Ma functor. Moreover, the combinatorial description of Young diagrams is used to describe irreducible representations of the symmetric group and Hecke algebra of type $A$ with standard tableaux on the Young diagram indexing the bases. Similarly, the skew shape and standard tableaux on it describes the irreducible representation of the affine Hecke algebra of type $A$. Moreover, in \cite{SV}, Suzuki and Vazirani introduced a description of some irreducible representations of the double affine Hecke algebra of type $A$ by periodic skew Young diagrams and periodic standard tableaux on it. In \cite{R}, Ram introduced the chambers and local regions and described the representations of the affine Hecke algebra. In \cite{D}, Daugherty introduce the combinatorial description of representations of degenerate extended two-boundary Hecke algebra. In \cite{DR}, Daugherty and Ram gave a Schur-Weyl duality approach to the affine Hecke algebra of type $C_n$.\\
	This paper focuses on the representations of the degenerate affine Hecke algebra of type $C_n$ and gives a combinatorial description which is similar to the combinatorial description in \cite{D} and \cite{DR} but is via a different structure, the Etingof-Freund-Ma functor. This paper is arranged as follow: Section 2-4 are about the Etingof-Freund-Ma functor, the degenerate affine Hecke algebra of type $C_n$ and $GL_N$-modules. In section 5 and section 6, we compute the underlying vector space of the image of Etingof-Freund-Ma functor and the $\Y$-actions. In section 7 and section 8, we talk about intertwining operators and define combinatorial moves. Section 9 concerns the irreducibility of the image. In section 10, we talk about how to recover a $GL_N$-module from a representation of degenerate affine Hecke algebra of type $C_n$.\\
	\textbf{Acknowledgments.} I would like to thank Monica Vazirani for her guidance and helpful discussions, Arun Ram for helpful comments on my first draft and for the suggestion on presenting the $\Y$-action, José Simental Rodríguez for detailed feedback and helpful discussions.
	\section{Definitions and notations}
	\subsection{Root system of type $C_n$}
	Let ${\h}^*$ be a finite-dimensional real vector space with basis $\{\epsilon_i | i = 1, \cdots, n \}$ and a positive definite symmetric bilinear form $(\cdot, \cdot)$ such that $(\epsilon_i, \epsilon_j) = \delta_{ij}$. Let $R_n$ be an irreducible root system of type $C_n$ with
	$$R_n = \{\epsilon_i + \epsilon_j | i, j = 1, \cdots, n\}\cup \{\epsilon_i - \epsilon_j | i, j =1, \cdots, n \text{ and } i \neq j\},$$
	and the positive roots are
	$$R_{n+} = \{\epsilon_i + \epsilon_j | i,j = 1, \cdots, n\} \cup \{\epsilon_i - \epsilon_j | 1 \leq i < j \leq n\}.$$
	For any root $\alpha$, the coroot is $\alpha^{\vee} = \dfrac{2\alpha}{(\alpha, \alpha)}$. Let $Q$ be the root lattice and $Q^{\vee}$ be the coroot lattice. 
	Let $\alpha_i = \epsilon_i - \epsilon_{i + 1}$, for $i = 1, \cdots, n - 1$ and $\alpha_n = 2\epsilon_n$. Then the collection of simple roots are
	$$\Pi_n = \{\alpha_i | i = 1 ,\cdots, n\}.$$
	For each simple root $\alpha_i$, define the reflection $s_i := s_{\alpha_i}$,
	$$s_{\alpha_i}(\lambda) = \lambda - (\lambda, \alpha_i^{\vee}) \alpha_i,$$
	where $\lambda \in {\h}^*$.
	Then the finite Weyl group $W_0$ of type $C_n$ is generated by the generators $s_1, \cdots, s_{n-1},s_n$ with the relations
	\begin{align}
		&s_i^2 = 1, \text{ for } i = 1, \cdots, n,\\
		&s_i s_{i + 1}s_i =  s_{i + 1}s_is_{i + 1}, \text{ for } i = 1, \cdots, n - 1,\\
		&s_{n - 1}s_ns_{n - 1}s_n = s_ns_{n - 1}s_ns_{n - 1},\\
		&s_i s_j =  s_j s_i, \text{ for } |i - j| > 1.
	\end{align}
	\subsection{Affine Weyl group of type $C_n$}
	Let $W = W_0 \ltimes Q^{\vee}$. For any $\iota \in {\h}^*$, where $\iota = \iota_1 \epsilon_1 + \cdots + \iota_n \epsilon_n$ and $\iota_k \in \Z$, let $y^{\iota} = y_1^{\iota_1}\cdots y_n^{\iota_n}$ and the action of $w \in W_0$ by $w . y^{\iota} = y^{w(\iota)}$.
	Let $W = W_0 \ltimes Q^{\vee}$ and the affine Weyl group of type $C_n$ is generated by $s_1,\cdots, s_{n-1},s_n$ and $Y_i$, for $i = 1, \cdots, n$ with the following additional relations to (1)-(4),
	\begin{align}
		& s_i Y_j = Y_j s_i, \text{ for } j \neq i, i + 1,\\
		& Y_i Y_j = Y_j Y_i, \\
		& s_i Y_i s_i = Y_{i + 1}, \text{ for } i = 1, \cdots, n - 1,\\
		& s_n Y_n s_n = {Y_n}^{-1}.
	\end{align}
	
	\subsection{Definition of degenerate affine Hecke algebra of type $C_n$}
	Let $\kappa_1$ and $\kappa_2$ be two parameters. The trigonometric degenerate affine Hecke algebra $H_n(\kappa_1, \kappa_2)$ is an algebra generated over $\C$ by $s_1, \cdots, s_{n-1},\gamma_n$, where we take $\gamma_n=s_n$, and $y_1, \cdots, y_n$ with relations (1)-(6) and the following relations
	\begin{align}
		& s_i y_i - y_{i + 1} s_i = \kappa_1, \text{ for } i = 1, \cdots, n - 1,\\
		& \gamma_n y_n + y_n \gamma_n = \kappa_2.
	\end{align}
	
	\subsection{$\Y$-semisimple degenerate affine Hecke algebra representations}
	Now let define what we mean by $\Y$-semisimple. Let $\Y = \C[y_1, \cdots, y_n]$ be the commutative subalgebra of the degenerate affine Hecke algebra $H_n(\kappa_1, \kappa_2)$. Let $L$ be a representation of $H_n(\kappa_1, \kappa_2)$. For a function $\zeta: \{1,\cdots,n\} \to \C$, let $\zeta_i$ denote $\zeta(i)$ and $\zeta=[\zeta_1,\cdots,\zeta_n]$. Define the simultaneous generalized eigenspace as
	$$L_{\zeta}^{gen} = \{v \in L | (y_i - \zeta_i)^k v = 0 \text{ for some } k \gg 0 \text{ and for all } i = 1, \cdots, n\}.$$
	Since the polynomial algebra $\Y$ is commutative, the restriction of $L$ on $\Y$ decomposes to a sum of simultaneous generalized eigenspace, i.e. $L = \oplus_{\zeta} L_{\zeta}^{gen}$. Similarly, define the simultaneous eigenspace
	$$L_{\zeta} = \{v \in L | y_i v = \zeta_i v \text{ for all } i = 1, \cdots, n\}.$$
	\begin{definition}
		If the restriction of $L$ on $\Y$ decomposes to a sum of simultaneous eigenspaces, i.e. $L = \oplus_{\zeta} L_{\zeta}$, then call $L$ is $\Y$-semisimple. The function $\zeta$ is called a weight and $L_{\zeta}$ is the weight space of weight $\zeta$.
	\end{definition}
	
	\section{Etingof-Freund-Ma Functor}
	We recall the definition of the Etingof-Freund-Ma functor $F_{n, p, \mu}$ in \cite{EFM}. Let $N$ be a positive number and $V$ be the vector representation of $\gl_N$. Let $p, q$ be positive integers such that $N = p + q$. Let $\ft = \gl_p \times \gl_q$ and $\ft_0$ be the subalgebra in $\ft$ consisting of all the traceless elements in $\ft$. Let $\chi$ is a character defined on $\ft$ as 
	\begin{equation}
	\chi(
	\begin{bmatrix}
		S &0\\
		0 &T
	\end{bmatrix})
	= q\cdot tr(S) - p \cdot tr(T),
	\end{equation}
	where $S \in \gl_p$ and $T \in \gl_q$. For a given $\mu \in \C$, define a functor $F_{n, p, \mu}$ from the category of $\gl_N$-modules to the category of representations of degenerate affine Hecke algebra $H_n(1, p -q-\mu  N)$
	$$F_{n, p, \mu}(M) = (M \otimes V^{\otimes n})^{\ft_0, \mu},$$
	where the $(\ft_0, \mu)$-invariant corresponds $A . v = \mu \chi(A) v$, for all $A \in \ft_0$. \\
	Let $M$ be the $0$-th tensor factor. Let $V_i$ be the $i$-th tensor factor with $V_i=V$ being the vector representation for $i=1, \cdots,n$.  In \cite{JM}, the action of the degenerate affine Hecke algebra $H_n(1, p-q-\mu N)$ is the quasi classical limit of the action of the affine Hecke algebra $\mathcal{H}_n(q,q^{\sigma},q^{(p-q-\tau)})$ generated by $T_1,\cdots, T_{n-1},T_n$ and $Y_1^{\pm}, \cdots, Y_n^{\pm}$. In the following figures, $V_i$ is the vector representation for $i=1, \cdots, n$. In \cite{JM}, the action of $T_i$ for $i=1, \cdots, n-1$ was defined by $\tau_{V_i, V_{i+1}} \circ R_{i,i+1}$, where the flip operator $\tau_{V_i,V_{i+1}}: V_i \otimes V_{i+1} \to V_{i+1} \otimes V_i$ is defined by $v_i \otimes v_{i+1} \mapsto v_{i+1} \otimes v_i$ and $R_{i,i+1}$ is the $R$ matrix acting on $V_{i} \otimes V_{i+1}$,
	\begin{center}
		\begin{tikzpicture}[scale=0.6]
			\draw (-2,1) node {$T_i$};
			\draw (-1,1) node {$=$};
			\draw [thick] plot (5,0) to[bend left=10] (4.65,0.65);
			\draw [thick] (4.4,1.05) to [bend left=10] (4,2);
			\draw [thick] (4,0) to[bend right=10] (5,2);
			\draw [thick] (0,2)--(0,0);
			\draw [thick] (1,2)--(1,0);
			\draw [thick] (8,2)--(8,0);
			\draw (2.5,1)node   {$\cdots$};
			\draw (6.5,1)node   {$\cdots$};
			\draw (0, -0.5) node {\tiny{$M$}};
			\draw (0, 2.5) node {\tiny{$M$}};
			\draw (1, -0.5) node {\tiny{$V_1$}};
			\draw (1, 2.5) node {\tiny{$V_1$}};
			\draw (4, -0.5) node {\tiny{$V_i$}};
			\draw (4, 2.5) node {\tiny{$V_i$}};
			\draw (5, -0.5) node {\tiny{$V_{i+1}$}};
			\draw (5, 2.5) node {\tiny{$V_{i+1}$}};
			\draw (8, -0.5) node {\tiny{$V_n$}};
			\draw (8, 2.5) node {\tiny{$V_n$}};
		\end{tikzpicture}
	\end{center}
Let $T_i=s_ie^{ \hbar s_i/2}$. Proposition 39 in \cite{J} and section 10.7 of \cite{JM} computed the action of $s_i$, i.e. $s_i$ acts on $F_{n,p,\mu}(M)$ by exchanging the $i$-th and ${(i + 1)}$-th tensor factors.\\
The action of $T_n$ was defined as follows
	\begin{center}
	\begin{tikzpicture}[scale=0.6]
		\draw (-2,1) node {$T_n$};
		\draw (-1,1) node {$=$};
		\draw [thick] (0,2)--(0,0);
		\draw [thick] (1,2)--(1,0);
		\draw [thick] (2,2)--(2,0);
		\draw [thick] (5,2)--(5,1.4);
		\draw [thick] (5,0.6)--(5,0);
		\draw (3.5,1) node   {$\cdots$};
		\draw (0, -0.5) node {\tiny{$M$}};
		\draw (0, 2.5) node {\tiny{$M$}};
		\draw (1, -0.5) node {\tiny{$V_1$}};
		\draw (1, 2.5) node {\tiny{$V_1$}};
		\draw (2, -0.5) node {\tiny{$V_2$}};
		\draw (2, 2.5) node {\tiny{$V_2$}};
		\draw (5, -0.5) node {\tiny{$V_n$}};
		\draw (5, 2.5) node {\tiny{$V_n$}};
        \draw [thick] (4.6,1.4) rectangle (5.4,0.6);
         \draw (5,1) node {\tiny{$J_V$}};
	\end{tikzpicture}
\end{center}
where the matrix $J_V$ is a right-handed numerical solution of the reflection equation $R_{21}(J_V)_1R_{12}(J_V)_2=(J_V)_2R_{21}(J_V)_1R_{12}$ in  section 7 of \cite{JM}. Section 10.7 of \cite{JM} compute the quasi classical limit of $T_n$. Then $\gamma_n$ acts on $F_{n,p,\mu}(M)$ by multiplying the $n$-th tensor factor by $J = diag(I_p,-I_q)$.\\
The action of $Y_1$ was define by $q^{\frac{2n}{N}+\mu(q-p)-N}R^{-1}_{01} \circ \tau_{V,M} \circ R^{-1}_{10} \circ \tau_{M,V}$.
\begin{center}
	\begin{tikzpicture}[scale=0.6]
		\draw (-6.5,1) node {$Y_1$};
			\draw (-3.5,1) node {\tiny{$=q^{\frac{2n}{N}+\mu(q-p)-N}$}};
		\draw [thick] plot [smooth] coordinates {(1,0) (0.85,0.3) (-0.8,1) (-0.2,1.45)};
	    \draw [thick] plot [smooth] coordinates {(0.2,1.5) (0.9,1.7) (1,2)};
	    \draw [thick] (0,2)--(0,0.8);
	     \draw [thick] (0,0.45)--(0,0);
	      \draw [thick] (2,2)--(2,0);
	      \draw [thick] (3,2)--(3,0);
	       \draw [thick] (6,2)--(6,0);
	       \draw (4.5,1)node   {$\cdots$};
	      \draw (0, -0.5) node {\tiny{$M$}};
	       \draw (0, 2.5) node {\tiny{$M$}};
	       \draw (1, -0.5) node {\tiny{$V_1$}};
	         \draw (1, 2.5) node {\tiny{$V_1$}};
	          \draw (2, -0.5) node {\tiny{$V_2$}};
	         \draw (2, 2.5) node {\tiny{$V_2$}};
	         \draw (6, -0.5) node {\tiny{$V_n$}};
	        \draw (6, 2.5) node {\tiny{$V_n$}};
	\end{tikzpicture}
\end{center}
Let $Y_1=e^{y_1 \hbar}$. By Proposition 10.13 in \cite{JM},
\begin{equation}
y_1=-\sum_{s,t}(E_s^t)_0 \otimes (E_t^s)_1+\dfrac{n}{N}+\dfrac{\mu (q-p)}{2} - \dfrac{N}{2},
\end{equation}
	where $E_s^t$ is the $N \times N$ matrix with the $(s, t)$ entry being $1$ and other entries being $0$ and $(E_s^t)_i$ means $E_s^t$ acting on the $i$-th tensor factor.  Let $s_{k,l}$ denote the transposition $(k, l) \in S_n$ and $\gamma_k \in W_0$ denote the action multiplying the $k$-th factor by $J$. In \cite{EFM}, the action of $y_1$ is given by 
	\begin{equation}
	-\sum_{s|t} (E_s^t)_0 \otimes (E_t^s)_1 + \dfrac{p - q - \mu N}{2} \gamma_1 + \dfrac{1}{2}\sum_{l > 1} s_{1,l} + \dfrac{1}{2}\sum_{l \neq 1} s_{1,l} \gamma_1 \gamma_l,
	\end{equation}
	where $\sum_{s | t} = \sum_{s = 1}^{p} \sum_{t = p + 1}^{n} + \sum_{t = 1}^p \sum_{s = p + 1}^n$. In section 6.1, we show that the computation via equation (13) agrees with equation (12). By the relation $y_k = s_{k-1}y_{k-1}s_{k-1} - s_{k-1}$, we could compute the action of $y_k$ for $k=1, \cdots, n$.
	\section{$GL$-module}
	We consider images of polynomial $GL_N$-modules under Etingof-Freund-Ma functor.
	Recall the facts about polynomial $GL_N$-modules. Let $M$ be a polynomial $GL_N$-module and $H \subset GL_N$ be the collection of invertible diagonal matrices. Let $v \in M$ satisfy 
	$$x.v=x_1^{\lambda_1}\cdots x_N^{\lambda_N}v,$$
	for any $x=diag(x_1,\cdots,x_N) \in H$. Then $v$ is a weight vector of $H$-weight $\lambda=(\lambda_1,\cdots,\lambda_N)$. The subspace $$M(\lambda)= \{v \in M | x.v=x_1^{\lambda_1}\cdots x_N^{\lambda_N}v, x \in H\}$$ is called the weight space of weight $\lambda$. Then the polynomial $GL_N$-module $M$ is a direct sum of weight spaces
	$$M=\bigoplus M(\lambda).$$
	Let $B \subset GL_N$ be the collection of all invertible upper triangular matrices. Let $v \in M$ be a generator of $M$. If $v$ satisfies $x.v=c(x)v$ for some function $c(x)$ and any $x \in B$, then $v$ is called a highest weight vector. If $M$ has the unique highest weight vector up to a scalar of the highest weight $\xi$, then $M$ is a highest weight module with the highest weight $\xi$ and let us denote $M$ by $V^{\xi}$. A $GL_N$-module $M$ is irreducible if and only if $M$ is a highest weight $GL_N$-module. Furthermore, two highest weight $GL_N$-modules are isomorphic if and only if they have the same highest weight. Let $\xi=\sum_{i=1}^N \xi_i \epsilon_i$ satisfying $\xi_1\geq \xi_2 \geq \cdots \geq \xi_N$ and $\xi_i \in \Z$ for $i=1,\cdots,N$. Then $\xi$ is an integral dominant weight of $GL_N$. Let $P^+$ denote the collection of all integral dominant weights and $P^+_{\geq 0}$ denote the collection of all integral dominant weights $\xi=\sum_{i=1}^N \xi_i \epsilon_i$ with $\xi_i \in \N$, for $i=1,\cdots,N$. Then the highest weight modules with highest weights $\xi \in P^+_{\geq 0}$ are all the irreducible polynomial $GL_N$-modules. Let $M$ be a rational $GL_N$-module. Then $M= det^m \otimes N$ for some $m \in \Z$ and a polynomial $GL_N$-module $N$. Then the highest weight modules with integral dominant highest weights are all the irreducible rational $GL_N$-modules.\\
	The collection $P^+_{\geq 0}$ has a one-to-one correspondence with the collection of partitions with at most $N$ parts and thus the one-to-one correspondence with Young diagrams with at most $N$ rows. For the ease of writing, for each irreducible polynomial $GL_N$-module $V^{\xi}$ with highest weight $\xi \in P^+_{\geq 0}$, let us denote the corresponding partition $(\xi_1, \cdots,\xi_N)$ and Young diagram also by $\xi$. Moreover, define $|\xi|=\sum_{i=1}^N \xi_i$ for $\xi \in P^+$.\\
	For a highest weight $GL_N$-module $V^{\xi}$, $\xi \in P^+_{\geq 0}$, with weight space decomposition $V^{\xi}=\bigoplus V^{\xi}(\lambda)$, the character of $V^{\xi}$
	$$\chi_{V^{\xi}}=\sum_{\lambda} dim(V^{\xi}(\lambda)) x_1^{\lambda_1}\cdots x_N^{\lambda_N}$$
	is the Schur polynomial $s_{\xi}(x_1, \cdots, x_N)$ of shape $\xi$.\\
	By Pieri's rule, $$s_{\xi}e_1=\sum_{\nu} s_{\nu},$$
	where $\nu \in P^+_{\geq 0}$ runs through all the shapes obtained by adding a cell to some row of $\xi$. Observe that $e_1=s_{\xi}$, where $\xi=(1)$, is the character of the vector representation $V$ of $GL_N$. This fact indicates how the tensor product of an irreducible polynomial $GL_N$-module and vector representation decomposes into a sum of irreducible polynomial $GL_N$-modules.
	\section{Invariant space}
In this section, we compute the underlying vector space $F_{n,p,\mu}(V^{\xi})=(M \otimes V^{\otimes n})^{\ft_0, \mu}$ by finding a special basis of it and then index the basis elements by a collection of standard tableaux.\\
	\subsection{Definition of the invariant space}~\\

	Let $M$ be a $GL_N$-module, then $M$ has a $\gl_N$-module structure. For any $X \in \gl_N$ and $v \in M$,
	$$X.v = \frac{d}{dt}(e^{tX}.v)_{t=0}.$$
	Recall the notations, $K=GL_p \times GL_q$, $Lie(K)=\ft$ and $\ft_0 \subset \ft$ which is the collection of traceless matrices in $\ft$.
	\begin{prop}
		The underlying vector space is invariant under tensoring powers of the determinant representation, i.e. $(det^{m} \otimes M\otimes V^{\otimes n})^{\ft_0, \mu} \cong (M \otimes V^{\otimes n})^{\ft_0, \mu}$, for any $m \in \C$.
	\end{prop}
	\begin{proof}
		Take any element from $(det^{m} \otimes M\otimes V^{\otimes n})^{\ft_0, \mu}$, we can denote it by $\1 \otimes w$, where $w \in M \otimes V^{\otimes n}$. According to the definition of invariant space
		\begin{align*}
			&(det^{m} \otimes M\otimes V^{\otimes n})^{\ft_0, \mu}\\
			= &\{ \1\otimes w | A . (\1\ \otimes w) =  \mu \chi(A) (\1 \otimes w) \text{, for any }A \in \ft_0 \}.
		\end{align*}
		Compute the action of $A \in \ft_0$
		\begin{align*}
			A . \1 &= \dfrac{d}{dt}(e^{tA} . \1)_{t=0} \\
			&= \dfrac{d}{dt}(det^{m}(e^{tA}))_{t=0} . \1 \\
			&= \dfrac{d}{dt}(e^{m \cdot tr(tA)})_{t=0} . \1 = 0,
		\end{align*}
		since $tr(A)=0$. Then it follows
		\begin{align*}
			A . (\1 \otimes w) &= (A . \1) \otimes w + \1 \otimes (A . w)\\
			&= \1 \otimes (A . w).
		\end{align*}
		Hence 
		\begin{align*}
			&(det^{m} \otimes M\otimes V^{\otimes n})^{\ft_0, \mu}\\
			=& \{ \1\otimes w | \1 \otimes (A . w) = \mu \chi(A) (\1 \otimes w) \text{, for any }A \in \ft_0 \} \\
			\cong & \{w | A . w = \mu \chi(A) w \text{, for any }A \in \ft_0 \} \\
			=& (M \otimes V^{\otimes n})^{\ft_0, \mu} .
		\end{align*}
	\end{proof}
	\begin{rmk}
		For an irreducible rational $GL_N$-module $M$, we could write $M = det^m \otimes V^{\xi}$ for some integer $m$ and some highest weight module $V^{\xi}$ with the highest weight $\xi \in P^+_{\geq 0}$ such that $\xi_N=0$. Then $(M \otimes V^{\otimes n})^{\ft_0, \mu} = (V^{\xi} \otimes V^{\otimes n})^{\ft_0, \mu}$. So it is enough to consider highest weight module $V^{\xi}$ with highest weight $\xi \in P^+_{\geq 0}$ such that $\xi_N=0$, which is associated to partitions $\xi$ of length at most $N-1$.
	\end{rmk}
	
	\subsection{Computation of the $(\ft_0,\mu)$ invariant space}~\\
	\begin{prop}
		The $(\ft_0, \mu)$ invariant space $F_{n,p,\mu}(V^{\xi})=(V^{\xi} \otimes V^{\otimes n})^{\ft_0,\mu}$, for $\mu \in \C$ and $\xi \in P^+_{\geq 0}$.
		\begin{align*}
			(V^{\xi} \otimes V^{\otimes n})^{\ft_0, \mu} \cong &  Hom_{\ft_0}({\1}_{\mu \chi }, Res_{\ft_0}^{\gl_N}V^{\xi} \otimes V^{\otimes n})\\
			\cong & Hom_{\ft}({\1}_{\theta}, Res_{\ft}^{\gl_N}V^{\xi} \otimes V^{\otimes n}),
		\end{align*}
		
		where ${\1}_{\theta}$ is a one-dimensional $\ft$-module and
		$${\1}_{\theta}= (\mu q + \frac{|\xi|+n}{N})tr_{\gl_p} + (-\mu p + \frac{|\xi| + n}{N}) tr_{\gl_q}.
		$$
		
	\end{prop}
	\begin{proof}
		The $(\ft_0, \mu)$ invariant space $F_{n,p,\mu}(V^{\xi})=(V^{\xi} \otimes V^{\otimes n})^{\ft_0, \mu}$ is defined to be the subspace
		$$
		\{v \in V^{\xi}  \otimes V^{\otimes n} | A v = \mu \chi(A) v\text{ for any A }\in \ft_0\}.
		$$
		To compute this subspace, we lift it to a $\ft$ invariant space. Let $\1_{\psi}$ the one-dimensional $\ft$-module such that
		\begin{align*}
			&(V^{\xi} \otimes V^{\otimes n})^{\ft_0, \mu}\\
			= &(Res_{\ft}^{\gl_N}(V^{\xi}  \otimes V^{\otimes n}) \otimes \1_{\psi})^{\ft}.
		\end{align*}
		Let $\ft = \ft_0 \oplus \mathbb{C}\{I_N\}$. For any $P \in \ft$, there is a unique decomposition $P = A + B$ such that $A \in \ft_0$ and $B = bI_N$ for some $b \in \C$. So the $\ft$-invariant corresponds to $$\{v \in V^{\xi}  \otimes V^{\otimes n} | Pv + \1_{\psi}(P)v = 0 \}$$. Then $Pv + \1_{\psi}(P)v = Av + Bv  + \1_{\psi}(P)v = 0$. And $B = bI_N$ acts by the scalar $$b(|\xi| + n) = (|\xi| + n) \dfrac{tr(B)}{N}$$. Also, we have $\chi(P) = \chi(A) + \chi(B) = \chi(A)$, since $\chi(B) = qbp-pbq = 0$. So 
		\begin{align*}
			&\{v \in V^{\xi}  \otimes V^{\otimes n} | Pv + \1_{\psi}(P)v = 0 \} \\
			= &\{v \in V^{\xi} \otimes V^{\otimes n} | Av = \mu \chi(A)v\}.
		\end{align*}
		
		For any $P \in \ft$ with
		$$
		P = 
		\begin{bmatrix}
			S &0\\
			0 &T
		\end{bmatrix}
		$$
		where $S \in \gl_p$ and $T \in \gl_q$, we have
		\begin{align*}
			\1_{\psi}(P) &=-\mu \chi (A) - \dfrac{|\xi|+n}{N} tr(B)\\
			&=-\mu \chi (P) - \dfrac{|\xi|+n}{N} tr(P)\\
			&= (-\mu q - \frac{|\xi| + n}{N})tr_{\gl_p}(S)+ (\mu p - \frac{|\xi| + n}{N})tr_{\gl_q}(T).\\
		\end{align*}
		Hence it follows that the one dimensional $\ft$-module $$\1_{\theta} = (\mu q + \frac{|\xi|+n}{N})tr_{\gl_p} + (-\mu p + \frac{|\xi| + n}{N}) tr_{\gl_q}.$$
	\end{proof}
\begin{rmk}
	The $(\ft, \1_{\theta})$ invariant space above is equivalent to the following $K$ invariant space.
\begin{align*}
	(V^{\xi} \otimes V^{\otimes n})^{\ft_0, \mu} \cong &  Hom_{\ft_0}({\1}_{\mu \chi }, Res_{\ft_0}^{\gl_N}V^{\xi} \otimes V^{\otimes n})\\
	\cong & Hom_{\ft}({\1}_{\theta}, Res_{\ft}^{\gl_N}V^{\xi} \otimes V^{\otimes n})\\
	\cong & Hom_K(det^a \boxtimes det^b, V^{\xi} \boxtimes V^{\otimes n}),
\end{align*}
where $a=\mu q +\frac{|\xi|+n}{N}$	and $b=-\mu p + \frac{|\xi|+n}{N}$.
\end{rmk}
	\subsection{A basis of invariant space and standard tableaux}~\\
	The characters of irreducible polynomial $GL_N$-modules are Schur functions. So we could consider the restriction of $V^{\xi} \otimes V^{\otimes n}$ by exploring Schur functions. Recall the following fact of Schur functions.
	\begin{prop}
		Let $s_{\nu}(x_1, \cdots, x_p, z_{p+1}, \cdots, z_N)$ be the character of $V^{\nu}$, then
		$$
		s_{\nu} (x_1, \cdots, x_p, z_{p+1}, \cdots, z_N) = \Sigma c^{\nu}_{\omega_1, \omega_2} s_{\omega_1}(x_1, \cdots, x_p) s_{\omega_2}(z_{p+1}, \cdots, z_N),
		$$
		where $\omega_1$ is a highest weight of $GL_p$ and $\omega_2$ is a highest weight of $GL_q$, $c^{\nu}_{\omega_1,\omega_2}$ is the Littlewood-Richardson coefficient.
	\end{prop}
The Littelwood-Richardson coefficient $c^{\nu}_{\omega_1,\omega_2}$ is the multiplicity of the $K$-module $V^{\omega_1} \boxtimes V^{\omega_2}$ in the restriction of $GL_N$-module $V^{\nu}$. 
Let $V^{\xi} \otimes V^{\otimes n}=\bigoplus_{\nu} m_{\nu} V^{\nu}$ as $GL_N$-modules, where $\nu \in P^+_{\geq 0}$ and $m_{\nu} \in \N$ is the multiplicity of $V^{\nu}$ in $V^{\xi}$. Then the $(\ft_0,\mu)$ invariant space
\begin{align}
	F_{n,p,\mu}(V^{\xi}) &= Hom_K(det^a \boxtimes det^b, Res_K^{GL_N} V^{\xi} \otimes V^{\otimes n})\\
	&= \bigoplus_{\nu} m_{\nu}Hom_K(det^a \boxtimes det^b, Res_K^{GL_N} V^{\nu}).
\end{align}
Since $\nu \in P^+_{\geq 0}$, to guarantee $Hom_K(det^a \boxtimes det^b, Res_K^{GL_N} V^{\nu})  \neq 0$ for each $\nu$ in (13), it suffices to consider $a,b \in \N$, otherwise $F_{n,p,\mu}(V^{\xi}) =(V^{\xi} \otimes V^{\otimes n})^{\ft_0, \mu} =0$. Our goal is to compute the $\nu$ such that the multiplicity of $det^a \boxtimes det^b$ in the $K$ restriction of the $GL_N$-module $V^{\nu}$ is nonzero. To do this, we need Okada's theorem \cite{O}.
	\begin{thm}
		For any two rectangular shapes $(a^p)$ and $(b^q)$, where $a$ and $b$ are nonnegative integers and $p \leq q$, then $$
		s_{a^p} \cdot s_{b^q} = \sum c^{\nu}_{(a^p) (b^q)}s_{\nu},
		$$
		where $c^{\nu}_{(a^p) (b^q)} = 1$ when $\nu$ satisfies the condition 
		\begin{align}
			&\nu_i + \nu_{p+q-i +1} = a+b, \quad i= 1, \cdots, p\\
			&\nu_{p} \geq max(a,b)\\
			&\nu_i = b, \quad i= p+1, \cdots, q
		\end{align}
		and $c^{\nu}_{(a^p) (b^q)} = 0$ otherwise.
	\end{thm}
\begin{cor}
   Now we have the following fact, the $(\ft_0,\mu)$ invariant space
    	\begin{align}
    	F_{n,p,\mu}(V^{\xi})&=(V^{\xi} \otimes V^{\otimes n})^{\ft_0, \mu}\\
    	&=\bigoplus_{\nu} Hom_{GL_N}(V^{\nu}, V^{\xi} \otimes V^{\otimes n}),
    \end{align}
where $\nu \in P^+_{\geq 0}$ runs through all partitions satisfying (16)-(18).
\end{cor} 
Moreover, by Pieri's rule, the vector space $Hom_{GL_N}(V^{\nu}, V^{\xi} \otimes V^{\otimes n})$ has a basis indexed by standard tableaux $T$ such that the shape of $T$ is $\nu / \xi$ and the dimension of this vector space
$$m_{\nu}=dim Hom_{GL_N}(V^{\nu}, V^{\xi} \otimes V^{\otimes n})$$
equals the number of standard tableaux $T$ with the shape of $T$ being $\nu / \xi$. If $m_{\nu} \neq 0$, then $\xi \subset \nu$ and $|\nu|=|\xi|+n$.
\begin{thm}
 The $(\ft_0,\mu)$ invariant space $F_{n,p,\mu}(V^{\xi})=(V^{\xi} \otimes V^{\otimes n})^{\ft_0,\mu}$ has a one to one correspondence to the set of standard tableaux $T$ such that the shape of $T$ is $\nu / \xi $ for $\nu \in P^+_{\geq 0}$ with $|\nu|= |\xi|+n$, $\nu$ runs through all the partitions satisfying (16)-(18) and $\xi \subset \nu$.
\end{thm}
	
	
	
	Let us consider the following example of $(\ft_0,\mu)$ invariant space.
	\begin{eg}
		Let $M = V^{\xi}$ be a $GL_3$-module, $\xi=2\epsilon_1+\epsilon_2$, $n= 3$, $p=1$ and $\mu=0$.\\
		
		Then $(a^p) = (2^1)$ and $(b^q) = (2^2)$.\\

		By Okada's theorem, we could compute the shapes $\nu$ such that the invariant space is nonzero.\\
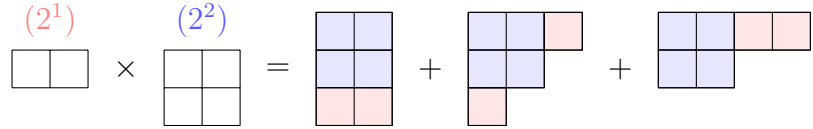
		\begin{center}
			\begin{tikzpicture}[scale=0.5][shift={(1,0)}]
				\begin{scope}[shift={(6,15)}]
					\draw (1,2.8) node[red!50] {$(2^1)$};
					\draw[step=1] (0,1) grid (2,2);
				\end{scope}
				
				\draw (9,16.5) node{$\times$};
				
				\begin{scope}[shift={(10,15)}]
					\draw (1,2.8) node[blue!65] {$(2^2)$};
					\draw[step=1] (0,0) grid (2,2);
				\end{scope}

				\draw (13,16.5) node{$=$};
				
				\begin{scope}[shift={(14,15)}]
					\draw[thin,fill=blue!10]  (0,1) rectangle (2,2);
					\draw[thin,fill=blue!10]  (0,2) rectangle (2,3);
					\draw[thin,fill=red!10]  (0,1) rectangle (2,0);
					\draw[step=1] (0,0) grid (2,3);
				\end{scope}
				
				\draw (17,16.5) node{$+$};
				
				\begin{scope}[shift={(18,15)}]
					\draw [thin,fill=blue!10] (0,2) rectangle (2,3);
					\draw [thin,fill=blue!10] (0,1) rectangle (2,2);
					\draw[thin,fill=red!10]  (0,1) rectangle (1,0);
					\draw[thin,fill=red!10]  (2,3) rectangle (3,2);
					\draw[step=1] (0,1) grid (2,3);
					\draw (2,2) rectangle (3,3);
					\draw (0,0) rectangle (1,1);
				\end{scope}
				
				\draw (22,16.5) node{$+$};
				
				\begin{scope}[shift={(23,15)}]
					\draw[thin,fill=blue!10]  (0,2) rectangle (2,3);
					\draw[thin,fill=blue!10]  (0,1) rectangle (2,2);
					\draw[thin,fill=red!10]  (2,3) rectangle (4,2);
					\draw[step=1] (0,1) grid (2,3);
					\draw[step=1] (2,2) grid (4,3);
				\end{scope}
			\end{tikzpicture}
		\end{center}
		Then a basis of the invariant space could be indexed by standard tableaux on skew shapes obtained by the shapes above skewed by $\xi$.\\
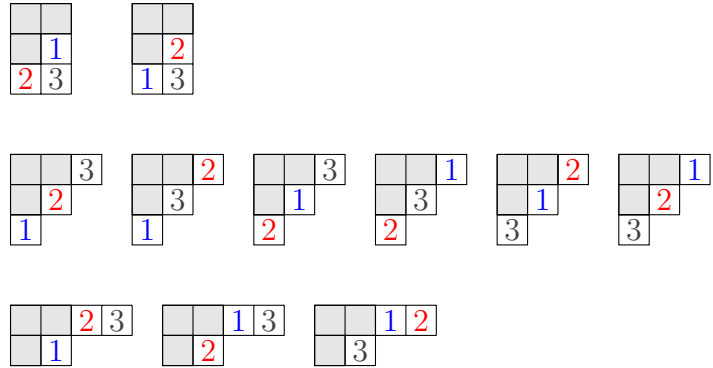
		\begin{center}
			\begin{tikzpicture}[scale=0.4]
				\begin{scope}[shift={(5,10)}]
					\draw [thin,fill=gray!20] (0,2) rectangle (2,3);
					\draw [thin,fill=gray!20] (0,1) rectangle (1,2);
					\draw[step=1] (0,0) grid (2,3);
					\draw (0.5, 0.5) node[red] {$2$};
					\draw (1.5, 1.5) node[blue] {$1$};
					\draw (1.5, 0.5) node[black!75] {$3$};
				\end{scope}
				
				\begin{scope}[shift={(9,10)}]
					\draw [thin,fill=gray!20] (0,2) rectangle (2,3);
					\draw [thin,fill=gray!20] (0,1) rectangle (1,2);
					\draw[step=1] (0,0) grid (2,3);
					\draw (0.5, 0.5) node[blue] {$1$};
					\draw (1.5, 1.5) node[red] {$2$};
					\draw (1.5, 0.5) node[black!75] {$3$};
				\end{scope}
				\begin{scope}[shift={(5,5)}]
					\draw [thin,fill=gray!20] (0,2) rectangle (2,3);
					\draw [thin,fill=gray!20] (0,1) rectangle (1,2);
					\draw[step=1] (0,1) grid (2,3);
					\draw (2,2) rectangle (3,3);
					\draw (0,0) rectangle (1,1);
					\draw (0.5, 0.5) node[blue] {$1$};
					\draw (1.5, 1.5) node[red] {$2$};
					\draw (2.5, 2.5) node[black!75] {$3$};
				\end{scope}
				
				\begin{scope}[shift={(9,5)}]
					\draw [thin,fill=gray!20] (0,2) rectangle (2,3);
					\draw [thin,fill=gray!20] (0,1) rectangle (1,2);
					\draw[step=1] (0,1) grid (2,3);
					\draw (2,2) rectangle (3,3);
					\draw (0,0) rectangle (1,1);
					\draw (0.5, 0.5) node[blue] {$1$};
					\draw (1.5, 1.5) node[black!75] {$3$};
					\draw (2.5, 2.5) node[red] {$2$};
				\end{scope}
				
				\begin{scope}[shift={(13,5)}]
					\draw [thin,fill=gray!20] (0,2) rectangle (2,3);
					\draw [thin,fill=gray!20] (0,1) rectangle (1,2);
					\draw[step=1] (0,1) grid (2,3);
					\draw (2,2) rectangle (3,3);
					\draw (0,0) rectangle (1,1);
					\draw (0.5, 0.5) node[red] {$2$};
					\draw (1.5, 1.5) node[blue] {$1$};
					\draw (2.5, 2.5) node[black!75] {$3$};
				\end{scope}
				
				\begin{scope}[shift={(17,5)}]
					\draw [thin,fill=gray!20] (0,2) rectangle (2,3);
					\draw [thin,fill=gray!20] (0,1) rectangle (1,2);
					\draw[step=1] (0,1) grid (2,3);
					\draw (2,2) rectangle (3,3);
					\draw (0,0) rectangle (1,1);
					\draw (0.5, 0.5) node[red] {$2$};
					\draw (1.5, 1.5) node[black!75] {$3$};
					\draw (2.5, 2.5) node[blue] {$1$};
				\end{scope}
				
				\begin{scope}[shift={(21,5)}]
					\draw [thin,fill=gray!20] (0,2) rectangle (2,3);
					\draw [thin,fill=gray!20] (0,1) rectangle (1,2);
					\draw[step=1] (0,1) grid (2,3);
					\draw (2,2) rectangle (3,3);
					\draw (0,0) rectangle (1,1);
					\draw (0.5, 0.5) node[black!75] {$3$};
					\draw (1.5, 1.5) node[blue] {$1$};
					\draw (2.5, 2.5) node[red] {$2$};
				\end{scope}
				
				\begin{scope}[shift={(25,5)}]
					\draw [thin,fill=gray!20] (0,2) rectangle (2,3);
					\draw [thin,fill=gray!20] (0,1) rectangle (1,2);
					\draw[step=1] (0,1) grid (2,3);
					\draw (2,2) rectangle (3,3);
					\draw (0,0) rectangle (1,1);
					\draw (0.5, 0.5) node[black!75] {$3$};
					\draw (1.5, 1.5) node[red] {$2$};
					\draw (2.5, 2.5) node[blue] {$1$};
				\end{scope}
				\begin{scope}[shift={(5,0)}]
					\draw [thin,fill=gray!20] (0,2) rectangle (2,3);
					\draw [thin,fill=gray!20] (0,1) rectangle (1,2);
					\draw[step=1] (0,1) grid (2,3);
					\draw[step=1] (2,2) grid (4,3);
					\draw (1.5, 1.5) node[blue] {$1$};
					\draw (2.5, 2.5) node[red] {$2$};
					\draw (3.5, 2.5) node[black!75] {$3$};
				\end{scope}
				\begin{scope}[shift={(10,0)}]
					\draw [thin,fill=gray!20] (0,2) rectangle (2,3);
					\draw [thin,fill=gray!20] (0,1) rectangle (1,2);
					\draw[step=1] (0,1) grid (2,3);
					\draw[step=1] (2,2) grid (4,3);
					\draw (1.5, 1.5) node[red] {$2$};
					\draw (2.5, 2.5) node[blue] {$1$};
					\draw (3.5, 2.5) node[black!75] {$3$};
				\end{scope}
				\begin{scope}[shift={(15,0)}]
					\draw [thin,fill=gray!20] (0,2) rectangle (2,3);
					\draw [thin,fill=gray!20] (0,1) rectangle (1,2);
					\draw[step=1] (0,1) grid (2,3);
					\draw[step=1] (2,2) grid (4,3);
					\draw (1.5, 1.5) node[black!75] {$3$};
					\draw (2.5, 2.5) node[blue] {$1$};
					\draw (3.5, 2.5) node[red] {$2$};
				\end{scope}
				
			\end{tikzpicture}
		\end{center}
		
		In this example, we obtain an invariant space of $11$ dimensions.
		
	\end{eg}
	\subsection{One skew shape}
	In this subsection, we associate a skew shape $\varphi_{n,p,\mu}^{\xi}$ to the image $F_{n,p,\mu}(V^{\xi})$ under Etingof-Freund-Ma functor. Let $\xi=\sum_{i=1}^N \xi_i \epsilon_i \in P^+_{\geq 0}$. The corresponding Young diagram $\xi = (\xi_1, \cdots, \xi_N)$. The first $q$ rows of $\xi$ forms a Young diagram denoted by $\xi^{(1)}$ and the last $p$ rows of $\xi$ forms a Young diagram denoted by $\xi^{(2)}$. The parameter $\mu$ gives a pair of rectangles $(a^p)$ and $(b^q)$ denoting the $K$-module $det^a \boxtimes det^b$, where
	$a=\mu q+\frac{|\xi|+n}{N}$
	and $b=-\mu p + \frac{|\xi|+n}{N}$.\\
	Suppose $p \leq q$. Placing the northwestern corner the rectangle $(a^p)$ next to the northeastern corner of the rectangle $(b^q)$ forms a Young diagram $\beta$. Delete the Young diagram $\xi^{(1)}$ from northwestern corner of $\beta$. Let \rotatebox[origin=c]{180}{$\xi^{(2)}$} denote the skew shape obtained by rotating $\xi^{(2)}$ by $\pi$. Delete the rotated $\xi^{(2)}$ from the southeastern corner of $\beta$, i.e. the skew shape $\varphi^{\xi}_{n,p,\mu}$ is defined by $\varphi^{\xi}_{n,p,\mu} = \nu / \xi^{(1)}$, where $\nu_i= a+b -\xi_{N-i+1}$ for $i = 1,\cdots, p$ and $\nu_i = b$ for $i=p+1, \cdots, q$.\\ 
	
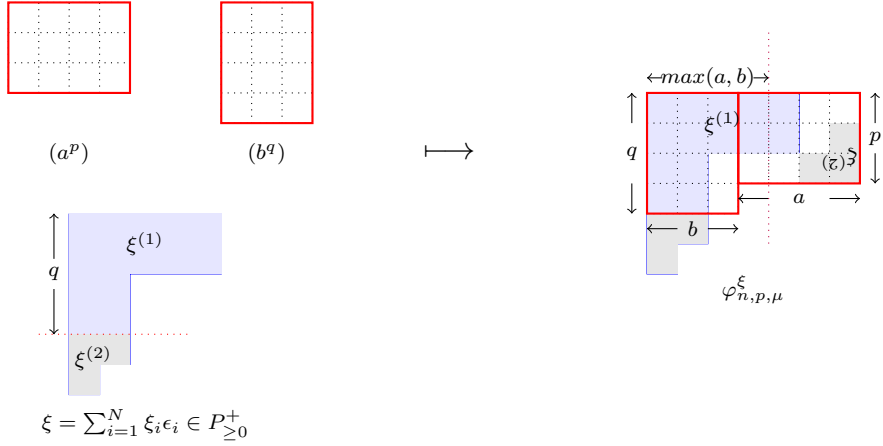
	\begin{center}
		\begin{tikzpicture}
			\begin{scope}[scale=0.4, shift={(-10,0)}]
				\draw[step=1, dotted] (-2,10) grid (2,7);
				\draw[step=1, dotted] (5,10) grid  (8,6);
				\draw [red,thick] (-2,10) rectangle (2,7);
				\draw [red,thick] (5,10) rectangle (8,6);
				\draw [scale=0.5](0, 10) node {\tiny{$(a^p)$}};
				\draw [scale=0.5](13,10) node {\tiny{$(b^q)$}};
				\draw [blue] (0,3)--(0,-3)--(1,-3)--(1,-2)--(2,-2)--(2,1)--(5,1)--(5,3)--(0,3);
				\draw [draw=none,fill=blue!10] (0,3) rectangle (2,-1);
				\draw [draw=none,fill=gray!20] (0,-1) rectangle (1,-3);
				\draw [draw=none, fill=gray!20] (1,-1) rectangle (2,-2);
				\draw [draw=none,fill=blue!10] (2,3) rectangle (5,1);
				\draw [dotted, red](-1,-1)--(4,-1);
				\draw (2.5,2) node {\tiny{$\xi^{(1)}$}};
				\draw (0.8,-1.8) node {\tiny{$\xi^{(2)}$}};
				\draw [scale=0.5] (5,-8) node {\tiny{$\xi = \sum_{i=1}^N \xi_i \epsilon_i \in P^+_{\geq 0}$}};
				\draw [->] (-0.5,0.5)--(-0.5,-1);
				\draw [->] (-0.5,1.5)--(-0.5,3);
				\draw (-0.5, 1) node {\tiny{$q$}};
			\end{scope}
			\begin{scope}[shift={(2,0)}]
				\draw (-1,2) node {$\longmapsto$};
			\end{scope}
			\begin{scope}[scale=0.4, shift={(9,4)}]
				\draw [blue] (0,3)--(0,-3)--(1,-3)--(1,-2)--(2,-2)--(2,1)--(5,1)--(5,3)--(0,3);
				\draw [draw=none,fill=blue!10] (0,3) rectangle (2,-1);
				\draw [draw=none,fill=gray!20] (0,-1) rectangle (1,-3);
				\draw [draw=none, fill=gray!20] (1,-1) rectangle (2,-2);
				\draw [draw=none,fill=blue!10] (2,3) rectangle (5,1);
				\draw [draw=none,fill=gray!20] (5,1) rectangle (7,0);
				\draw [draw=none,fill=gray!20] (6,2) rectangle (7,1);
				\draw[dotted] (0,3) grid (3, -1);
				\draw[dotted] (3,3) grid (7,0);
				\draw [red,thick] (0,3) rectangle (3,-1);
				\draw [red,thick] (3,3) rectangle (7,0);
				\draw (2.5,2) node {\tiny{$\xi^{(1)}$}};
				\draw (6.3,0.8) node {\tiny{\rotatebox{180}{$\xi^{(2)}$}}};
				\draw [->] (1,-1.5)--(0,-1.5);
				\draw [->] (2,-1.5)--(3,-1.5);
				\draw [->] (4,-0.5)--(3,-0.5);
				\draw [->] (6,-0.5)--(7,-0.5);
				\draw [->] (0.5,3.5)--(0,3.5);
				\draw [->] (3.5,3.5)--(4,3.5);
				\draw (2,3.5) node {\tiny{$max(a,b)$}};
				\draw (1.5,-1.5) node {\tiny{$b$}};
				\draw (5,-0.5) node {\tiny{$a$}};
				\draw [->] (-0.5,2)--(-0.5,3);
				\draw [->] (-0.5,0)--(-0.5,-1);
				\draw (-0.5,1) node {\tiny{$q$}};
				\draw [->] (7.5,2)--(7.5,3);
				\draw (7.5, 1.5) node {\tiny{$p$}};
				\draw [->] (7.5,1)--(7.5,0);
				\draw [dotted, purple] (4,5)--(4,-2);
				\draw (3.5, -3.5) node {\tiny{$\varphi_{n,p,\mu}^{\xi}$}};
			\end{scope}
		\end{tikzpicture}
	\end{center}
	Let $\varphi=\varphi_{n,p,\mu}^{\xi}$. If a cell $(i,j)$ of the skew shape $\varphi$ satisfy $(i+1,j) \notin \varphi$ and $(i,j+1) \notin \varphi$, then call $(i,j)$ a corner of $\varphi$. Define $\gamma$-move on a skew shape $\varphi$: delete a corner $(i,j) \in \varphi$ such that $j> max(a,b)$ and $1 \leq i \leq p$, and add the cell $(p+q-i+1,a+b-j+1)$. Denote the $\gamma$-move by $\varphi \to \varphi'$ where $\varphi'= \varphi \setminus (i,j) \cup (p+q-i+1,a+b-j+1)$. Note that for a given $\varphi$, the $\gamma$-move stops when there is no cell $(i,j)$ such that $j > max(a,b)$. Given the skew shape $\varphi^{\xi}_{n,p,\mu}$, a collection $D(\varphi^{\xi}_{n,p,\mu})$ of skew shapes consists of $\varphi^{\xi}_{n,p,\mu}$ and all the skew shapes obtained by applying $\gamma$-moves on $\varphi^{\xi}_{n,p,\mu}$ for finitely many times. The shape $\varphi^{\xi}_{n,p,\mu}$ is called the minimal shape of the representation $F_{n,p,\mu}(V^{\xi})$.
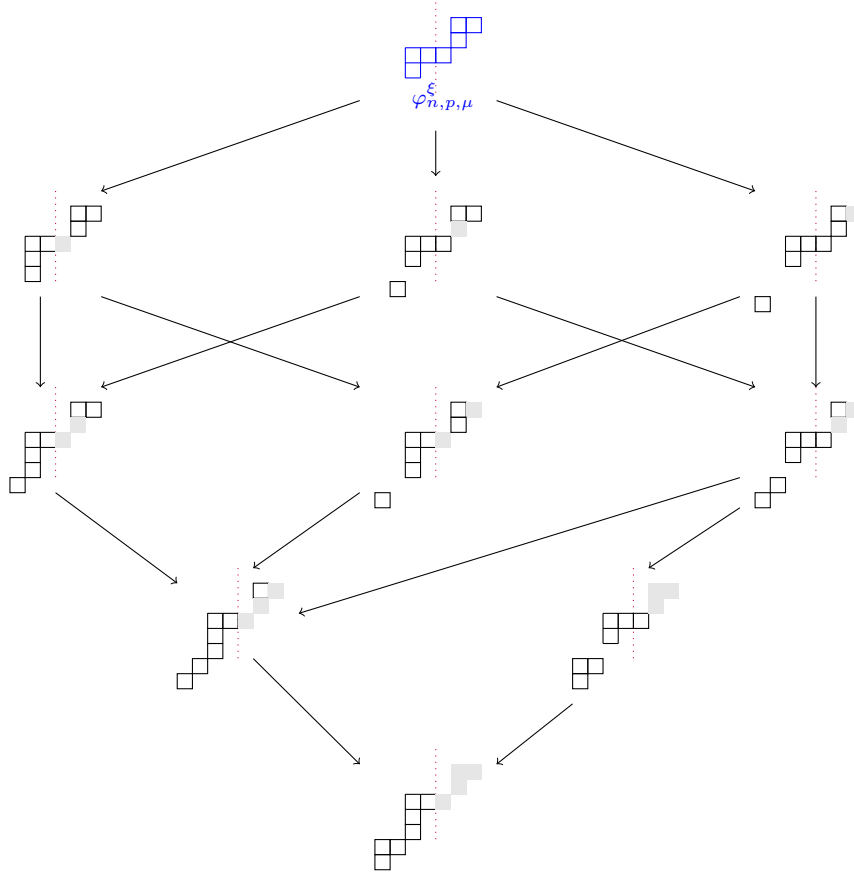
	\begin{center}
		\begin{tikzpicture}
			\begin{scope}[blue, scale=0.2,shift={(0,17.5)}]
				\draw [dotted, purple] (2,3)--(2,-3);
				\draw (0,0) grid (1,-2);
				\draw (1,0) grid (2,-1);
				\draw (2,0) rectangle (3,-1);
				\draw (3,2) rectangle (4,1);
				\draw (4,2) rectangle (5,1);
				\draw (3,1) rectangle (4,0);
				\draw (2.5,-3.2) node {\tiny{$\varphi^{\xi}_{n,p,\mu}$}};
			\end{scope}
			\begin{scope}[scale=0.2,shift={(-25,5)}]
				\draw [dotted, purple] (2,3)--(2,-3);
				\draw (0,0) grid (1,-2);
				\draw (1,0) grid (2,-1);
				\draw [draw=none, fill=gray!20] (2,0) rectangle (3,-1);
				\draw (0,-2) rectangle (1,-3);
				\draw (3,2) rectangle (4,1);
				\draw (4,2) rectangle (5,1);
				\draw (3,1) rectangle (4,0);
			\end{scope}
			\begin{scope}[scale=0.2,shift={(0,5)}]
				\draw [dotted, purple] (2,3)--(2,-3);
				\draw (0,0) grid (1,-2);
				\draw (1,0) grid (2,-1);
				\draw (2,0) rectangle (3,-1);
				\draw (3,2) rectangle (4,1);
				\draw (4,2) rectangle (5,1);
				\draw [draw=none, fill=gray!20] (3,1) rectangle (4,0);
				\draw (-1,-3) rectangle (0,-4);
			\end{scope}
			\begin{scope}[scale=0.2,shift={(25,5)}]
				\draw [dotted, purple] (2,3)--(2,-3);
				\draw (0,0) grid (1,-2);
				\draw (1,0) grid (2,-1);
				\draw (2,0) rectangle (3,-1);
				\draw (3,2) rectangle (4,1);
				\draw [draw=none, fill=gray!20] (4,2) rectangle (5,1);
				\draw (-2,-4) rectangle (-1,-5);
				\draw (3,1) rectangle (4,0);
			\end{scope}
			
			\begin{scope}[scale=0.2,shift={(-25,-8)}]
				\draw [dotted, purple] (2,3)--(2,-3);
				\draw (0,0) grid (1,-2);
				\draw (1,0) grid (2,-1);
				\draw [draw=none, fill=gray!20](2,0) rectangle (3,-1);
				\draw (0,-2) rectangle (1,-3);
				\draw (3,2) rectangle (4,1);
				\draw (4,2) rectangle (5,1);
				\draw [draw=none, fill=gray!20] (3,1) rectangle (4,0);
				\draw (-1,-3) rectangle (0,-4);
			\end{scope}
			\begin{scope}[scale=0.2,shift={(25,-8)}]
				\draw [dotted, purple] (2,3)--(2,-3);
				\draw (0,0) grid (1,-2);
				\draw (1,0) grid (2,-1);
				\draw (2,0) rectangle (3,-1);
				\draw (3,2) rectangle (4,1);
				\draw [draw=none, fill=gray!20](4,2) rectangle (5,1);
				\draw (-2,-4) rectangle (-1,-5);
				\draw [draw=none, fill=gray!20] (3,1) rectangle (4,0);
				\draw (-1,-3) rectangle (0,-4);
			\end{scope}
			\begin{scope}[scale=0.2,shift={(0,-8)}]
				\draw [dotted, purple] (2,3)--(2,-3);
				\draw (0,0) grid (1,-2);
				\draw (1,0) grid (2,-1);
				\draw [draw=none, fill=gray!20] (2,0) rectangle (3,-1);
				\draw (0,-2) rectangle (1,-3);
				\draw (3,2) rectangle (4,1);
				\draw [draw=none, fill=gray!20] (4,2) rectangle (5,1);
				\draw (-2,-4) rectangle (-1,-5);
				\draw (3,1) rectangle (4,0);
			\end{scope}
			
			\begin{scope}[scale=0.2,shift={(-13,-20)}]
				\draw [dotted, purple] (2,3)--(2,-3);
				\draw (0,0) grid (1,-2);
				\draw (1,0) grid (2,-1);
				\draw [draw=none, fill=gray!20] (2,0) rectangle (3,-1);
				\draw (0,-2) rectangle (1,-3);
				\draw (3,2) rectangle (4,1);
				\draw [draw=none, fill=gray!20] (4,2) rectangle (5,1);
				\draw (-2,-4) rectangle (-1,-5);
				\draw [draw=none, fill=gray!20] (3,1) rectangle (4,0);
				\draw (-1,-3) rectangle (0,-4);
			\end{scope}
			
			\begin{scope}[scale=0.2,shift={(13,-20)}]
				\draw [dotted, purple] (2,3)--(2,-3);
				\draw (0,0) grid (1,-2);
				\draw (1,0) grid (2,-1);
				\draw (2,0) rectangle (3,-1);
				\draw [draw=none, fill=gray!20] (3,2) rectangle (4,1);
				\draw (-2,-3) rectangle (-1,-4);
				\draw [draw=none, fill=gray!20] (4,2) rectangle (5,1);
				\draw (-2,-4) rectangle (-1,-5);
				\draw [draw=none, fill=gray!20] (3,1) rectangle (4,0);
				\draw (-1,-3) rectangle (0,-4);
			\end{scope}
			\begin{scope}[scale=0.2,shift={(0,-32)}]
				\draw [dotted, purple] (2,3)--(2,-3);
				\draw (0,0) grid (1,-2);
				\draw (1,0) grid (2,-1);
				\draw [draw=none, fill=gray!20] (2,0) rectangle (3,-1);
				\draw (0,-2) rectangle (1,-3);
				\draw  [draw=none, fill=gray!20] (3,2) rectangle (4,1);
				\draw (-2,-3) rectangle (-1,-4);
				\draw [draw=none, fill=gray!20] (4,2) rectangle (5,1);
				\draw (-2,-4) rectangle (-1,-5);
				\draw  [draw=none, fill=gray!20] (3,1) rectangle (4,0);
				\draw (-1,-3) rectangle (0,-4);
			\end{scope}
			\draw [scale=0.2] [->] (-3,14)--(-20,8); 
			\draw [scale=0.2] [->] (2,12)--(2,9); 
			\draw [scale=0.2] [->] (6,14)--(23,8); 
			\draw [scale=0.2] [->] (-24,1)--(-24,-5); 
			\draw [scale=0.2] [->] (-20,1)--(-3,-5); 
			\draw [scale=0.2] [->] (6,1)--(23,-5); 
			\draw [scale=0.2] [->] (-3,1)--(-20,-5); 
			\draw [scale=0.2] [->] (27,1)--(27,-5);
			\draw [scale=0.2] [->] (22,1)--(6,-5); 
			\draw [scale=0.2] [->] (-23,-12)--(-15,-18); 
			\draw [scale=0.2] [->] (-3,-12)--(-10,-17); 
			\draw [scale=0.2] [->] (22,-11)--(-7,-20); 
			\draw [scale=0.2] [->] (22,-13)--(16,-17); 
			\draw [scale=0.2] [->] (-10,-23)--(-3,-30);
			\draw [scale=0.2] [->] (11,-26)--(6,-30); 
		\end{tikzpicture}
	\end{center}
	%
	%
	Continue Example 5.9, the representation  $F_{3,1,0}(V^{\Yboxdim{4pt}\young(\quad \quad,\quad)})$ is index by the following skew shape $\varphi$.
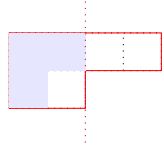
	\begin{center}
		\begin{tikzpicture}[scale=0.5]
			\draw [dotted] (0,2) grid (2,0);
			\draw [dotted] (2,2) grid (4,1);
			\draw [red] (0,2) rectangle (2,0);
			\draw [red] (2,2) rectangle (4,1);
			\draw [draw=none, fill=blue!10] (0,2) rectangle (2,1);
			\draw [draw=none, fill=blue!10] (0,1) rectangle (1,0);
			\draw [dotted, purple] (2,3)--(2,-1);
		\end{tikzpicture}
	\end{center}
	The collection $D(\varphi)$ of skew shapes is obtained as follows:\\
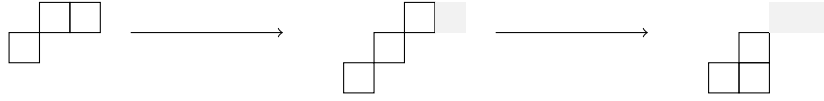
	\begin{center}
		\begin{tikzpicture}[scale=0.4]
			\begin{scope}[shift={(-12,0)}]
				\draw (1,1) rectangle (2,0);
				\draw (2,2) rectangle (3,1);
				\draw (3,2) rectangle (4,1);
			\end{scope}
			\begin{scope}[shift={(0,0)}]
				\draw (1,1) rectangle (2,0);
				\draw (2,2) rectangle (3,1);
				\draw [draw=none, fill=gray!10](3,2) rectangle (4,1);
				\draw (0,0) rectangle (1,-1);
			\end{scope}
			\begin{scope}[shift={(12,0)}]
				\draw (1,1) rectangle (2,0);
				\draw [draw=none, fill=gray!10](2,2) rectangle (3,1);
				\draw (2,0) rectangle (1,-1);
				\draw [draw=none, fill=gray!10](3,2) rectangle (4,1);
				\draw (1,0) rectangle (0,-1);
			\end{scope}
			\draw [->] (-7,1)--(-2,1);
			\draw [->] (5,1)--(10,1);
		\end{tikzpicture}
	\end{center}
	\subsection{Skew shapes and standard tableaux}
	For the ease of description, let us use the following definition of skew shapes and standard tableaux. Given a partition $\xi=(\xi_1, \cdots, \xi_l)$, the corresponding Young diagram $\xi$ is a subset of $\Z^2$, consisting of $(i,j)$ such that $1 \leq i \leq l$ and $1 \leq j \leq \xi_i$. Let $\nu=(\nu_1, \cdots, \nu_l)$ and $\xi=(\xi_1, \cdots, \xi_l)$ such that $\nu_i \geq \xi_i$ for $1 \leq i \leq l$, then for the corresponding Young diagrams $\xi \subset \nu$ holds. A skew shape $\nu / \xi$ is the subset $\nu \setminus \xi$ of $\Z^2$. For example, let $\nu=(7,6,5,3,2,1)$ and $\xi =(5,5,2,2,2,1)$, then Young diagrams $\nu$ and $\xi$ and the skew shape $\nu / \xi$ are the following subsets of $\Z^2$.
	$$\nu=\{(i,j)|1 \leq i \leq 6, 1 \leq j \leq \nu_i\},$$
	$$\xi=\{(i,j)|1 \leq i \leq 6, 1 \leq j \leq \xi_i\}$$
	and
	$$\nu /\xi = \{(1,6),(1,7),(2,6),(3,3),(3,4),(3,5),(4,3)\}.$$
	
	Define a tableau $T$ on $n$-indices $\{1, \cdots, n\}$ to be an injective map $T$
	\begin{align*}
		T: \{1, \cdots, n\} &\to \Z^2\\
		k &\mapsto (\ci(k), \cj(k))
	\end{align*}
	
	where $\ci$ and $\cj$ being two maps from $\{1, \cdots,n\}$ to $\Z$ and the image $Im(T)$ of $T$ being a skew shape. The image $Im(T)$ is also called the shape of the tableaux $T$. Let $cont_T$ be a map
	\begin{align*}
		cont_T: \{1, \cdots, n\} &\to \Z\\
		& k \mapsto \cj(k)-\ci(k),
	\end{align*}
	call $cont_T(k)$ is the content of $k$ in the tableau $T$. If $T^{-1}(i + 1, j) > T^{-1} (i, j)$ and $T^{-1}(i,j + 1) > T^{-1}(i, j)$ hold for each cell $(i, j) \in Im(T)$, then call $T$ is a standard tableau.\\
	Let $$\tab = \{T| T \text{ is a standard tableau and } Im(T) \in D(\varphi^{\xi}_{n,p,\mu})\}.$$
	The invariant space $F_{n,p,\mu}(V^{\xi})=(V^{\xi} \otimes V^{\otimes n})^{\ft_0, \mu}$ has a basis indexed by a collection of standard tableaux on the skew shapes in $D(\varphi^{\xi}_{n,p,\mu})$, i.e. all the tableaux in $\tab$. Let $v_T$ denote the basis vector indexed by $T \in \tab$. Then as a vector space
	\begin{align*}
		F_{n,p,\mu}(V^{\xi})&=(V^{\xi} \otimes V^{\otimes n})^{\ft_0, \mu}\\
		&= span_{\C}\{v_T|T \in \tab\}.
	\end{align*}
	
	\section{$\Y$- semisimplicity}
	\subsection{Action of $\Y$}
	In this subsection let us computer the $\Y$-actions on the invariant space $F_{n,p,\mu}(V^{\xi})=(V^{\xi} \otimes V^{\otimes n})^{\ft_0, \mu}$. In \cite{J}, Jordan computed the action of $y_1$ and used the fact that Etingof-Freund-Ma functor is a trigonometric degeneration of the quantum case. Now let us review the computation and conduct it in the degenerate case.
	Let us use the following notations in \cite{EFM} for sums
	\begin{align}
		&\sum_{s,t}= \sum_{s=1}^N \sum_{p = 1}^N\\
		&\sum_{s|t} = \sum_{s = 1}^p \sum_{t = p + 1}^N + \sum_{t= 1}^p \sum_{s = p + 1}^N\\
		&\sum_{st} = \sum_{s= 1}^p \sum_{t=1}^p + \sum_{s = p +1}^N \sum_{t = p+1}^N
	\end{align}
	It is easy to observe that the sum of $(22)$ and $(23)$ equals $(21)$.\\
	Review the definition of $y_1$ on the $(\ft_0,\mu)$-invariant space $F_{n,p,\mu}(V^{\xi})=(V^{\xi} \otimes V^{\otimes n})^{\ft_0,\mu}$ in \cite{EFM},
	\begin{align*}
		y_1 = -\sum_{s|t} (E_s^t)_0 \otimes (E_t^s)_1 + \dfrac{p-q-\mu N}{2}\gamma_1 + \dfrac{1}{2}\sum_{l>1}s_{1,l} + \dfrac{1}{2}\sum_{l \neq 1}s_{1,l} \gamma_1 \gamma_l.
	\end{align*}
	Compute the last two terms of $y_1$, we have 
	\begin{align*}
		&\dfrac{1}{2}\sum_{l> 1}s_{1,l} + \dfrac{1}{2}\sum_{l \neq 1}s_{1,l} \gamma_1 \gamma_l\\
		=& \dfrac{1}{2}\sum_{l>1} \sum_{s,t}(E_s^t)_1 \otimes (E_t^s)_l + \dfrac{1}{2}\sum_{l > 1}\sum_{s,t}(E_s^t J)_1 \otimes (E_t^s J)_l\\
		=& \sum_{l>1}\sum_{st}(E_s^t)_1 \otimes (E_t^s)_l\\
		=& \sum_{st}(E_s^t)_1 (\sum_{l>1} 1 \otimes (E_t^s)_l)\\
		=& \sum_{st} (E_s^t)_1 (\Delta^{(n)}(E_t^s) - (E_t^s)_0 - (E_t^s)_1)\\
	\end{align*}
	The last step follows the fact that $\sum_{l>1} 1 \otimes (E_t^s)_l= \Delta^{(n)}(E_t^s)-(E_t^s)_0-(E_t^s)_1$, where $\Delta$ denotes the comultiplication of Lie algebra $\gl_N$ and $\Delta^{(n)}(E_t^s)=\sum_{l=0}^n (E_t^s)_l$.\\\\
	Applying the fact that $y_1$ preserves on the $(\ft_0, \mu)$-invariant space $F_{n,p,\mu}(V^{\xi})=(V^{\xi} \otimes V^{\otimes n})^{\ft_0,\mu}$, the computation of the last two terms of $y_1$ above continues as follows.
	\begin{align*}
		& \sum_{st} (E_s^t)_1 (\Delta^{(n)}(E_t^s) - (E_t^s)_0 - (E_t^s)_1)\\
		=& \sum_{s= 1}^p (\mu q + \dfrac{|\xi| +n}{N})(E_s^s)_1 + \sum_{s = p +1}^N (-\mu p + \dfrac{|\xi| + n}{N})(E_s^s)_1 \\
		& - \sum_{s =1}^p p (E_s^s)_1 - \sum_{s = p+1}^N q(E_s^s)_1 - \sum_{st}(E_s^t)_1 \otimes (E_t^s)_0\\
		=& (\mu q -p +\dfrac{|\xi| + n}{N}) \sum_{s =1}^p(E_s^s)_1+ (-\mu p -q + \dfrac{|\xi|+n}{N}) \sum_{s=p+1}^N (E_s^s)_1 \\
		&- \sum_{st}(E_t^s)_0 \otimes (E_s^t)_1
	\end{align*}
	Combining other terms in the definition of $y_1$,
	\begin{align*}
		y_1 = &-\sum_{s,t}(E_s^t)_0 \otimes (E_t^s)_1 + \dfrac{p-q-\mu N}{2} \gamma_1\\
		&+ (\mu q -p +\dfrac{|\xi| + n}{N}) \sum_{s =1}^p(E_s^s)_1+ (-\mu p -q + \dfrac{|\xi|+n}{N}) \sum_{s=p+1}^N (E_s^s)_1 \\
		=& -\sum_{s,t}(E_s^t)_0 \otimes (E_t^s)_1 + (\mu q -p + \dfrac{|\xi|+n}{N} + \dfrac{p-q-\mu N}{2})\sum_{s=1}^p (E_s^s)_1\\
		&+(-\mu p - q + \dfrac{|\xi| +n}{N} - \dfrac{p-q-\mu N}{2}) \sum_{s=p+1}^N (E_s^s)_1\\
		=& -\sum_{s,t}(E_s^t)_0 \otimes (E_t^s)_1 + (\dfrac{|\xi|+n}{N} + \dfrac{\mu q - \mu p}{2} - \dfrac{N}{2})\sum_{s =1}^N (E_s^s)_1\\
		=& -\sum_{s,t}(E_s^t)_0 \otimes (E_t^s)_1 + \dfrac{|\xi|+n}{N} + \dfrac{\mu q - \mu p}{2} - \dfrac{N}{2},
	\end{align*}
\begin{rmk}
Since the action in \cite{JM} was define on $F_{n,p,\mu}(M)$ for $M$ is a $\mathcal{D}$-module, there is a difference between equation (12) and the above result. If we input a $\mathcal{D}$-module instead of $V^{\xi}$, the above result will be the same with equation (12).
\end{rmk}
	Moreover, the action of $y_k$ for $k>1$ is computed by induction.
	\begin{prop}
		The action of $y_k$, for $k =1, \cdots, n$, on the invariant space $(V^{\xi} \otimes V^{\otimes n})^{\ft_0, \mu}$ is computed by
		$$y_k = -\sum_{s,t} (\Delta^{(k-1)}E_s^t)_{(0,k)} \otimes (E_t^s)_k + \dfrac{|\xi|+n}{N} + \dfrac{\mu q - \mu p}{2} - \dfrac{N}{2},$$
		where $(E_s^t)_{(0,k)}$ denotes the tensor product $(V^{\xi} \otimes V^{\otimes (k-1)})$ and hence $\Delta^{(k-1)}E_s^t$ acting on $(E_s^t)_{(0,k)}$.
	\end{prop}
	\begin{proof}
		We verified the action of $y_1$ above. Suppose the statement is true for $y_i$, $i < k$. Let compute the action of $y_k$. By the relation $s_{k-1}y_{k-1} - y_ks_{k-1} = \kappa_1=1$ and the inductive hypothesis, it follows
		\begin{align*}
			y_k &= s_{k-1}y_{k-1}s_{k-1} - s_{k-1}\\
			&= -\sum_{s,t,j,l} (\Delta^{(k-2)}E_s^t)_{(0,k-1)} \otimes (E_l^t  E_t^s E_s^j)_{k-1} \otimes (E_t^l E_j^s)_k \\
			&- \sum_{s,t}(E_s^t)_{k-1} \otimes (E_t^s)_k +
			\dfrac{|\xi|+n}{N} + \dfrac{\mu q - \mu p}{2} - \dfrac{N}{2}\\
			&= -\sum_{s,t,j}   (\Delta^{(k-2)}E_s^t)_{(0,k-1)} \otimes (E_j^j)_{k-1} \otimes (E_t^s)_k \\
			&- \sum_{s,t}(E_s^t)_{k-1} \otimes (E_t^s)_k +\dfrac{|\xi|+n}{N} + \dfrac{\mu q - \mu p}{2} - \dfrac{N}{2}
		\end{align*}
		Take the fact $\sum_{j} (E_j^j)_{k-1} = (I_N)_{k-1}$. The above computation continues
		\begin{align*}
			&= -\sum_{s,t} (\Delta^{(k-2)}E_s^t)_{(0,k-1)} \otimes (I_N)_{k-1} \otimes (E_t^s)_k \\
			&- \sum_{s,t}(E_s^t)_{k-1} \otimes (E_t^s)_k + \dfrac{|\xi|+n}{N} + \dfrac{\mu q - \mu p}{2} - \dfrac{N}{2}\\
			&= -\sum_{s,t}(\Delta^{(k-1)}E_s^t)_{(0,k)} \otimes (E_t^s)_k + \dfrac{|\xi|+n}{N} + \dfrac{\mu q - \mu p}{2} - \dfrac{N}{2}.
		\end{align*}
	\end{proof}

	The Lie algebra $\gl_N$ has a basis $\{E_s^t| 1 \leq s,t \leq N\}$ with the dual basis $\{E_t^s\}$ with respect to the Killing form. Let $C$ denote the Casimir element of $U(\gl_N)$, then $C = \sum_{s,t}E_s^t E_t^s$. The following computation follows
	\begin{align*}
		\Delta(C) =& \sum_{s,t} \Delta(E_s^t) \Delta(E_t^s)\\
		=& \sum_{s,t} (E_s^t \otimes 1 + 1 \otimes E_s^t)(E_t^s \otimes 1 + 1 \otimes E_t^s)\\
		=& (\sum_{s,t} E_{s}^t E_t^s)\otimes 1 + 1 \otimes (\sum_{s,t}E_s^t E_t^s) + 2\sum_{s,t} E_s^t \otimes E_t^s.
	\end{align*}
	Thus
	$$\sum_{s,t} E_s^t \otimes E_t^s = \dfrac{\Delta(C) - C \otimes 1 - 1 \otimes C}{2}.$$
	\subsection{Weights and contents}
	In \cite{R2}, Ram talked about the standard tableaux and representations of affine Hecke algebra of type $C$ and analyzed the weights in terms of boxes. Now let us analyze the weights of $F_{n,p,\mu}(V^{\xi})$ in terms of contents. In section 5, we obtain a basis of the $(\ft_0,\mu)$-invariant space $F_{n,p,\mu}(V^{\xi})=(V^{\xi} \otimes V^{\otimes n})^{\ft_0,\mu}$ indexed by $Tab(\varphi_{n,p,\mu}^{\xi})$, i.e. standard tableaux on a family of skew shapes $\nu / \xi$ where $\nu$ are obtained by Okada's theorem. The action of $y_k$ on the basis element indexed by standard tableau $T$ is by a scalar. Moreover, this scalar is computed in terms of the content of the box fixed by $k$.
	\begin{thm}
		Let $v_T$ denote the basis element of the invariant space indexed by standard tableau $T$. Then $v_T$ is an eigenvector of $y_k$ and the eigenvalue is computed as
		$$
		-cont_T(k) +\s,
		$$
		where $\s= \dfrac{|\xi|+n}{N} +\dfrac{\mu q - \mu p}{2}-\dfrac{N}{2}$.
	\end{thm}
	\begin{proof}
		Let us $T \in Tab(\varphi_{n,p,\mu}^{\xi})$. Since $T$ is a standard tableau, then $T$ corresponds to a sequence $(\nu^{(k)})_{k =0}^{k = n}$ of Young diagrams, where 
		\begin{align*}
			&\nu^{(0)}  =\xi,\\
			&\nu^{(1)} = \xi \cup T(\{1\}),\\
			&\nu^{(2)} = \xi \cup T(\{1, 2\}),\\
			& \cdots \\
			&\nu^{(n)} = \xi \cup T(\{1, 2, \cdots, n\}),
		\end{align*}
		where $T(\{1, \cdots, k\})$ is the collection of cells filled by numbers $1, \cdots, k$, i.e. the Young diagram $\nu^{(k)}$ is formed by adding the cells filled by numbers $1, \cdots, k$ to the Young diagram $\xi$. So it follows, for $k = 1, \cdots, n$, $$v_T \in (V^{\xi} \otimes V^{\otimes k})[\nu^{(k)}] \otimes V^{\otimes (n - k)},$$ where $(V^{\xi} \otimes V^{\otimes k})[V^{\nu^{(k)}}]$ denotes the $V^{\nu_k}$-isotopic component of the tensor product $V^{\xi} \otimes V^{\otimes k}$.
		By the previous subsection 6.1, it follows that the term $\sum_{s,t}(\Delta^{(k-1)}(E_s^t))_{(0,k)} \otimes (E_t^s)_k$ acts on $v_T$ by
		$$\dfrac{C_{(0,k+1)} -C_{(0,k)} \otimes 1_k -1_{(0,k)} \otimes C_k  }{2}.$$
		Moreover, the Casimir element acts on the highest weight module $V^{\nu}$ by the scalar $\langle \nu, \nu + 2 \rho \rangle$, where the weight $2 \rho = \sum_{i = 1}^{N} (N - 2i + 1) \epsilon_i$. So for each $k$ such that $1 \leq k \leq N$, $C_{(0,k+1)}$ acts on $V^{\nu^{(k)}}$ by the scalar $\langle \nu^{(k)}, \nu^{(k)} + 2\rho \rangle$, $C_{(0,k)}$ acts on $V^{\nu^{(k-1)}}$ by the scalar $\langle \nu^{(k-1)}, \nu^{(k-1)} + 2\rho \rangle$ and $C_k$ acts on $V$ by the scalar $\langle \epsilon, \epsilon + 2\rho \rangle = N$, namely $$\dfrac{C_{(0,k+1)} -C_{(0,k)} \otimes 1_k -1_{(0,k)} \otimes C_k  }{2}$$ acts by
		$$\dfrac{1}{2}(\langle \nu^{(k)}, \nu^{(k)} + 2\rho \rangle - \langle \nu^{(k - 1)}, \nu^{(k - 1)} + 2 \rho \rangle-  \langle \epsilon, \epsilon + 2 \rho \rangle) .$$
		Let $T(k)$ be the cell $(\ci(k), \cj(k))$, then $\nu^{(k)}_{\ci(k)} = \cj(k) = \nu^{(k-1)}_{\ci(k)} + 1$ and $\nu_i^{(k)} = \nu_i^{(k-1)}$, for $i \neq \ci(k)$.
		\begin{align*}
			&\dfrac{1}{2}(\langle \nu^{(k)}, \nu^{(k)} + 2\rho \rangle - \langle \nu^{(k - 1)}, \nu^{(k - 1)} + 2 \rho \rangle - \langle \epsilon, \epsilon + 2 \rho \rangle)\\
			= &\dfrac{1}{2}((\cj(k) + N -2\ci(k) + 1)(\cj(k)) - (\cj(k)+N-2\ci(k))(\cj(k)-1) -N)\\
			= &  \cj(k) - \ci(k).
		\end{align*}
		Then the statement follows.
	\end{proof}
	\begin{thm}
		Let $F_{n,p,\mu}(V^{\xi})$ denote the image of the irreducible $GL_N$-module $V^{\xi}$, for some $\xi \in P^+$, under Etingof-Freund-Ma functor. Then $F_{n,p,\mu}(V^{\xi})$ has a basis indexed tableaux in $\tab$, i.e. $\{v_T|T \in \tab\}$. This basis is a weight basis with each basis vector $v_T$ is a weight vector of weight $\zeta_T=-cont_T+\s$. So $F_{n,p,\mu}(V^{\xi})$ is a $\Y$-semisimple representation of $H_n(1,p-q-\mu N)$. Moreover, it is obvious different standard tableaux give different weights. Hence each weight space is one dimensional.
	\end{thm}

	\section{Intertwining operators}
	\subsection{Definition of intertwining operators}~\\
	\begin{definition}
		For $i = 1, \cdots, n-1$, define the intertwining operators
		$$\phi_i = [s_i , y_i],$$
		and for $\gamma_n$, define
		$$\phi_n = [\gamma_n, y_n].$$
	\end{definition}
	\begin{prop}
		The intertwining operators $\phi_i$ satisfy the braid relations
		\begin{align*}
			&\phi_i \phi_{i+1} \phi_i = \phi_{i+1} \phi_i \phi_{i+1}, i = 1, \cdots, n - 1,\\
			&\phi_i \phi_j = \phi_j \phi_i, |i - j| > 1,\\
			&\phi_{n-1} \phi_n \phi_{n-1} \phi_n = \phi_n \phi_{n-1} \phi_n \phi_{n-1}.
		\end{align*}
	\end{prop}
	Since the operators $\phi_i$'s satisfy the same braid relations with $s_i$'s and $\gamma_n$, it makes sense to define the following.
	\begin{definition}
		Let $W_0$ denote the finite Weyl group of type $C_n$, for each $w \in W_0 $, it has a reduced expression $w = s_{i_1} s_{i_2} \dots s_{i_m}$, $l(w)=m$, here we take the convention $s_n = \gamma_n$. Define
		$$\phi_w = \phi_{i_1} \phi_{i_2} \dots \phi_{i_m}.$$
	\end{definition}
	
	\subsection{Properties of intertwining operators}~\\
	Some computations on intertwining operators:
	\begin{enumerate}
		\item $\phi_i = s_i (y_i - y_{i+1}) - 1$,\\
		$\phi_n = 2 \gamma_n y_n - \kappa_2$.\\
		\item $\phi_i ^2 = (1 - y_i + y_{i+1})(1 + y_i - y_{i+1})$,\\\\
		$\phi_n ^2 = (\kappa_2 - 2y_n)(\kappa_2 + 2y_n)$.\\
	\end{enumerate}
	\begin{definition}
		Define the actions of $W_0$ on weight $\zeta=[\zeta_1, \cdots, \zeta_n]$: for an arbitrary $w \in W_0$, the action of $w$ is 
		$$w . \zeta =  \zeta \circ w^{-1},$$
		where we take $\zeta_{-k}=-\zeta_k$.
	\end{definition}

	\begin{thm}
		Let $L$ be a $\Y$-semisimple module and $L_{\zeta}$ denote the weight space of weight $\zeta$, then
		$$\phi_{w} L_{\zeta} \subset L_{w. \zeta}.$$
	\end{thm}
	\begin{proof}
		It suffices to show the statement is true for each operator $\phi_i$.\\
		Case 1. When $1 \leq i \leq n-1$. We have the following facts that
		$$y_i \phi_i  = \phi_i y_{i + 1},$$
		$$y_{i + 1} \phi_i  = \phi_i y_i,$$
		and 
		$$y_j \phi_i = \phi_i y_j, j \neq i \text{ or } i+1.$$
		Case 2. Consider $\phi_n$. We have facts that
		$$y_n \phi_n = -\phi_n y_n,$$
		$$y_j \phi_n = \phi_n y_j,  j \neq n. $$
	\end{proof}
	\begin{rmk}
		Since each weight space of $\F_{n,p,\mu}(V^{\xi})$ is one dimensional, so the action of $\phi_i$ is either $0$ or an isomorphism.
	\end{rmk}
	
	\begin{lemma}
		If   $\zeta_i - \zeta_{i+1} \neq \pm 1$ for some $i \in \{1, 2, \cdots, n-1 \}$, then $\phi_i v_{\zeta} \neq 0$, where $v_{\zeta}$ is the weight vector of the weight $\zeta$.
	\end{lemma}
	\begin{proof}
		Suppose that $\phi_i v_{\zeta} = 0$. Then $\phi_i^2 v_{\zeta} = 0$. By the computation above $\phi_i ^2 = (1 - y_i + y_{i+1})(1 + y_i - y_{i+1})$. Then $\phi_i^2 v_{\zeta} = (1 - \zeta_i + \zeta_{i+1})(1 + \zeta_i - \zeta_{i+1})v_{\zeta} = 0$. Then we have that $\zeta_i - \zeta_{i+1} = \pm 1$.
	\end{proof}
	Similarly, we have the following fact.
	\begin{lemma}
		If $\zeta_n \neq \pm \frac{\kappa_2}{2}$, then $\phi_n v_{\zeta} \neq 0$, where $v_{\zeta}$ is the weight vector of the weight $\zeta$.
	\end{lemma}
	\begin{proof}
		Suppose that $\phi_n v_{\zeta} = 0$. Then $\phi_n^2 v_{\zeta} = 0$. By the computation above $\phi_n ^2 = (\kappa_2 - 2y_n)(\kappa_2 + 2y_n)$. Then $\phi_n^2 v_{\zeta} = \phi_n ^2 = (\kappa_2 - 2 \zeta_n)(\kappa_2 + 2 \zeta_n)v_{\zeta} = 0$. Then we have that $\zeta(n) = \pm \frac{\kappa_2}{2}$.
	\end{proof}
	
	\subsection{Properties of irreducible $\Y$-semisimple representations}
	Let $L$ be an irreducible $\Y$-semisimple representation of $H_n(\kappa_1,\kappa_2)$. Let $\zeta=[\zeta_1,\cdots,\zeta_n]$ is a weight $L$.
	\begin{thm}
		If $\zeta_i = \zeta_{i+1}$ for some $1 \leq i \leq n-1$ , then $L_{\zeta} = 0$.
	\end{thm}
	\begin{proof}
		Let $\zeta$ be a weight such that $\zeta_i = \zeta_{i+1}$. Suppose there exists a nonzero element $v \in L_{\zeta}$. Consider the vector $s_i v$. Since $\phi_i = s_i (y_i - y_{i+1}) - 1 = (y_{i+1} - y_i)s_i + 1$, we have $\phi_i v = -v$.
		\begin{align*}
			(y_i - y_{i+1})s_i v = &(1 - \phi_i) v \\
			=  &2v \neq 0.
		\end{align*}
		And act again by $y_i - y_{i+1}$,
		\begin{align*}
			&(y_i - y_{i+1})^2 s_i v\\
			= &2(y_i - y_{i+1})v =0.
		\end{align*}
		This means $s_i v$ belongs to the generalized eigenspace of $y_i - y_{i+1}$ and does not belong to the eigenspace of $y_i - y_{i+1}$, which contradicts $\Y$-semisimplicity.
	\end{proof}
	
	\begin{thm}
		Let $\kappa_2 \neq 0$. If $\zeta_n = 0$, then $L_{\zeta} = 0$.
	\end{thm}
	\begin{proof}
		Let $\zeta$ be a weight such that $\zeta_n = 0$. Suppose there exists a nonzero element $v \in L_{\zeta}$. Consider the vector $\gamma_n v$. Since $\phi_n = 2\gamma_n y_n - \kappa_2 = -2y_n \gamma_n + \kappa_2$, we have $\phi_n v = -\kappa_2 v$.
		\begin{align*}
			2y_n\gamma_n v = &(\kappa_2 - \phi_n) v\\
			= &2\kappa_2 v \neq 0.
		\end{align*}
		Act again by $y_n$, we have
		\begin{align*}
			&2{y_n}^2 \gamma_n v\\
			= &2 \kappa_2 y_n v = 0.
		\end{align*}
		his means $s_i v$ belongs to the generalized eigenspace of $y_n$ and does not belong to the eigenspace of $y_n$, which contradicts $\Y$-semisimplicity.
	\end{proof}
	
	\begin{rmk}
		When $\kappa_2 = 0$, it is possible for an irreducible $\Y$-semisimple module $L$ to contain a nonzero weight space $L_{\zeta}$ with $\zeta_n = 0$. In this case, $\gamma_n v  \in \C v$. Otherwise, the vector $v + \gamma_n v$ generalizes a nonzero proper submodule of $L$, which contradicts the irreducibility.
	\end{rmk}
	
	\begin{lemma}
		For any arbitrary $w \in W_0$, the intertwining operator 
		$$\phi_w = w \Pi_{\alpha_{ij} \in R(w)} (y_i - y_j) + \sum_{x <w} x P(y),$$
		where $P(y)$ is a polynomial of $y_1, \cdots, y_n$.
	\end{lemma}
	
	\begin{thm}
		Let $\zeta$ be a weight of $L$ such that $L_{\zeta} \neq 0$. Let $v$ be a nonzero weight vector in $L_{\zeta}$. Then the set $\{\phi_w v  | w \in  W_0\}$ spans the irreducible representation $L$.
	\end{thm}
	
	\begin{proof}
		We need to show $w. v$ lies in the span of $\{\phi_w v  | w \in S_n \ltimes (\Z/2\Z)^n\}$ for any arbitrary $w \in S_n \ltimes (\Z/2\Z)^n$. We prove by induction on the length of $w$. When the length of $w$ is zero, the statement is trivial. Now assume for $w$ with $l(w) <k $, the statement holds, i.e. $w. v$ can be expressed by a linear combination of elements in $\{\phi_w v  | w \in W_0\}$. Set $w$ is of length $k$ and $w = s_{i_1} \cdots s_{i_k}$. Then by Lemma 7.12, we have $\phi_w \cdot v= \Pi_{\alpha_{ij} \in R(w)} (\zeta_i - \zeta_j) \cdot w \cdot v+ \Sigma_{x <w} c_x x \cdot v$. Since $l(x) < k$, the terms $x \cdot v$ can be express by $\{\phi_w v  | w \in S_n \ltimes (\Z/2\Z)^n\}$. As long as the coefficient $ \Pi_{\alpha_{ij} \in R(w)}(\zeta_i - \zeta_j) \neq 0$, $w \cdot v$ can be express by $\{\phi_w v  | w \in W_0\}$. So it is reduced to consider only the case when $ \Pi_{\alpha_{ij} \in R(w)}(\zeta_i - \zeta_j) = 0$.\\
		In this case, there exists $p \in [1,k]$ such that $\Pi_{\alpha_{ij} \in R(s_{i_{p+1}} \cdots s_{i_k})}(\zeta_i - \zeta_j) \neq 0$ and $\Pi_{\alpha_{ij} \in R(s_{i_p} \cdots s_{i_k})}(\zeta_i - \zeta_j) = 0$. Set $u = s_{i_{p+1}} \cdots s_{i_k}$. When $i_p \in [1, n-1]$, this implies $(y_{i_p} - y_{i_{p + 1  }})\phi_u v = 0$ and hence $\phi_u v = 0$ by theorem 4.1. And when $i_p = n$, this implies $2y_n \phi_u v = 0$ and hence $\phi_u v = 0$ by theorem 4.2. It follows $\Pi_{a_{ij} \in R(u)}(\zeta_i-\zeta_j)u. v= \sum_{x<u}xP(y)v$ and hence $\Pi_{a_{ij} \in R(u)}w. v= \sum_{x<u}s_{i_1}\cdots s_{i_p}xP(y)$. Since $l(s_{i_1}\cdots s_{i_p}x)<k$, then $(s_{i_1}\cdots s_{i_p}x) . v$ and hence $w . v$ can be expressed by a linear combination of elements in $\{\phi_w v  | w \in W_0\}$.
	\end{proof}
	
	\begin{thm}
		Let $\zeta$ be a weight such that $L_{\zeta} \neq 0$. Let $ w \neq 1 \in W_0$ such that $w. \zeta = \zeta$. Then $\phi_w v = 0$.
	\end{thm}
	\begin{proof}
		Let $w=s_{i_1}\cdots s_{i_k}$. since $w . \zeta=\zeta$, there is $1 \leq p \leq k$ such that $s_{i_1} \cdots s_{i_p} = (h m)$ where $\zeta_h=\zeta_m$. Consider $\phi_{i_{p-1}}\cdots \phi_{i_1} \phi_w v = \Pi_{1 \leq j \leq p} (1-\zeta_{i_j}+\zeta_{i_j+1})(1+\zeta_{i_j}-\zeta_{i_j+1}) \phi_u v$. It follows $\phi_u v =0$ and hence $\phi_w v=0$.
	\end{proof}
	
	\begin{cor}
		Let $\zeta$ be a weight such that $L_{\zeta} \neq 0$. Then it follows $dim(L_{\zeta}) = 1$.
	\end{cor}

	\begin{prop}
		\begin{enumerate}
			\item Let $v$ be a nonzero weight vector of weight $\zeta$ such that $|\zeta_i -\zeta_{i+1}| = 1$. Then $\phi_i v =0$.\\
			\item Let $v$ be a nonzero weight vector of weight $\zeta$ such that $\zeta_n = \pm \frac{\kappa_2}{2}$. Then $\phi_n v =0$.
		\end{enumerate}
	\end{prop}
	\begin{rmk}
		Some similar results also happen in degenerate affine Hecke algebra of type $A_{n-1}$. Let $H_n(1)$ be the degenerate affine Hecke algebra generated by $s_i (i = 1, \cdots n-1)$ and $y_i(i = 1 \cdots n)$ with the following relations:
		\begin{align*}
			&s_i^2 = 1, i = 1, \cdots, n - 1,\\
			&s_i s_j = s_j s_i, |i - j| > 1,\\
			&s_i s_{i + 1} s_i = s_{i + 1} s_i s_{i + 1}, i = 1, \cdots, n - 1,\\
			&y_i y_j = y_j y_i,\\
			&s_i y_i - y_{i + 1} s_i = 1,\\
			&s_i y_j = y_j s_i, j \neq i, i+1.
		\end{align*}
		There is the same definition of $\Y$-semisimple representation. And for any $\Y$-semisimple representation $M$, if a weight $\zeta$ with $\zeta_i = \zeta_{i+1}$, then $M_{\zeta} = 0$.\\
		Furthermore, we still could define the intertwining operator $\phi = s_i y_i - y_i s_i$, then we will also have $\phi_i^2  = (1- y_i + y_{i + 1})(1 + y_i - y_{i + 1})$. This also implies the fact that if $\phi_i v_{\zeta} = 0$ then we have $\zeta_i - \zeta_{i + 1} = \pm 1$. For the double affine Hecke algebra of type $A$, \cite{SV} explored similar properties in details.
		
	\end{rmk}
	

	\section{Combinatorial moves}
	
	
	\subsection{Moves among standard tableaux}~\\
	Let $\tab$ denote the collection of standard tableaux indexing the basis of $F(V^{\xi})$ in section 5. We define a set of moves $\m_1, \cdots,\m_n$ on $\tab \sqcup \{\0\}$ as follows. The move $\m_i$ for $i-1, \cdots, n-1$ is defined as
	$$
	\m_i(T) =
	\begin{cases}
		T', &T' \text{ is a standard tableau}\\
		\0, &\text{ otherwise, }
	\end{cases}
	$$
	where $T'(k)=T(s_i(k))$.
	The move $\m_n$ is defined to be 
	$$
	\m_n \cdot T =
	\begin{cases}
		\0, & \ci(n) \leq max(p,q) \text{ and } \cj(n) \leq max(a,b)\\
		T'', &\text{ otherwise,}
	\end{cases}
	$$
	where $T''(j) =T(j)$ for each $j \neq n$ and $T''(n) = (N - \ci(n)+ 1, a+b-\cj(n)+1)$.
	\begin{rmk}
		There is a easy observation.  For any shape $\varphi' \in D(\varphi^{\xi}_{n,p,\mu})$ and any $i \leq min(p,q)$, the sum of the column number of the last cell of the $i$-th row and the column number of the last cell of the $(N-i+1)$-th row equal $a+b$. So $ T''(n)=(N-\ci(n)+1, a+b -\cj(n)+1)$ means that the $\m_n$-move takes the cell filled by $n$ to the end of the $(N-\ci(n)+1)$-th row.
	\end{rmk}
	here $m$ be the column number of the last cell of the $(N-\ci(n)+1)$-th row of $Im(T)$.
	
	\subsection{Correspondence between algebraic actions and combinatorial moves}~\\
	Let $v_T$ denote the basis vector indexed by $T \in \tab$ and $\zeta_T$ denote the weight of $v_T$, i.e. $\zeta_T=-cont_T+\s$.
	
	\begin{prop}
		\begin{enumerate}
			\item For $i=1, \cdots, n-1$, if $\m_i(T) \neq \0$ holds, then $\m_i(T) \in \tab$ and the common eigenbasis vector $v_{\m_i(T)}$ is of weight $\zeta_{\m_i(T)}=s_i. \zeta_T$.
			\item If $\m_n(T) \neq \0$, then $\m_n(T) \in \tab$ and the common eigenbasis vector $v_{\m_n(T)}$ is of weight $\zeta_{\m_i(T)}=\gamma_n.\zeta_T$
		\end{enumerate}
	\end{prop}
	\begin{proof}
		First, for $i=1, \cdots, n-1$, if $\m_i(T) \neq \0$, then by the definition of the move $\m_i$, $T \in \tab$ and we want to show $\zeta_{\m_i(T)}=s_i.\zeta_T$.
		
		Then let us consider the case when $w = \gamma$. In this case $w$ moves the box filled by $n$ in the $\ci$-th row of  tableau $T$ to the end of the $(N-\ci+1)$-th row. So the only box in the new tableau $\gamma . T$ with a different position comparing with the tableau $T$ is the box filled by $n$. Thus the only difference in the new weight associated to $\gamma . T$ comparing with $\zeta_T$ is the eigenvalue of $y_n$. Let $(\ci , \cj)$ denote the coordinates of the box filled by $n$ in the tableau $T$. Then the coordinates of the box filled by $n$ in the new tableau $\gamma . T$ is $(N - \ci + 1, \mu(q-p) + 2 \dfrac{|\xi| + n}{N} - \cj + 1)$. Then the eigenvalue of $y_n$ in the new weight $\zeta_{\gamma . T}$ associated to $\gamma . T$ is $\cj - \ci -\dfrac{|\xi|+n}{N}+\dfrac{N}{2}+\dfrac{\mu(p-q)}{2}$. So the new weight equals $\gamma . \zeta_T$.\\

	\end{proof}

	\begin{prop}
		If $w . T \neq 0$ for some $w \in W_0$, then $\phi_w v_T \neq 0$.
	\end{prop}
	\begin{proof}
		It is enough to verify the statement when $w$ is the transposition $s_i$ or $\gamma_n$.\\
		First, consider the case when $w = s_i$, $i=1, \cdots,n-1$. Suppose $\phi_i  v_T = 0$ for some $1 \leq i \leq n-1$ implies that $\phi_i^2  v_T = 0$ and $\phi_i ^2 = (1 - y_i + y_{i+1})(1 + y_i - y_{i+1})$. Then $\zeta_T(i) - \zeta_T(i + 1) = \pm1$. In this case the contents of boxes filled by $i$ and $i + 1$ differ by $1$ and hence the two boxes are adjacent and in the same row or in the same column. We have $s_i . T = 0$ in this case. This contradicts the condition. So we have $\phi_i  v_T \neq 0$.\\
		Second, consider the case when $w = \gamma_n$. Suppose $\phi_n  v_T = 0$ which implies the eigenvalue of $y_n$ is $\pm \frac{\kappa_2}{2}$. Since $\phi_n^2  v_T = 0$ in this case and $\phi_n ^2 = (\kappa_2 - 2y_n)(\kappa_2 + 2y_n)$. Then the box filled by $n$ is either $(p, \mu q + \frac{|\xi|+n}{N})$ or $(q, -\mu p + \frac{|\xi|+n}{N})$. But by the definition of action of $\gamma_n$ on the tableau $T$, we have in both cases that $\gamma_n . T = 0$. This contradicts the condition. Hence we have that $\phi_n v_T \neq 0$.
	\end{proof}
	
	\begin{rmk}
		\begin{enumerate}
			\item If $\m_i(T) \neq \0$, then $\phi_i v_T=c v_{\m_i(T)}$ up to a nonzero scalar $c \in \C$ for $i=1, \cdots, n$.
			\item If $\m_i(T) = \0$, then $\phi_i v_T=0$ for $i=1, \cdots, n$.
		\end{enumerate}
	\end{rmk}
	
	\begin{eg}
		In example 5.9, the action of intertwining operators are as follows. The diagonals give the eigenvalue of $y_i$'s.
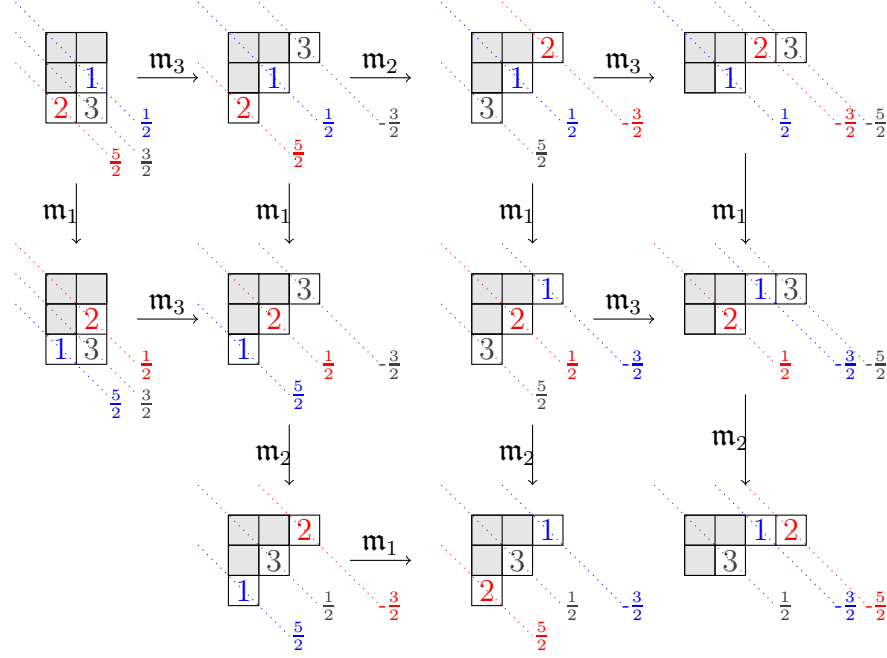
		\begin{center}
			\begin{tikzpicture}[scale=0.4]
				\begin{scope}[shift={(-10,16)}]
					\draw [thin,fill=gray!20] (0,2) rectangle (2,3);
					\draw [thin,fill=gray!20] (0,1) rectangle (1,2);
					\draw[step=1] (0,0) grid (2,3);
					\draw (0.5, 0.5) node[red] {$2$};
					\draw (1.5, 1.5) node[blue] {$1$};
					\draw (1.5, 0.5) node[black!75] {$3$};
					\draw [dotted,blue] (-1,4) -- (3, 0);
					\draw [dotted,red] (-1,2) -- (2, -1);
					\draw [dotted,black!75] (-1,3) -- (3, -1);
					\draw (3.3, 0) node[blue]{\scriptsize $\frac{1}{2}$};
					\draw (2.3, -1.3) node[red]{\scriptsize $\frac{5}{2}$};
					\draw (3.3, -1.3) node[black!75]{\scriptsize $\frac{3}{2}$};
				\end{scope}
				
				\draw [->](-9,14)--(-9,12);
				\draw (-9.5,13) node {$\m_1$};
				\draw [->](-7, 17.5) --(-5, 17.5);
				\draw (-6,18) node {$\m_3$};
				\draw [->](0, 17.5) --(2, 17.5);
				\draw (1,18) node {$\m_2$};
				\draw [->](8, 17.5) --(10, 17.5);
				\draw (9,18) node {$\m_3$};
				\draw [->](8, 9.5) --(10, 9.5);
				\draw (9,10) node {$\m_3$};
				\draw [->](0, 1.5) --(2, 1.5);
				\draw (1,2) node {$\m_1$};
				
				\begin{scope}[shift={(-10,8)}]
					\draw [thin,fill=gray!20] (0,2) rectangle (2,3);
					\draw [thin,fill=gray!20] (0,1) rectangle (1,2);
					\draw[step=1] (0,0) grid (2,3);
					\draw (0.5, 0.5) node[blue] {$1$};
					\draw (1.5, 1.5) node[red] {$2$};
					\draw (1.5, 0.5) node[black!75] {$3$};
					\draw [dotted,red] (-1,4) -- (3, 0);
					\draw [dotted,blue] (-1,2) -- (2, -1);
					\draw [dotted,black!75] (-1,3) -- (3, -1);
					\draw (3.3, 0) node[red]{\scriptsize $\frac{1}{2}$};
					\draw (2.3, -1.3) node[blue]{\scriptsize $\frac{5}{2}$};
					\draw (3.3, -1.3) node[black!75]{\scriptsize $\frac{3}{2}$};
				\end{scope}
				
				\draw [->](-7, 9.5) --(-5, 9.5);
				\draw (-6,10) node {$\m_3$};
				
				\begin{scope}[shift={(-4,8)}]
					\draw [thin,fill=gray!20] (0,2) rectangle (2,3);
					\draw [thin,fill=gray!20] (0,1) rectangle (1,2);
					\draw[step=1] (0,1) grid (2,3);
					\draw (2,2) rectangle (3,3);
					\draw (0,0) rectangle (1,1);
					\draw (0.5, 0.5) node[blue] {$1$};
					\draw (1.5, 1.5) node[red] {$2$};
					\draw (2.5, 2.5) node[black!75] {$3$};
					\draw [dotted,red] (-1,4) -- (3, 0);
					\draw [dotted,blue] (-1,2) -- (2, -1);
					\draw [dotted,black!75] (1,4) -- (5, 0);
					\draw (2.3, -1) node[blue]{\scriptsize $\frac{5}{2}$};
					\draw (3.3, 0) node[red]{\scriptsize $\frac{1}{2}$};
					\draw (5.3, 0) node[black!75]{\scriptsize -$\frac{3}{2}$};
				\end{scope}
				
				\begin{scope}[shift={(-4,0)}]
					\draw [thin,fill=gray!20] (0,2) rectangle (2,3);
					\draw [thin,fill=gray!20] (0,1) rectangle (1,2);
					\draw[step=1] (0,1) grid (2,3);
					\draw (2,2) rectangle (3,3);
					\draw (0,0) rectangle (1,1);
					\draw (0.5, 0.5) node[blue] {$1$};
					\draw (1.5, 1.5) node[black!75] {$3$};
					\draw (2.5, 2.5) node[red] {$2$};
					\draw [dotted,black!75] (-1,4) -- (3, 0);
					\draw [dotted,blue] (-1,2) -- (2, -1);
					\draw [dotted,red] (1,4) -- (5, 0);
					\draw (2.3, -1) node[blue]{\scriptsize $\frac{5}{2}$};
					\draw (3.3, 0) node[black!75]{\scriptsize $\frac{1}{2}$};
					\draw (5.3, 0) node[red]{\scriptsize -$\frac{3}{2}$};
				\end{scope}
				
				\begin{scope}[shift={(-4,16)}]
					\draw [thin,fill=gray!20] (0,2) rectangle (2,3);
					\draw [thin,fill=gray!20] (0,1) rectangle (1,2);
					\draw[step=1] (0,1) grid (2,3);
					\draw (2,2) rectangle (3,3);
					\draw (0,0) rectangle (1,1);
					\draw (0.5, 0.5) node[red] {$2$};
					\draw (1.5, 1.5) node[blue] {$1$};
					\draw (2.5, 2.5) node[black!75] {$3$};
					\draw [dotted,blue] (-1,4) -- (3, 0);
					\draw [dotted,red] (-1,2) -- (2, -1);
					\draw [dotted,black!75] (1,4) -- (5, 0);
					\draw (2.3, -1) node[red]{\scriptsize $\frac{5}{2}$};
					\draw (3.3, 0) node[blue]{\scriptsize $\frac{1}{2}$};
					\draw (5.3, 0) node[black!75]{\scriptsize -$\frac{3}{2}$};
				\end{scope}

				\begin{scope}[shift={(4,0)}]
					\draw [thin,fill=gray!20] (0,2) rectangle (2,3);
					\draw [thin,fill=gray!20] (0,1) rectangle (1,2);
					\draw[step=1] (0,1) grid (2,3);
					\draw (2,2) rectangle (3,3);
					\draw (0,0) rectangle (1,1);
					\draw (0.5, 0.5) node[red] {$2$};
					\draw (1.5, 1.5) node[black!75] {$3$};
					\draw (2.5, 2.5) node[blue] {$1$};
					\draw [dotted,black!75] (-1,4) -- (3, 0);
					\draw [dotted,red] (-1,2) -- (2, -1);
					\draw [dotted,blue] (1,4) -- (5, 0);
					\draw (2.3, -1) node[red]{\scriptsize $\frac{5}{2}$};
					\draw (3.3, 0) node[black!75]{\scriptsize $\frac{1}{2}$};
					\draw (5.3, 0) node[blue]{\scriptsize -$\frac{3}{2}$};
				\end{scope}
				
				\begin{scope}[shift={(4,16)}]
					\draw [thin,fill=gray!20] (0,2) rectangle (2,3);
					\draw [thin,fill=gray!20] (0,1) rectangle (1,2);
					\draw[step=1] (0,1) grid (2,3);
					\draw (2,2) rectangle (3,3);
					\draw (0,0) rectangle (1,1);
					\draw (0.5, 0.5) node[black!75] {$3$};
					\draw (1.5, 1.5) node[blue] {$1$};
					\draw (2.5, 2.5) node[red] {$2$};
					\draw [dotted,blue] (-1,4) -- (3, 0);
					\draw [dotted,black!75] (-1,2) -- (2, -1);
					\draw [dotted,red] (1,4) -- (5, 0);
					\draw (2.3, -1) node[black!75]{\scriptsize $\frac{5}{2}$};
					\draw (3.3, 0) node[blue]{\scriptsize $\frac{1}{2}$};
					\draw (5.3, 0) node[red]{\scriptsize -$\frac{3}{2}$};
				\end{scope}
				
				\begin{scope}[shift={(4,8)}]
					\draw [thin,fill=gray!20] (0,2) rectangle (2,3);
					\draw [thin,fill=gray!20] (0,1) rectangle (1,2);
					\draw[step=1] (0,1) grid (2,3);
					\draw (2,2) rectangle (3,3);
					\draw (0,0) rectangle (1,1);
					\draw (0.5, 0.5) node[black!75] {$3$};
					\draw (1.5, 1.5) node[red] {$2$};
					\draw (2.5, 2.5) node[blue] {$1$};
					\draw [dotted,red] (-1,4) -- (3, 0);
					\draw [dotted,black!75] (-1,2) -- (2, -1);
					\draw [dotted,blue] (1,4) -- (5, 0);
					\draw (2.3, -1) node[black!75]{\scriptsize $\frac{5}{2}$};
					\draw (3.3, 0) node[red]{\scriptsize $\frac{1}{2}$};
					\draw (5.3, 0) node[blue]{\scriptsize -$\frac{3}{2}$};
				\end{scope}
				
				\begin{scope}[shift={(11,16)}]
					\draw [thin,fill=gray!20] (0,2) rectangle (2,3);
					\draw [thin,fill=gray!20] (0,1) rectangle (1,2);
					\draw[step=1] (0,1) grid (2,3);
					\draw[step=1] (2,2) grid (4,3);
					\draw (1.5, 1.5) node[blue] {$1$};
					\draw (2.5, 2.5) node[red] {$2$};
					\draw (3.5, 2.5) node[black!75] {$3$};
					\draw [dotted,blue] (-1,4) -- (3, 0);
					\draw [dotted,black!75] (2,4) -- (6, 0);
					\draw [dotted,red] (1,4) -- (5, 0);
					\draw (6.3, 0) node[black!75]{\scriptsize -$\frac{5}{2}$};
					\draw (3.3, 0) node[blue]{\scriptsize $\frac{1}{2}$};
					\draw (5.3, 0) node[red]{\scriptsize -$\frac{3}{2}$};
				\end{scope}
				\begin{scope}[shift={(11,8)}]
					\draw [thin,fill=gray!20] (0,2) rectangle (2,3);
					\draw [thin,fill=gray!20] (0,1) rectangle (1,2);
					\draw[step=1] (0,1) grid (2,3);
					\draw[step=1] (2,2) grid (4,3);
					\draw (1.5, 1.5) node[red] {$2$};
					\draw (2.5, 2.5) node[blue] {$1$};
					\draw (3.5, 2.5) node[black!75] {$3$};
					\draw [dotted,red] (-1,4) -- (3, 0);
					\draw [dotted,black!75] (2,4) -- (6, 0);
					\draw [dotted,blue] (1,4) -- (5, 0);
					\draw (6.3, 0) node[black!75]{\scriptsize -$\frac{5}{2}$};
					\draw (3.3, 0) node[red]{\scriptsize $\frac{1}{2}$};
					\draw (5.3, 0) node[blue]{\scriptsize -$\frac{3}{2}$};
				\end{scope}
				\begin{scope}[shift={(11,0)}]
					\draw [thin,fill=gray!20] (0,2) rectangle (2,3);
					\draw [thin,fill=gray!20] (0,1) rectangle (1,2);
					\draw[step=1] (0,1) grid (2,3);
					\draw[step=1] (2,2) grid (4,3);
					\draw (1.5, 1.5) node[black!75] {$3$};
					\draw (2.5, 2.5) node[blue] {$1$};
					\draw (3.5, 2.5) node[red] {$2$};
					\draw [dotted,black!75] (-1,4) -- (3, 0);
					\draw [dotted,red] (2,4) -- (6, 0);
					\draw [dotted,blue] (1,4) -- (5, 0);
					\draw (6.3, 0) node[red]{\scriptsize -$\frac{5}{2}$};
					\draw (3.3, 0) node[black!75]{\scriptsize $\frac{1}{2}$};
					\draw (5.3, 0) node[blue]{\scriptsize -$\frac{3}{2}$};
				\end{scope}
				
				\draw [->] (-2,14)--(-2,12);
				\draw (-2.5,13) node {$\m_1$};
				\draw [->] (6,14)--(6,12);
				\draw (5.5,13) node {$\m_1$};
				\draw [->](13,15)--(13,12);
				\draw (12.5,13) node {$\m_1$};
				
				\draw [->](-2,6)--(-2,4);
				\draw (-2.5,5) node {$\m_2$};
				\draw [->](6,6)--(6,4);
				\draw (5.5,5) node {$\m_2$};
				\draw [->](13, 7)--(13,4);
				\draw (12.5,5.5) node {$\m_2$};
				
			\end{tikzpicture}
		\end{center}
		
	\end{eg}
	
	Let $k$ be the filling of the cell $(q,b)$,we could compute that the eigenvalue of $y_k$ is $-\frac{\kappa_2}{2}$. Similarly, let $k$ be the filling of the cell $(p,a)$, it follows the eigenvalue of $y_k$ is $\frac{\kappa_2}{2}$. Furthermore, $\kappa_2 = p-q -a+b$.

	\section{Irreducible representations}
	\subsection{The image $
		F_{n,p,\mu}(V^{\xi})$ is irreducible}
	
	\begin{lemma}
		Let $\varphi_1$ and $\varphi_2$ be two skew shapes in $D(\varphi)$ with $\varphi_1 \to \varphi_2$. Then there exist standard tableaux $T_1$ and $T_2$ with $Im(T_1) = \varphi_1$ and $Im(T_2)=\varphi_2$ such that $\gamma_n(T_1) = T_2$.
	\end{lemma}
	
	\begin{proof}
		The $\varphi_1 \to \varphi_2$ implies that $\varphi_2$ is obtained by moving a corner $(i, \varphi_{i})$ of $\varphi_1$ to the end of the $(N-i+1)$-th row of $\varphi_1$. Since $(i, \varphi_1)$ is a corner of $\varphi_1$, there exists a standard tableau $T_1$ such that $(i, \varphi_1)$ is filled by $n$. Applying the $\gamma_n$ move to $T_1$, let $T_2 = \gamma_n(T_1)$. Then $T_2$ is a standard tableau with $Im(T_2)=\varphi_2$.
	\end{proof}
	
	We show in the following the representation of degenerate affine Hecke algebra obtained through Etingof-Freund-Ma functor is irreducible.
	\begin{thm}
		The image $\F_{n,p,\mu}(V^{\xi})$ of a finite dimensional irreducible $\gl_N$-module $V^{\xi}$ under the Etingof-Freund-Ma functor is irreducible.
	\end{thm}
	\begin{proof}
		A basis of $F_{n,p,\mu}(V^{\xi})$ is indexed by $$\mathcal{T} = \{T| T \text{ is a standard tableau and } Im(T) \in D(\varphi_{n,p,\mu}^{\xi})\}.$$ It's obvious to see that the underlying vector space of $F_{n,p,\mu}(V^{\xi})$ is isomorphic to $span_{\C}\{v_T | T \in \mathcal{T}\}$. Let $N$ be a submodule of $F_{n,p,\mu}(V^{\xi})$. Then $N$ contains at least one weight vector of $F_{n,p,\mu}(V^{\xi})$. Let $v_T$ be a weight vector associated to the tableau $T \in \tab$ and the submodule $N$ contains $v_T$.\\
		We show in the following we could get every other weight vector from an arbitrary weight vector $v_T$. We could consider the actions of signed permutations on standard tableaux since the actions of signed permutations on standard tableaux are compatible with the actions of intertwining operators on weight vectors.\\
		Case 1. For any the standard tableau $T'$ with the same shape of the tableau $T$, there exists $w \in S_n$ such that $T' = w . T$. Equivalently $v_{T'} =c \phi_{\omega}v_T$ where $c \in \C$ is nonzero.\\
		Case 2. For standard $T_1$ and $T_2$ with $Im(T_1) \to Im(T_2)$, combining Lemma 8.4 and Case 1, it follows $T_2 = \omega(T_1)$ for some $\omega \in W(BC_n)$ and hence $v_{T_2} =c \phi_{\omega} v_{T_1}$ where $c \in \C$ is nonzero.\\ Furthermore, consider two arbitrary standard tableaux $T_1$ and $T_2$ in $\mathcal{T}$. Let $T$ be a standard tableaux of shape $\varphi$. There is a path $\varphi \to \varphi_1 \to \cdots \to Im(T_1)$ and hence $v_{T_1} = c_1 \phi_{\omega}v_{T_0}$.
	\end{proof}
	\subsection{Irreducible representation associated to a skew shape $\varphi_{n,p,\mu}^{\xi}$}
	Define a representation $L^{\varphi_{n,p,\mu}^{\xi}}$ of $H_n(1,p-q-\mu N)$ as follows. Let the underlying vector space be $span_{\C}\{w_T | T \in \mathcal{T}\}$. The action of $H_n(1, p-q-\mu N)$ is defined by
	\begin{align}
		&y_k w_T = (-cont_T(k)+\s) w_T,\\
		&s_i w_T =\frac{(1-cont_T(i)+cont_T(i+1))w_{s_i(T)}}{cont_T(i) -cont_T(i+1)}+\frac{1}{cont_T(i)-cont_T(i+1)}w_T,\\
		&\gamma_n w_T = \dfrac{(p-q-\mu N -2cont_T(n)) w_{\gamma_n(T)}}{2cont_T(n)}+(p-q-\mu N) \dfrac{1}{2cont_T(n)}w_T.
	\end{align}
	
	\begin{thm}
		The representation $F_{n,p,\mu}(V^{\xi})$ is isomorphic to $L^{\varphi_{n,p,\mu}^{\xi}}$.
	\end{thm}
	
	\begin{proof}
		Fix a $T \in \mathcal{T}$. Define a map $f:F_{n,p,\mu}(V^{\xi}) \to L^{\varphi_{n,p,\mu}^{\xi}}$ by $$f(v_T) = w_t$$
		and $f(\phi_i v_T) = (1-cont_T(i)+cont_T(i+1)) w_{s_i(T)}$.
	\end{proof}
	
	\section{Combinatorial description}
	In this section, we first discuss some properties of a representation of the degenerate affine Hecke algebra $H_n(1,\kappa_2)$ obtained via the Etingof-Freund-Ma functor, where $\kappa_2=p-q-\mu N$, and then we show that any representation satisfying these properties is the image of some irreducible polynomial representation of $GL_N$ via the Etingof-Freund-Ma functor.
	\subsection{Some facts of $F_{n,p,\mu}(V^{\xi})$}
	Let $F = F_{n,p,\mu}(V^{\xi})$ be a representation $H_n(1, p-q-\mu N)$ obtained through Etingof-Freund-Ma functor and $\zeta = (\zeta_1, \cdots, \zeta_n)$ be weight of $F$ such that $F_{\zeta} \neq 0$. For $i=1, \cdots, n$, if there is an increasing sequence $i=i_0 < i_1< \cdots <i_m \leq n$ such that $|\zeta_{i_k} -\zeta_{k+1}|=1$ for $k=0, \cdots, m-1$ and $\zeta_{i_m}= \pm \frac{\kappa_2}{2}$, then we call the coordinate $\zeta_i$ is fixed. It is easy to observe the following properties.
	\begin{property}
		For $i=1, \cdots, n$, if $|\zeta_{i}| \leq |\frac{\kappa_2}{2}|$, then $\zeta_i$ is fixed, i.e. there is an increasing sequence $i=i_0 < i_1< \cdots <i_m \leq n$ such that $|\zeta_{i_k} -\zeta_{k+1}|=1$ for $k=0, \cdots, m-1$ and $\zeta_{i_m}= \pm \frac{\kappa_2}{2}$.
	\end{property}
	\begin{property}
		The parameter $\kappa_2$ is an integer. If $\kappa_2$ is even, then all $\zeta_i$'s, for $i=1, \cdots, n$, are integers. If $\kappa_2$ is odd, then all $\zeta_i$'s, for $i=1, \cdots, n$ are half integers. 
	\end{property}
	Recall that the the cell $(p,a)$ in $\varphi^{\xi}_{n,p,\mu}$ gives the eigenvalue $\frac{\kappa_2}{2}$ and that the cell $(q,b)$ gives the eigenvalue $-\frac{\kappa_2}{2}$. Then Property 2 follows.\\
	
	In \cite{R2}, Ram explored the facts of weights of a semisimple affine Hecke algebra representation. Now let us explore facts of weights in the degenerate case. Let $L$ be an irreducible and $\Y$-semisimple representation of $H_n(1, \kappa_2)$ satisfying Property 1 and Property 2 above and $\zeta$ be a weight such that $L_{\zeta} \neq 0$. Then $\zeta$ satisfies the following property.
	\begin{prop}
		If there exist $1 \leq i<j \leq n$ such that $\zeta_i = \zeta_j$, then there exist $i<k_1<j$ such that $\zeta_{k_1} = \zeta_i+1$ and $i<k_2<j$ such that $\zeta_{k_2} = \zeta_i-1$.
	\end{prop}
	
	\begin{proof}
		Let $\zeta$ be a weight such that $L_{\zeta} \neq 0$. Suppose there exist $1 \leq i<j\leq n$ such that $\zeta_i =\zeta_j$ and there is no $i<k<j$ such that $\zeta_k=\zeta_i$. We proof by induction on $j-i$.\\ First, if $j-i=1$, then $\zeta_i =\zeta_{i+1}$ which contradicts theorem 7.9.\\
		Second, if $j-i=2$, by Theorem 7.9 and Lemma 7.7, it follows $\zeta_{i+1} = \zeta_i \pm1 = \zeta_{i+1} \pm 1$. Let $v$ be a nonzero weight vector of weight $\zeta$. Proposition 7.16 implies $\phi_i v = \phi_{i+1} v = 0$. Combining the definition of the intertwining operators, it follows $s_i v = \mp v$ and $s_{i+1}v = \pm v$ and hence $$\pm v = s_is_{i+1}s_iv = s_{i+1} s_i s_{i+1}v = \mp v,$$ which is a contradiction. \\
		So the base case of the induction is $j-i=3$. If $\zeta_i \neq  \zeta_{i+1} \pm 1$ or $\zeta_{j-1} \neq \zeta_j \pm 1$. Lemma 7.7 implies the existence of a weight satisfying the condition in the second case which is a contradiction. So it hold $|\zeta_i - \zeta_{i+1}| =1$ and $|\zeta_{j-1} -\zeta_j| =1$. If $\zeta_i = \zeta_{i+1} + 1$ and $\zeta_{j-1} = \zeta_j + 1$, then $k_1 = j-1$ and $k_2=i+1$. Similarly, if $\zeta_i=\zeta_{i+1} -1$ and  $\zeta_{j-1}= \zeta_j -1$, then $k_1=i+1$ and $k_2=j-1$. If $\zeta_i = \zeta_{i+1} \pm 1$ and $\zeta_{j-1} = \zeta_j \mp 1$, then $\zeta_{i+1} = \zeta_{i+2}$ which contradicts theorem 7.9.\\
		Suppose the statement is true for all $i-j<m$, consider the case $j-i = m$.\\
		Case1. If $|\zeta_i - \zeta_{i +1}| \neq 1$ or $|\zeta_{j-1}-\zeta_j| \neq 1$ and let $v$ be a nonzero weight vector of weight $\zeta$, then $\phi_iv$ or $\phi_{j-1}v$ will be a nonzero weight vector of weight $s_i \zeta$ or respectively $s_{j-1} \zeta$ with $s_i \zeta$ or $s_{j-1} \zeta$ has $\zeta_{i+1} =\zeta_j$ or respectively $\zeta_i = \zeta_{j-1}$. Then the $k_1$ and $k_2$ exist by the inductive hypothesis.\\
		Case 2. If $\zeta_i =\zeta_{i+1} \pm 1$ and $\zeta_{j-1} =\zeta_j \mp 1$, this implies $\zeta_{i+1} = \zeta_{j-1}$, the statement still holds by inductive hypothesis.\\
		Case 3. If $\zeta_i=\zeta_{i+1}+1$ and $\zeta_{j-1}=\zeta_j+1$, then $k_1=j-1$ and $k_2=i+1$.\\
		Case 4. If $\zeta_i=\zeta_{i+1} -1$ and $\zeta_{j-1} = \zeta_j -1$, then $k_1=i+1$ and $k_2=j-1$.
		
	\end{proof}
	Next let us explore another fact of $L$.
	\begin{lemma}
		Let $\zeta = [\zeta_1, \cdots, \zeta_n]$ be a weight of $L$ such that $L_{\zeta} \neq 0$ and $\zeta$ satisfies $\zeta_i > \frac{|\kappa_2|}{2}$ for $i=k, \cdots, n$. Then there is weight $$\zeta'=[\zeta_1, \cdots, \zeta_{k-1}, -\zeta_n, -\zeta_{n-1}, \cdots, -\zeta_{k+1}, -\zeta_k]$$ such that $L_{\zeta'} \neq 0$.
	\end{lemma}
	
	\begin{proof}
		Let $v$ be a nonzero weight vector of $\zeta$. Acting on $v$ by $$h= \phi_n (\phi_{n-1} \phi_n)\cdots(\phi_k \phi_{k+1} \cdots \phi_n),$$ the vector
		$h v \in L_{\zeta'}$ and $h v \neq 0$ by Lemma 7.7 and Lemma 7.8.
	\end{proof}
	\begin{definition}
		Let $\zeta = [\zeta_1, \cdots, \zeta_n]$ be a weight of $L$ such that $L_{\zeta} \neq 0$ and $\zeta$ satisfies the condition: if a coordinate $\zeta_{i} > 0 $, then $\zeta_i$ is fixed, i.e. there exists an increasing sequence $i=i_0<i_1 < \cdots < i_m \leq n$ such that $|\zeta_{i_{k}} -\zeta_{i_{k+1}}| = 1$ and $\zeta_{i_m}= \pm \frac{\kappa_2}{2}$. Then we call $\zeta$ is a minimal weight of $L$.
	\end{definition}
	\begin{prop}
		There exists at least one minimal weight $\zeta = [\zeta_1, \cdots, \zeta_n]$ of $L$ such that $L_{\zeta} \neq 0$.
	\end{prop}
	
	\begin{proof}
		Let $\zeta$ be any weight such that $L_{\zeta} \neq 0$. If $0< \zeta_{i} \leq \frac{|\kappa_2|}{2}$, then $\zeta_i$ is fixed since $L$ satisfies Property 1. So it suffices to consider the coordinate $\zeta_{i} > \frac{|\kappa_2|}{2}$. We want to show that starting with any weight $\zeta$ such that $L_{\zeta} \neq 0$, there is an algorithm to obtain a weight $\zeta'$ such that $L_{\zeta'} \neq 0$ and $\zeta'$ satisfies the condition: if a coordinate $\zeta'_i>0$, then $\zeta'_i$ is fixed.  \\
		Suppose $\{\zeta_{r_1}, \zeta_{r_2},\cdots, \zeta_{r_l}\}$ is the collection of all the coordinates such that $\zeta_{r_i}> \frac{|\kappa_2|}{2}$ and $\zeta_{r_i}$ is not fixed, for $1 \leq r_1 < r_2 < \cdots < r_l \leq n$. Let $v$ be a nonzero weight vector of weight $\zeta$. We start with the rightmost coordinate $\zeta_{r_l}$ in this collection. If $r_l \neq n$, there are only the following two cases.\\
		Case 1. There exists an increasing sequence $r_l+1=j_0<j_1 < \cdots < j_l \leq n$ such that $|\zeta_{j_{k+1}} -\zeta_{j_k}| = 1$ and $\zeta_{j_l}= \pm \frac{\kappa_2}{2}$. Then $|\zeta_{r_l}-\zeta_{r_l+1}| \neq 1$, otherwise there is an increasing sequence $r_l=j_{-1}< j_1<j_1 < \cdots < j_l \leq n$ such that $|\zeta_{j_{k+1}} -\zeta_{j_k}| = 1$ and $\zeta_{j_l}= \pm \frac{\kappa_2}{2}$. So $\phi_{r_l} v$ is a nonzero vector of weight $\zeta^{(1)}=s_{r_l}\zeta$.\\
		Case 2. If $\zeta_{r_l+1}<-\frac{|\kappa_2|}{2}$, then $|\zeta_{r_l}-\zeta_{r_l+1}|>1$ and hence $\phi_{r_l} v$ is a nonzero weight vector of weight $\zeta^{(1)}=s_{r_l}\zeta$.\\
		Then we consider $\zeta^{(1)}_{r_l+1}$ and we are in the same situation. Hence we repeat this process for $(n-r_l)$ times and obtain a nonzero weight vector $(\phi_{n-1} \cdots \phi_{r_l+1} \phi_{r_l})v$ of weight $$\zeta^{(n-r_l)}=(s_{n-1}\cdots s_{r_l+1}s_{r_l}) \zeta.$$\\
		Next, we deal with the second rightmost coordinate $\zeta_{r_{l-1}}=\zeta^{(n-r_l)}_{r_{l-1}}$ in the collection above and repeat the process above for $(n-1-r_{l-1})$ times. We obtain a nonzero weight vector $$(\phi_{n-2} \cdots \phi_{r_{l-1}+1} \phi_{r_{l-1}})(\phi_{n-1} \cdots \phi_{r_l+1} \phi_{r_l})v$$ of weight $$\zeta^{(2n-1-r_{l-1} -r_l)}=(s_{n-2}\cdots s_{r_{l-1}+1}s_{r_{l-1}})(s_{n-1} \cdots, s_{r_l+1}s_{r_l})\zeta.$$
		Next, we continue to deal with other coordinates in the collection in the order of $\zeta_{r_{l-2}}, \zeta_{r_{l-3}}, \cdots, \zeta_{r_1}$ and repeat the process for $(n-k-r_k)$ times for the coordinate $\zeta_{r_k}$ for $k=1, \cdots, l$. We obtain a nonzero weight vector 
		$$(\phi_{n-l} \cdots \phi_{r_1+1}\phi_{r_1})(\phi_{n-l+1} \cdots \phi_{r_2+1}\phi_{r_2}) \cdots (\phi_{n-1} \cdots \phi_{r_l+1}\phi_{r_l})v$$
		of weight 
		$$\zeta^{(ln-l(l-1)/2-r_1-r_2 \cdots-r_l)}=(s_{n-l} \cdots s_{r_1+1}s_{r_1})(s_{n-l+1} \cdots s_{r_2+1}s_{r_2})\cdots (s_{n-1} \cdots s_{r_l+1}s_{r_l})\zeta.$$
		The weight $\zeta^{(ln-l(l-1)/2-r_1-r_2 \cdots-r_l)}$ satisfies the condition that $$\zeta^{(ln-l(l-1)/2-r_1-r_2 \cdots-r_l)}_i > \frac{|\kappa_2|}{2}$$ for $i=n-l+1, \cdots, n$. Moreover, for $i=1,\cdots,n-l$, it follows either $$\zeta^{(ln-l(l-1)/2-r_1-r_2 \cdots-r_l)}_i<0$$ or the coordinate $\zeta^{(ln-l(l-1)/2-r_1-r_2 \cdots-r_l)}_i$ is fixed. Combining Lemma 10.2, there is a weight $$\zeta'=\gamma_n (s_{n-1} \gamma_n) \cdots (s_{n-l+1}\cdots s_{n-1}\gamma_n) \zeta^{(ln-l(l-1)/2-r_1-r_2 \cdots-r_l)}$$ such that $L_{\zeta^{(ln-l(l-1)/2-r_1-r_2 \cdots-r_l)}} \neq 0$ and satisfying the condition: if $$\zeta^{(ln-l(l-1)/2-r_1-r_2 \cdots-r_l)}_i >0,$$ then $\zeta^{(ln-l(l-1)/2-r_1-r_2 \cdots-r_l)}_i$ is fixed for any $i=1, \cdots, n$.
	\end{proof}
	\begin{rmk}
		Lemma 10.2 and Proposition 10.4 indicate that for any weight $\zeta$ such that $L_{\zeta} \neq 0$ and a nonzero $v \in L_{\zeta}$, there is a nonzero weight vector $\phi_{\omega} v \in L_{\zeta'}$ such that $\zeta'$ satisfies the condition in Proposition 10.4.
	\end{rmk}
	\begin{eg}
		Let $\zeta=[-2,2,\textcolor{blue}{4},\textcolor{blue}{5},\textcolor{blue}{6},-3,1]$ and $v \in L$ is a nonzero weight vector of weight $\zeta$. Locate the collection of all the coordinates which are positive and not fixed: $\{\zeta_3=4, \zeta_4=5, \zeta_5=6\}$, i.e. there are three coordinates with $r_1=3, r_2=4$ and $r_3=5$. We deal with these coordinates from right to left. First, we deal with the rightmost coordinate $\zeta_5=6$ in this collection and apply the process for $(n-r_3)=2$ times. We obtain a nonzero weight vector $$(\phi_{n-1} \cdots\phi_{r_3})v=(\phi_6 \phi_5)v$$ of weight $$\zeta^{(n-r_3)}=\zeta^{(2)}=(s_6s_5)\zeta=[-2,2,\textcolor{blue}{4},\textcolor{blue}{5},-3,1,\textcolor{blue}{6}].$$
		Then we work on with the coordinate $\zeta_4=\zeta^{(2)}_4=5$ and apply the process for $(n-1-r_2)$ times. We obtain a nonzero weight vector $$(\phi_{n-2}\cdots \phi_{r_2})(\phi_{n-1} \cdots \phi_{r_3})v=(\phi_5\phi_4)(\phi_6\phi_5)v$$
		of weight $$\zeta^{(2n-1-r_1-r_2)}=\zeta^{(4)}=(s_5s_4)\zeta^{(2)}=(s_5s_4)(s_6s_5)\zeta=[-2,2,\textcolor{blue}{4},-3,1,\textcolor{blue}{5},\textcolor{blue}{6}].$$
		Finally, we deal with the coordinate $\zeta_3=\zeta^{(4)}_3=4$ and apply the process for $n-2-r_3$ times. We obtain a nonzero weight vector $$(\phi_{n-3}\cdots\phi_{r_1})(\phi_{n-2} \cdots \phi_{r_2})(\phi_{n-1}\cdots \phi_{r_3})v=(\phi_4\phi_3)(\phi_5\phi_6)(\phi_6\phi_5)v$$ of weight 
		$$\zeta^{(3n-3-r_1-r_2-r_3)}=\zeta^{(6)}=(s_4s_3)\zeta^{(4)}=[-2,2,-3,1,\textcolor{blue}{4},\textcolor{blue}{5},\textcolor{blue}{6}].$$
		Now the weight $\zeta^{(6)}$ satisfies the condition in Lemma 10.2 with $\zeta^{(6)}_i>\frac{|\kappa_2|}{2}$ for $i=5,6,7$. Moreover, for each $i=1,\cdots, 4$, either $\zeta^{(6)}_i<0$ or that $\zeta^{(6)}_i$ is fixed.\\
		Applying Lemma 10.2, we obtain a nonzero weight vector $$\phi_7(\phi_6\phi_7)(\phi_5\phi_6\phi_7)(\phi_4\phi_3)(\phi_5\phi_6)(\phi_6\phi_5)v$$ of weight $$\zeta'=\gamma_7(s_6 \gamma_7)(s_5s_6\gamma_7) \zeta^{(6)}=[-2,2,-3,1,\textcolor{red}{-6},\textcolor{red}{-5},\textcolor{red}{-4}].$$
	\end{eg}
	\begin{eg}
		Let $\zeta = [0,4,-1,6,-2,5,1]$ and $v \in L$ is a nonzero weight vector of weight $\zeta$. There are three coordinates $\zeta_2 =4$, $\zeta_4 =6$ and $\zeta_6=5$ satisfying the condition that $i=2,4,6$, there is no increasing sequence $i<i_1 < \cdots < i_l \leq n$ such that $|\zeta_{i_{k+1}} -\zeta_{i_k}| = 1$ and $|\zeta_{i_l}|= \pm \frac{\kappa_2}{2}$. Starting with the coordinate with maximal index $i=6$ and applying the intertwining operators, it follows
		\begin{center}
			\begin{tikzpicture}[scale=0.4]
				\draw (0,0) node {\tiny{$[0,\textcolor{blue}{4},-1,\textcolor{blue}{6},-2,\textcolor{blue}{5},1]$}};
				\draw (10,0) node {\tiny{$[0,\textcolor{blue}{4},-1,\textcolor{blue}{6},-2,1,\textcolor{blue}{5}]$}};
				\draw (20,0) node {\tiny{$[0,\textcolor{blue}{4},-1,-2,1,\textcolor{blue}{6},\textcolor{blue}{5}]$}};
				\draw (30,0) node {\tiny{$[0,-1,-2,1,\textcolor{blue}{4},\textcolor{blue}{6},\textcolor{blue}{5}]$}};
				\draw [->] (3.5,0)--(6.5,0);
				\draw [->] (13.5,0)--(16.5,0);
				\draw [->] (23.5,0)--(26.5,0);
				\draw (5,0.5) node {\tiny{$s_6$}};
				\draw (15,0.5) node {\tiny{$s_5 s_4$}};
				\draw (25,0.5) node {\tiny{$s_4 s_3 s_2$}};
			\end{tikzpicture}
		\end{center}
		and by Lemma 10.2
		\begin{center}
			\begin{tikzpicture}[scale=0.4]
				\draw (0,0) node {\tiny{$[0,-1,-2,1,\textcolor{blue}{4},\textcolor{blue}{6},\textcolor{blue}{5}]$}};
				\draw (10,0) node {\tiny{$[0,-1,-2,1,\textcolor{red}{-5},\textcolor{blue}{4},\textcolor{blue}{6}]$}};
				\draw (20.5,0) node {\tiny{$[0,-1,-2,1,\textcolor{red}{-5},\textcolor{red}{-6},\textcolor{blue}{4}]$}};
				\draw (31,0) node {\tiny{$[0,-1,-2,1,\textcolor{red}{-5},\textcolor{red}{-6},\textcolor{red}{-4}]$}};
				\draw [->] (3.5,0)--(6.5,0);
				\draw [->] (13.5,0)--(16.5,0);
				\draw [->] (24.5,0)--(27,0);
				\draw (5,0.5) node {\tiny{$s_5 s_6 \gamma_7$}};
				\draw (15,0.5) node {\tiny{$s_6 \gamma_7$}};
				\draw (26,0.5) node {\tiny{$\gamma_7$}};
			\end{tikzpicture}
		\end{center}
		Let $\zeta' = [0,-1,-2,1,\textcolor{red}{-5},\textcolor{red}{-6},\textcolor{red}{-4}]$. Then there is a nonzero weight vector $$\phi_7 (\phi_6 \phi_7)(\phi_5 \phi_6 \phi_7)(\phi_4 \phi_3 \phi_2)(\phi_5 \phi_4)\phi_6 v \in L_{\zeta'}.$$
	\end{eg}
	\begin{rmk}
		For any minimal weight $\zeta$ of $F=F_{n,p,\mu}(V^{\xi})$ such that $F_{\zeta} \neq 0$, let $T_{\zeta}$ be the corresponding standard tableau. Then $Im(T_{\zeta})$ is the minimal shape $\varphi^{\xi}_{n,p,\mu}$ of $F_{n,p,\mu}(V^{\xi})$.
	\end{rmk}
	Before introducing the third property of $F_{n,p,\mu}(V^{\xi})$, we need the following definition and lemma.
	\begin{definition}
		Let $\zeta=[\zeta_1, \cdots, \zeta_n]$ be a weight. If a coordinate $\zeta_i$, $i=1,2,\cdots,n$, satisfies the condition that there is no $i<k \leq n$ such that $\zeta_k=\zeta_i \pm 1$, then the coordinate $\zeta_i$ is a corner of $\zeta$.
	\end{definition}
	\begin{rmk}
		Let $\zeta=[\zeta_1,\cdots, \zeta_n]$ and $T_{\zeta}$ is the corresponding standard tableau. For $i=1,\cdots, n$, $\zeta_i$ is a corner of $\zeta$ if and only if $T(i)$ is a southeastern corner of $Im(T_{\zeta})$.
	\end{rmk}
	\begin{eg}
		Let $\zeta=[0,-1,\textcolor{brown}{-2},\textcolor{brown}{1},-5,\textcolor{brown}{-6},\textcolor{brown}{-4}]$. Then $\zeta_3=-2$, $\zeta_4=1$, $\zeta_6=-6$ and $\zeta_7=-4$ are corners of $\zeta$. The corresponding standard tableau $T_{\zeta}$ has southeastern corners $3,4,6$ and $7$.
		
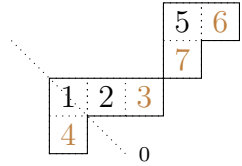
\begin{center}
			\begin{tikzpicture}[scale=0.5]
				\draw (2,1)--(2,-1)--(3,-1)--(3,0)--(5,0)--(5,1)--(6,1)--(6,2)--(7,2)--(7,3)--(5,3)--(5,1)--(2,1);
				\draw [dotted] (2,1) grid (3,-1);
				\draw [dotted] (3,1) grid (5,0);
				\draw [dotted] (5,3) grid (6,1);
				\draw [dotted] (6,3) grid (7,2);
				\draw  (2.5,0.5) node {$1$};
				\draw [dotted] (1,2)--(4,-1);
				\draw (4.5,-1) node {\tiny{$0$}};
				\draw[brown]  (2.5,-0.5) node {$4$};
				\draw  (3.5,0.5) node {$2$};
				\draw[brown]  (4.5,0.5) node {$3$};
				\draw [brown] (5.5,1.5) node {$7$};
				\draw  (5.5,2.5) node {$5$};
				\draw [brown] (6.5,2.5) node {$6$};
			\end{tikzpicture}
		\end{center}
	\end{eg}
	\begin{lemma}
		Let $L$ be an irreducible and $\Y$-semisimple representation of $H_n(1, \kappa_2)$ satisfying Property 1. Let $\zeta$ be a minimal weight of $L$ such that $L_{\zeta} \neq 0$. For $i=1, \cdots, n$, if the coordinate $\zeta_i$ is a corner of $\zeta$, then $\zeta_i= \pm \frac{\kappa_2}{2}$ or $\zeta < -\frac{|\kappa_2|}{2}$.
	\end{lemma}
	\begin{proof}
		First, since $L$ satisfies Property 1, if $|\zeta_i|<\frac{|\kappa_2|}{2}$, then $\zeta_i$ is fixed, i.e. there is an increasing sequence $i=i_0 < i_1< \cdots <i_m \leq n$ such that $|\zeta_{i_k} -\zeta_{k+1}|=1$ for $k=0, \cdots, m-1$ and $\zeta_{i_m}= \pm \frac{\kappa_2}{2}$. This contradicts the fact that $\zeta_i$ is a corner of $\zeta$.\\
		Second, suppose $\zeta_i> \frac{|\kappa_2|}{2}$. Since $\zeta$ is a minimal weight, $\zeta_i$ if fixed, which again contradicts the fact that $\zeta_i$ is a corner.
	\end{proof}
	Now we introduce the third property of $F_{n,p,\mu}(V^{\xi})$.
	\begin{property}
		Let $\zeta$ be a minimal weight such that $F_{\zeta} \neq 0$. If $\zeta_k$ is the rightmost coordinate equal to $\frac{|\kappa_2|}{2}$ and $\zeta_r$ is the rightmost coordinate equal to $-\frac{|\kappa_2|}{2}$, then at least one of these two coordinates is not a corner.
	\end{property}
	\begin{proof}
		Let $T_{\zeta}$ be the corresponding standard tableau of weight $\zeta$. Since $\zeta$ is a minimal weight, the shape $Im(T_{\zeta})$ is the minimal shape $\varphi = \varphi_{n,p,\mu}^{\xi}$. So it suffices to show that it is impossible for  $T_{\zeta}$ to have $T_{\zeta}(k)$ and $T_{\zeta}(r)$ at southeastern corners simultaneously, equivalently, it is impossible for $\varphi$ to have a southeast corner at eigenvalue $\frac{\kappa_2}{2}$ and a southeastern corner at eigenvalue $-\frac{\kappa_2}{2}$ simultaneously. Let $p \leq q$, $$a=\mu q+ \dfrac{|\xi|+n}{N}$$ and $$b=-\mu p+ \dfrac{|\xi|+n }{N}.$$ Suppose $\varphi$ simultaneously has a southeast corner at eigenvalue $\frac{\kappa_2}{2}$ and a southeastern corner at eigenvalue $-\frac{\kappa_2}{2}$, then $p<q$ and $a>b$ follow. In this case, $\varphi$ has cell $(p,a)$ at eigenvalue $-\frac{|\kappa_2|}{2}$ and cell $(q,b)$ at eigenvalue $\frac{|\kappa_2|}{2}$. Furthermore, the fact that cell $(p,a)$ is a southeastern corner indicates $\xi^{(2)}_1=\xi_{q+1}=b$. The fact that cell $(q,b) \in \varphi$ indicates $\xi^{(1)}_q=\xi_q < b$. This contradicts $\xi \in P^+_{\geq 0}$.
	\end{proof}
	
	\subsection{Combinatorial description of irreducible representations in $\co$}
	Let $\co(H_n(1,\kappa_2))$ be collection of $\Y$-semisimple representations of $H_n(1, \kappa_2)$ satisfying Properties 1-3. In this subsection, we show that any irreducible representation in $\co(H_n(1, \kappa_2))$ is isomorphic to the image $F_{n, p,\mu}(V^{\xi})$ for a tuple of $n,p,\mu$ and some $\xi \in P^+_{\geq 0}$.
	
	Let $L \in \co(H_n(1, \kappa_2))$ be irreducible and $\zeta$ be a minimal weight such that $L_{\zeta} \neq 0$. Recall, if $\zeta_i \geq 0$, then there is an increasing sequence $k_1< \cdots < k_m$ such that $\zeta_{k_{i+1}} = \zeta_{k_i} \pm 1$ and $\zeta_{k_m} = \pm \frac{\kappa_2}{2}$. The weight $\zeta$ gives a standard tableau $T_{\zeta}$ such that $\zeta_k = -cont_{T_{\zeta}}(k)+s$ for some fixed number $s$ where $s-\kappa_2$ is an integer. Let $Im(T_{\zeta}) = \nu / \beta$ such that $\beta_1<\nu_1$ and $\beta_{\ell(\nu)}<\nu_{\ell(\nu)}$. Let us explore in different cases depending on corners. According to Lemma 10.11, if $\zeta_i$ is a corner of $\zeta$, for some $i=1,\cdots,n$, then $\zeta_i=\pm \frac{\kappa_2}{2}$ or $\zeta_i < -\frac{|\kappa_2|}{2}$. For any minimal $\zeta$, there is at least one corner of $\zeta$. Let the coordinate $\zeta_{r_1}$ be the corner of $\zeta$ such that $\ci(r_1)$ is the maximal of $\{\ci(i)|\zeta_{i} \text{ is corner of } \zeta\}$ and the coordinate $\zeta_{r_2}$ is the corner of $\zeta$ such that $\ci(r_2)$ is the second largest number in $\{\ci(i)|\zeta_{i} \text{ is corner of } \zeta\}$ if $\zeta_{r_2}$ exists. It is obvious $\zeta_{r_2}<\zeta_{r_1}$. There are the following cases. If $\zeta_{r_1}=\frac{|\kappa_2|}{2}$, then $\zeta_{r_2}<-\frac{|\kappa_2|}{2}$ or $\zeta_{r_2}$ doesn't exist. By Lemma 10.11, if $\zeta_{r_1}=\frac{|\kappa_2|}{2}$ and $\zeta_{r_2}=-\frac{|\kappa_2|}{2}$, then $\zeta$ violates Property 3. When $\zeta_{r_1}=-\frac{|\kappa_2|}{2}$, $\zeta_{r_2}<-\frac{|\kappa_2|}{2}$ or there is no $\zeta_{r_2}$. When $\zeta_{r_1}<-\frac{|\kappa_2|}{2}$, $\zeta_{r_2}<-\frac{|\kappa_2|}{2}$ or $\zeta_{r_2}$ doesn't exist. So let us discuss in five cases.\\
	\textbf{Case 1.} The corner $\zeta_{r_1}=\frac{|\kappa_2|}{2}$ and the corner $\zeta_{r_2}<-\frac{|\kappa_2|}{2}$.\\
	Denote $T_{\zeta}(r_1)=(i_1,j_1)$ and $T_{\zeta}(r_2)=(i_2,\nu_{i_2})$. Let $j_2=i_2+s+\frac{|\kappa_2|}{2}$. In this case, set two rectangles $$(a^p)=((\nu_1-j_1)^{i_2})$$ and $$(b^q)=((\nu_1-j_2)^{i_1}).$$
	\begin{claim}
		Following the setting above, the number $\nu_{i_2} - j_1-j_2 \geq 0$.
	\end{claim}
	\begin{proof}
		Since $\zeta_{r_2}$ is a corner, there exists a weight $\tilde{\zeta}$ such that $L_{\tilde{\zeta}} \neq 0$, $Im(T_{\tilde{\zeta}})=Im(T_{\zeta})$ and $T_{\tilde{\zeta}}(n)=(i_2, \nu_{i_2})$, where $T_{\tilde{\zeta}}$ denotes the standard tableau given by the weight $\tilde{\zeta}$. Let $v$ be a nonzero weight vector of weight $\tilde{\zeta}$. Since $\tilde{\zeta}_n \neq \pm \frac{\kappa_2}{2}$, it follows that $\phi_n v$ is a nonzero weight vector of weight $\gamma_n \tilde{\zeta}$. Moreover, the standard tableau $T_{\gamma_n \tilde{\zeta}}$ given by $\gamma_n \tilde{\zeta_n}$ satisfies that $$Im(T_{\gamma_n \tilde{\zeta}})) = Im(T_{\zeta}) \setminus \{(i_2, \nu_{i_2})\} \cup \{(i_1+1, j_1+j_2-\nu_{i_2}+1)\}$$
		since $(\gamma_n \tilde{\zeta})_n=-\tilde{\zeta}_n$. Lemma 10.1 implies $T_{\gamma_n \tilde{\zeta}}$ is a standard tableau and hence $Im(T_{\gamma_n \tilde{\zeta}})$ is a skew shape. This fact forces $j_1+j_2-\nu_{i_2}+1 \leq 1$ and thus $$\nu_{i_2}-j_1-j_2 \geq 0.$$
	\end{proof}
	Set $\xi^{(1)}=(\xi^{(1)}_1,\cdots, \xi^{(1)}_{i_1})$ with $$\xi^{(1)}_k= \beta_k + \nu_1-j_1-j_2,$$
	for $k=1, \cdots, i_1$ and $\xi^{(2)}=(\xi^{(2)}_1, \cdots, \xi^{(2)}_{i_2})$ with 
	$$\xi^{(2)}_k=\nu_1-\nu_{i_2-k+1},$$
	for $k=1, \cdots, i_2$. Furthermore, set $\xi=(\xi_1, \cdots, \xi_{i_1+i_2})$ with $$\xi_k=\xi^{(1)}_k,$$
	for $k=1, \cdots, i_1$ and
	$$\xi_{k}=\xi^{(2)}_{k-i_1},$$
	for $k=i_1+1, \cdots, i_1+i_2$.
	\begin{rmk}
		Claim 10.13 implies the following two facts.
		\begin{enumerate}
			\item It follows $\nu_1-j_1-j_2 \geq 0$.
			\item The inequality $\nu_1-\nu_{i_2}=\xi^{(2)}_1 \leq \xi^{(1)}_{i_1}=\beta_{i_1}+\nu_1-j_1-j_2$ holds and hence $\xi$ is a well-defined Young diagram.
		\end{enumerate}
	\end{rmk}
	\begin{eg}
		Continue Example 10.7. An irreducible representation $L$ in $\co(H_7(1,-2))$, we start with a minimal weight $\zeta=[0,-1,\textcolor{brown}{-2},\textcolor{brown}{1},-5,\textcolor{brown}{-6},\textcolor{brown}{-4}]$ and the standard tableau of $\zeta$. The corners of $\zeta$ are $\zeta_3=-2$, $\zeta_4=1$,$\zeta_6=-6$ and $\zeta_7=-4$. Furthermore, $\zeta_{r_1}=\zeta_4=1$ and $\zeta_{r_2}=\zeta_3=-2$
		
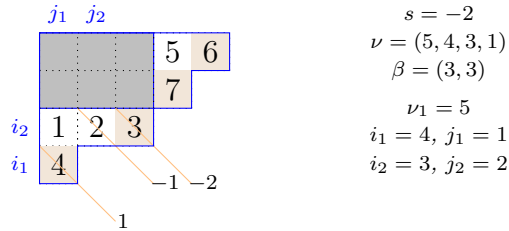
\begin{center}
			\begin{tikzpicture}
				\begin{scope}[scale=0.5, shift={(-8,4)}]
					\draw [blue] (2,1)--(5,1)--(5,3);
					\draw [dotted] (2,1) grid (3,-1);
					\draw [dotted] (3,1) grid (5,0);
					\draw [dotted] (5,3) grid (6,1);
					\draw [dotted] (6,3) grid (7,2);
					\draw [draw=none, fill=brown!20] (2,0) rectangle (3,-1);
					\draw [draw=none, fill=brown!20] (4,1) rectangle (5,0);
					\draw [draw=none, fill=brown!20] (5,2) rectangle (6,1);
					\draw [draw=none, fill=brown!20] (6,3) rectangle (7,2);
					\draw  (2.5,0.5) node {$1$};
					\draw  (2.5,-0.5) node {$4$};
					\draw [blue]  (1.5,-0.5) node {\tiny{$i_1$}};
					\draw [blue]  (2.5,3.5) node {\tiny{$j_1$}};
					\draw [blue]  (1.5,0.5) node {\tiny{$i_2$}};
					\draw [blue]  (3.5,3.5) node {\tiny{$j_2$}};
					\draw  (3.5,0.5) node {$2$};
					\draw  (4.5,0.5) node {$3$};
					\draw  (5.5,1.5) node {$7$};
					\draw  (5.5,2.5) node {$5$};
					\draw  (6.5,2.5) node {$6$};
					\draw [draw=blue, fill=gray!50] (2,3) rectangle (5,1);
					\draw [orange!50] (3,1)--(5,-1);
					\draw (5.3,-1) node {\tiny{$-1$}};
					\draw [orange!50] (2,0)--(4,-2);
					\draw (4.2,-2) node {\tiny{$1$}};
					\draw [dotted] (2,3) grid (5,1);
					\draw [orange!50] (4,1)--(6,-1);
					\draw (6.3,-1) node {\tiny{$-2$}};
					\draw [blue] (2,3)--(2,-1)--(3,-1)--(3,0)--(5,0)--(5,1)--(6,1)--(6,2)--(7,2)--(7,3)--(2,3);
				\end{scope}
				\begin{scope}[scale=0.5, shift={(2,6)}]
					\draw [scale=0.5] (5,3) node {\tiny{$s=-2$}};
					\draw [scale=0.5] (5,1.5) node {\tiny{$\nu=(5,4,3,1)$}};
					\draw [scale=0.5] (5,0) node {\tiny{$\beta=(3,3)$}};
					\draw [scale=0.5] (5,-2) node {\tiny{$\nu_1=5$}};
					\draw [scale=0.5] (5,-3.5) node {\tiny{$i_1=4$, $j_1=1$}};
					\draw [scale=0.5] (5,-5) node {\tiny{$i_2=3$, $j_2=2$}};
				\end{scope}
			\end{tikzpicture}
		\end{center}
		Place the southeastern corner of $((\nu_1-j_2)^{i_1})$ at the cell $(i_1,j_1)$ and northeastern corner of ${(\nu_1-j_1)^{i_2}}$ at the cell $(1,\nu_1)$. The gray part on the left forms $\xi^{(1)}$ and the gray part on the right forms \rotatebox[origin=c]{180}{$\xi^{(2)}$}.
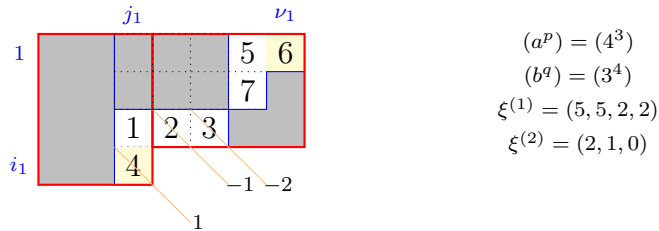
		\begin{center}
			\begin{tikzpicture}
				\begin{scope}[scale=0.5, shift={(-10,-4)}]
					\draw [dotted] (2,1) grid (3,-1);
					\draw [dotted] (3,1) grid (5,0);
					\draw [dotted] (5,3) grid (6,1);
					\draw [dotted] (6,3) grid (7,2);
					\draw [draw=none, fill=yellow!20] (2,0) rectangle (3,-1);
					\draw [draw=none, fill=yellow!20] (6,3) rectangle (7,2);
					\draw [draw=none, fill=gray!50] (5,1) rectangle (6,0);
					\draw [draw=none, fill=gray!50] (6,2) rectangle (7,0);
					\draw [draw=none, fill=gray!50] (0,3) rectangle (2,-1);
					\draw [draw=blue, fill=gray!50] (2,3) rectangle (5,1);
					\draw [blue] (2,3)--(2,-1)--(3,-1)--(3,0)--(5,0)--(5,1)--(6,1)--(6,2)--(7,2)--(7,3)--(2,3);
					\draw [red,thick] (0,3) rectangle (3,-1);
					\draw [red,thick] (3,3) rectangle (7,0);
					\draw  (2.5,0.5) node {$1$};
					\draw  (2.5,-0.5) node {$4$};
					\draw [blue] (-0.5,-0.5) node {\tiny{$i_1$}};
					\draw [blue] (2.5,3.5) node {\tiny{$j_1$}};
					\draw [blue] (-0.5,2.5) node {\tiny{$1$}};
					\draw [blue] (6.5,3.5) node {\tiny{$\nu_1$}};
					\draw  (3.5,0.5) node {$2$};
					\draw  (4.5,0.5) node {$3$};
					\draw  (5.5,1.5) node {$7$};
					\draw  (5.5,2.5) node {$5$};
					\draw  (6.5,2.5) node {$6$};
					\draw [orange!50] (3,1)--(5,-1);
					\draw (5.3,-1) node {\tiny{$-1$}};
					\draw [orange!50] (2,0)--(4,-2);
					\draw (4.2,-2) node {\tiny{$1$}};
					\draw [dotted] (2,3) grid (5,1);
					\draw [orange!50] (4,1)--(6,-1);
					\draw (6.3,-1) node {\tiny{$-2$}};
				\end{scope}
				\begin{scope}[scale=0.3, shift={(2,-5)}]
					\draw (5,3) node {\tiny{$(a^p)=(4^3)$}};
					\draw (5,1.5) node {\tiny{$(b^q)=(3^4)$}};
					\draw (5,0) node {\tiny{$\xi^{(1)}=(5,5,2,2)$}};
					\draw (5,-1.5) node {\tiny{$\xi^{(2)}=(2,1,0)$}};
				\end{scope}
			\end{tikzpicture}
		\end{center}
		Furthermore, we obtain other parameters of Etingof-Freund-Ma functor as $N=p+q=7$, $p=3$ and $\mu=\frac{a-b}{N}=\frac{1}{7}$.
		
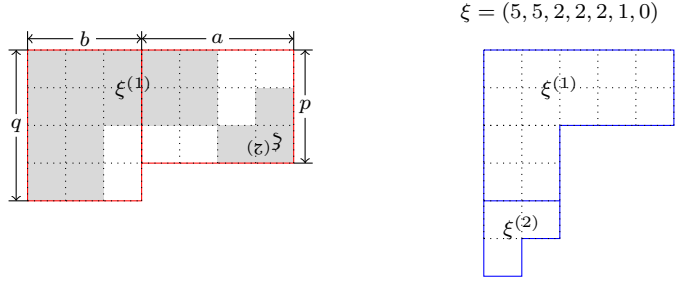
\begin{center}
			\begin{tikzpicture}
				\begin{scope}[scale=0.5, shift={(2,-12)}]
					\draw [blue] (0,3)--(0,-1)--(2,-1)--(2,1)--(5,1)--(5,3)--(0,3);
					\draw [blue] (0,-1)--(0,-3)--(1,-3)--(1,-2)--(2,-2)--(2,-1);
					\draw (2,4) node {\tiny{$\xi=(5,5,2,2,2,1,0)$}};
					\draw (2,2) node {\tiny{$\xi^{(1)}$}};
					\draw (1,-1.7) node {\tiny{$\xi^{(2)}$}};
					\draw [dotted] (0,3) grid (5,1);
					\draw [dotted] (0,1) grid (2,-2);
				\end{scope}
				\begin{scope}[scale=0.5, shift={(-10,-12)}]
					\draw [draw=none, fill=gray!30] (0,3) rectangle (5,1);
					\draw [draw=none, fill=gray!30] (0,1) rectangle (2,-1);
					\draw (2.8,2) node {\tiny{$\xi^{(1)}$}};
					\draw [draw=none, fill=gray!30] (5,1) rectangle (7,0);
					\draw [draw=none, fill=gray!30] (6,2) rectangle (7,1);
					\draw (6.2,0.5) node {\tiny{\rotatebox{180}{$\xi^{(2)}$}}};
					\draw [red] (0,3) rectangle (3,-1);
					\draw [red] (3,3) rectangle (7,0);
					\draw [dotted] (0,3) grid (3,-1);
					\draw [dotted] (3,3) grid (7,0);
					\draw (1.5,3.3) node {\tiny{$b$}};
					\draw (5,3.3) node {\tiny{$a$}};
					\draw (7,3)--(7,3.5);
					\draw (0,3)--(0,3.5);
					\draw (3,3)--(3,3.5);
					\draw[->] (1.2,3.3)--(0,3.3);
					\draw[->] (1.8,3.3)--(3,3.3);
					\draw[->] (4.7,3.3)--(3,3.3);
					\draw[->] (5.3,3.3)--(7,3.3);
					\draw (-0.3,1) node {\tiny{$q$}};
					\draw (7.3,1.5) node {\tiny{$p$}};
					\draw (0,3)--(-0.5,3);
					\draw (0,-1)--(-0.5,-1);
					\draw (7,0)--(7.5,0);
					\draw (7,3)--(7.5,3);
					\draw[->] (-0.3,1.3)--(-0.3,3);
					\draw[->] (-0.3,0.7)--(-0.3,-1);
					\draw[->] (7.3,1.8)--(7.3,3);
					\draw[->] (7.3,1.2)--(7.3,0);
				\end{scope}
			\end{tikzpicture}
		\end{center}
	\end{eg}
	
	\textbf{Case 2.} The corner $\zeta_{r_1}=-\frac{|\kappa_2|}{2}$ and the corner $\zeta_{r_2}<-\frac{|\kappa_2|}{2}$.\\
	Denote $T_{\zeta}(r_1)=(i_1,j_1)$ and $T_{\zeta}(r_2)=(i_2,\nu_{i_2})$. Let $j_2=i_2+s-\frac{|\kappa_2|}{2}$. In this case, set two rectangles $$(a^p)=((\nu_1-j_1)^{i_2})$$ and $$(b^q)=((\nu_1-j_2)^{i_1}).$$
	We have a similar claim to that in Case 1.
	\begin{claim}
		Following the setting above, the number $\nu_{i_2} - j_1-j_2 \geq 0$.
	\end{claim}
	The proof is the same with that in Case 1.\\
	Similarly, let $\xi^{(1)}=(\xi^{(1)}_1,\cdots, \xi^{(1)}_{i_1})$ with $$\xi^{(1)}_k= \beta_k + \nu_1-j_1-j_2,$$
	for $k=1, \cdots, i_1$ and $\xi^{(2)}=(\xi^{(2)}_1, \cdots, \xi^{(2)}_{i_2})$ with 
	$$\xi^{(2)}_k=\nu_1-\nu_{i_2-k+1},$$
	for $k=1, \cdots, i_2$. Furthermore, set $\xi=(\xi_1, \cdots, \xi_{i_1+i_2})$ with $$\xi_k=\xi^{(1)}_k,$$
	for $k=1, \cdots, i_1$ and
	$$\xi_{k}=\xi^{(2)}_{k-i_1},$$
	for $k=i_1+1, \cdots, i_1+i_2$.\\
	\begin{eg}
		Let $L$ be an irreducible representation in $\co(H_7(1,-2))$ with a minimal weight $\zeta=[-1,1,0,-2,\textcolor{brown}{-1},\textcolor{brown}{-5},\textcolor{brown}{-3}]$ and the standard tableau of $\zeta$. The corners of $\zeta$ are $\zeta_4=-6$,$\zeta_6=-4$ and $\zeta_7=-2$. Furthermore, $\zeta_{r_1}=\zeta_5=-1$ and $\zeta_{r_2}=\zeta_7=-3$
		
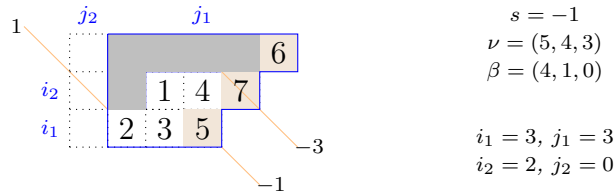
\begin{center}
			\begin{tikzpicture}
				\begin{scope}[scale=0.5, shift={(-8,4)}]
					\draw [blue] (6,3)--(6,2)--(3,2)--(3,1)--(2,1);
					\draw [dotted] (2,1) grid (5,0);
					\draw [dotted] (3,2) grid (6,1);
					\draw [draw=none, fill=brown!20] (5,2) rectangle (6,1);
					\draw [draw=none, fill=brown!20] (6,3) rectangle (7,2);
					\draw [orange!50] (5,2)--(7,0);
					\draw (7.3,0) node {\tiny{$-3$}};
					\draw [orange!50] (4,1)--(6,-1);
					\draw (6.3,-1) node {\tiny{$-1$}};
					\draw [orange!50] (-0.2,3.2)--(2,1);
					\draw (-0.4,3.2) node {\tiny{$1$}};
					\draw [draw=none, fill=brown!20] (4,1) rectangle (5,0);
					\draw  (3.5,1.5) node {$1$};
					\draw  (4.5,1.5) node {$4$};
					\draw  (2.5,0.5) node {$2$};
					\draw  (3.5,0.5) node {$3$};
					\draw  (5.5,1.5) node {$7$};
					\draw  (4.5,0.5) node {$5$};
					\draw [blue] (0.5,0.5) node {\tiny{$i_1$}};
					\draw [blue] (4.5,3.5) node {\tiny{$j_1$}};
					\draw [dotted] (1,3) grid (2,0);
					\draw [blue] (0.5,1.5) node {\tiny{$i_2$}};
					\draw [blue] (1.5,3.5) node {\tiny{$j_2$}};
					\draw  (6.5,2.5) node {$6$};
					\draw [draw=none, fill=gray!50] (2,3) rectangle (6,2);
					\draw [draw=none, fill=gray!50] (2,2) rectangle (3,1);
					\draw [blue] (2,3)--(2,0)--(5,0)--(5,1)--(6,1)--(6,2)--(7,2)--(7,3)--(2,3);
				\end{scope}
				\begin{scope}[scale=0.5, shift={(3,6)}]
					\draw [scale=0.5] (5,3) node {\tiny{$s=-1$}};
					\draw [scale=0.5] (5,1.5) node {\tiny{$\nu=(5,4,3)$}};
					\draw [scale=0.5] (5,0) node {\tiny{$\beta=(4,1,0)$}};
					\draw [scale=0.5] (5,-3.5) node {\tiny{$i_1=3$, $j_1=3$}};
					\draw [scale=0.5] (5,-5) node {\tiny{$i_2=2$, $j_2=0$}};
				\end{scope}
			\end{tikzpicture}
		\end{center}
		Place the southeastern corner of $(b^q)$ at the cell $(i_1,j_1)$ and northeastern corner of $(a^p)$ at the cell $(1,\nu_1)$. The gray part on the left forms $\xi^{(1)}$ and the gray part on the right forms \rotatebox[origin=c]{180}{$\xi^{(2)}$}.
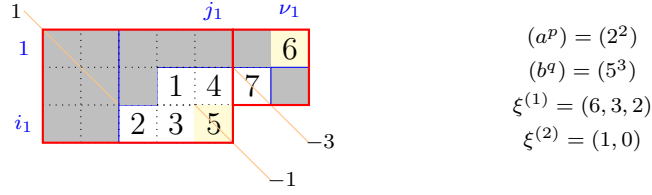
		\begin{center}
			\begin{tikzpicture}
				\begin{scope}[scale=0.5, shift={(-10,-4)}]
					\draw [blue] (6,3)--(6,2)--(3,2)--(3,1)--(2,1);
					\draw [dotted] (2,1) grid (5,0);
					\draw [dotted] (3,2) grid (6,1);
					\draw [draw=none, fill=yellow!20] (6,3) rectangle (7,2);
					\draw [draw=none, fill=yellow!20] (4,1) rectangle (5,0);
					\draw  (3.5,1.5) node {$1$};
					\draw  (4.5,1.5) node {$4$};
					\draw  (2.5,0.5) node {$2$};
					\draw  (3.5,0.5) node {$3$};
					\draw  (5.5,1.5) node {$7$};
					\draw  (4.5,0.5) node {$5$};
					\draw  (6.5,2.5) node {$6$};
					\draw [blue] (-0.5,0.5) node {\tiny{$i_1$}};
					\draw [blue] (4.5,3.5) node {\tiny{$j_1$}};
					\draw [blue] (-0.5,2.5) node {\tiny{$1$}};
					\draw [blue] (6.5,3.5) node {\tiny{$\nu_1$}};
					\draw [draw=none, fill=gray!50] (2,3) rectangle (6,2);
					\draw [draw=none, fill=gray!50] (2,2) rectangle (3,1);
					\draw [draw=none, fill=gray!50] (0,3) rectangle (2,0);
					\draw [draw=none, fill=gray!50] (6,2) rectangle (7,1);
					\draw [orange!50] (5,2)--(7,0);
					\draw (7.3,0) node {\tiny{$-3$}};
					\draw [orange!50] (4,1)--(6,-1);
					\draw (6.3,-1) node {\tiny{$-1$}};
					\draw [orange!50] (-0.5,3.5)--(2,1);
					\draw (-0.7,3.5) node {\tiny{$1$}};
					\draw [dotted] (0,3) grid (2,0);
					\draw [dotted] (2,3) grid (5,2);
					\draw [blue] (2,3)--(2,0)--(5,0)--(5,1)--(6,1)--(6,2)--(7,2)--(7,3)--(2,3);
					\draw [red,thick] (0,3) rectangle (5,0);
					\draw [red,thick] (5,3) rectangle (7,1);
				\end{scope}
				\begin{scope}[scale=0.3, shift={(2,-5)}]
					\draw (5,3) node {\tiny{$(a^p)=(2^2)$}};
					\draw (5,1.5) node {\tiny{$(b^q)=(5^3)$}};
					\draw (5,0) node {\tiny{$\xi^{(1)}=(6,3,2)$}};
					\draw (5,-1.5) node {\tiny{$\xi^{(2)}=(1,0)$}};
				\end{scope}
			\end{tikzpicture}
		\end{center}
		Furthermore, we obtain other parameters of Etingof-Freund-Ma functor as $N=q+p=5$, $q=3$ and $\mu=\frac{b-a}{N}=\frac{3}{5}$.
		
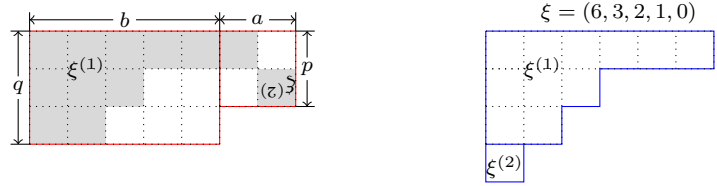
\begin{center}
			\begin{tikzpicture}
				\begin{scope}[scale=0.5, shift={(2,-12)}]
					\draw [blue] (0,3)--(6,3)--(6,2)--(3,2)--(3,1)--(2,1)--(2,0)--(0,0)--(0,3);
					\draw [blue] (0,0)--(0,-1)--(1,-1)--(1,0);
					\draw (3.5,3.5) node {\tiny{$\xi=(6,3,2,1,0)$}};
					\draw (1.5,2) node {\tiny{$\xi^{(1)}$}};
					\draw (0.5,-0.6) node {\tiny{$\xi^{(2)}$}};
					\draw [dotted] (0,3) grid (6,2);
					\draw [dotted] (0,2) grid (2,0);
				\end{scope}
				\begin{scope}[scale=0.5, shift={(-10,-12)}]
					\draw [draw=none,fill=gray!30] (0,3) rectangle (6,2);
					\draw [draw=none,fill=gray!30] (0,2) rectangle (3,1);
					\draw [draw=none,fill=gray!30] (0,1) rectangle (2,0);
					\draw [draw=none,fill=gray!30] (6,2) rectangle (7,1);
					\draw [red] (0,3) rectangle (5,0);
					\draw [red] (5,3) rectangle (7,1);
					\draw [dotted] (0,3) grid (5,0);
					\draw [dotted] (5,3) grid (7,1);
					\draw (2.5,3.3) node {\tiny{$b$}};
					\draw (6,3.3) node {\tiny{$a$}};
					\draw (7,3)--(7,3.5);
					\draw (0,3)--(0,3.5);
					\draw (5,3)--(5,3.5);
					\draw[->] (2.2,3.3)--(0,3.3);
					\draw[->] (2.8,3.3)--(5,3.3);
					\draw[->] (5.7,3.3)--(5,3.3);
					\draw[->] (6.3,3.3)--(7,3.3);
					\draw (-0.3,1.5) node {\tiny{$q$}};
					\draw (7.3,2) node {\tiny{$p$}};
					\draw (0,3)--(-0.5,3);
					\draw (0,0)--(-0.5,0);
					\draw (7,1)--(7.5,1);
					\draw (7,3)--(7.5,3);
					\draw[->] (-0.3,1.8)--(-0.3,3);
					\draw[->] (-0.3,1.2)--(-0.3,0);
					\draw[->] (7.3,2.3)--(7.3,3);
					\draw[->] (7.3,1.7)--(7.3,1);
					\draw (1.5,2) node {\tiny{$\xi^{(1)}$}};
					\draw (6.5, 1.5) node {\tiny{\rotatebox{180}{$\xi^{(2)}$}}};
				\end{scope}
			\end{tikzpicture}
		\end{center}
	\end{eg}
	\textbf{Case 3.} The corner $\zeta_{r_1}=\frac{|\kappa_2|}{2}$ and the corner $\zeta_{r_2}$ doesn't exist. Let $j=s+\frac{|\kappa_2|}{2}$. Then the cell $(0,j)$ on the diagonal of eigenvalue $-\frac{|\kappa_2|}{2}$. We explore the following in two subcases.\\
	\textbf{Case 3a.} $j \geq 1$. Set two rectangles $$(a^p)=(j^1)$$ and $$(b^q)=(\nu_1^{\ell(\nu)+1}).$$ Moreover, $\xi=(\xi_1, \cdots, \xi_{\ell(\nu)})$ with $\xi_1=\nu_1+j$ and $\xi_k=\beta_{k-1}$.\\
	\begin{eg}
		Let $L$ be an irreducible representation in $\co(H_7(1,-2))$ with a minimal weight $\zeta=[-1,2,1,0,3,2,\textcolor{brown}{1}]$ such that $L_{\zeta} \neq 0$. There is only one corner $\zeta_{7}=1$. So $$\zeta_{r_1}=\zeta_7=1=\frac{|\kappa_2|}{2}.$$
		The standard tableau of $\zeta$ is as follows.
		
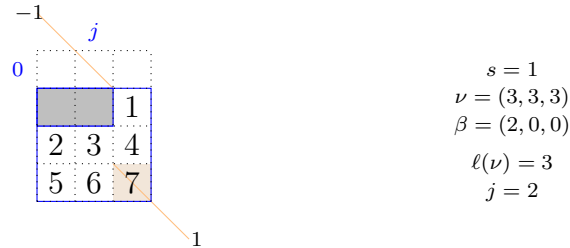
\begin{center}
			\begin{tikzpicture}
				\begin{scope}[scale=0.5, shift={(-8,4)}]
					\draw [blue] (2,3)--(2,2)--(0,2);
					\draw [draw=none, fill=brown!20] (2,1) rectangle (3,0);
					\draw [orange!50] (2,1)--(4,-1);
					\draw (4.2,-1) node {\tiny{$1$}};
					\draw [orange!50] (0,5)--(2,3);
					\draw (-0.2,5) node {\tiny{$-1$}};
					\draw  (2.5,2.5) node {$1$};
					\draw  (2.5,1.5) node {$4$};
					\draw  (0.5,1.5) node {$2$};
					\draw  (1.5,1.5) node {$3$};
					\draw  (2.5,0.5) node {$7$};
					\draw  (0.5,0.5) node {$5$};
					\draw  (1.5,0.5) node {$6$};
					\draw [draw=blue, fill=gray!50] (0,3) rectangle (2,2);
					\draw [blue] (0,3) rectangle (3,0);
					\draw [dotted] (0,4) grid (3,0);
					\draw [blue] (1.5,4.5) node {\tiny{$j$}};
					\draw [blue] (-0.5,3.5) node {\tiny{$0$}};
				\end{scope}
				\begin{scope}[scale=0.5, shift={(2,6)}]
					\draw [scale=0.5] (5,3) node {\tiny{$s=1$}};
					\draw [scale=0.5] (5,1.5) node {\tiny{$\nu=(3,3,3)$}};
					\draw [scale=0.5] (5,0) node {\tiny{$\beta=(2,0,0)$}};
					\draw [scale=0.5] (5,-2) node {\tiny{$\ell(\nu)=3$}};
					\draw [scale=0.5] (5,-3.5) node {\tiny{$j=2$}};
				\end{scope}
			\end{tikzpicture}
		\end{center}
		The two rectangles are $(a^p)=(2^1)$ and $(b^q)=(3^4)$. Place the southeastern corner of $(b^q)$ at $T_{\zeta}(r_1)=T_{\zeta}(7)$ and the northwestern corner of $(a^p)$ at the cell $(0, \nu_1+1)$. The gray area forms $\xi$.\\
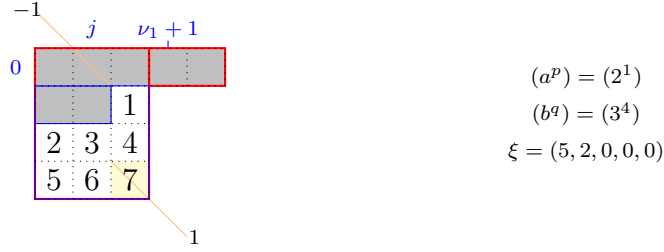
		\begin{center}
			\begin{tikzpicture}
				\begin{scope}[scale=0.5, shift={(-10,4)}]
					\draw [draw =none, fill=yellow!20] (2,1) rectangle (3,0);
					\draw [blue] (1.5,4.5) node {\tiny{$j$}};
					\draw [blue] (3.5,4.5) node {\tiny{$\nu_1+1$}};
					\draw [blue] (3.5,4.2)--(3.5,3.9);
					\draw [blue] (-0.5,3.5) node {\tiny{$0$}};
					\draw [draw=blue, fill=gray!50] (0,3) rectangle (2,2);
					\draw [draw=none, fill=gray!50] (0,4) rectangle (5,3);
					\draw [orange!50] (2,1)--(4,-1);
					\draw (4.2,-1) node {\tiny{$1$}};
					\draw [orange!50] (0,5)--(2,3);
					\draw (-0.2,5) node {\tiny{$-1$}};
					\draw [red,thick] (0,4) rectangle (3,0);
					\draw [red,thick] (3,4) rectangle (5,3);
					\draw [dotted] (0,3) grid (3,0);
					\draw [dotted] (0,4) grid (5,3);
					\draw [blue] (0,3) rectangle (3,0);
					\draw  (2.5,2.5) node {$1$};
					\draw  (2.5,1.5) node {$4$};
					\draw  (0.5,1.5) node {$2$};
					\draw  (1.5,1.5) node {$3$};
					\draw  (2.5,0.5) node {$7$};
					\draw  (0.5,0.5) node {$5$};
					\draw  (1.5,0.5) node {$6$};
				\end{scope}
				\begin{scope}[scale=0.5, shift={(2,4)}]
					\draw [scale=0.5] (5,6.5) node {\tiny{$(a^p)=(2^1)$}};
					\draw [scale=0.5] (5,4.5) node {\tiny{$(b^q)=(3^4)$}};
					\draw [scale=0.5] (5,2.5) node {\tiny{$\xi=(5,2,0,0,0)$}};
				\end{scope}
			\end{tikzpicture}
		\end{center}
		Furthermore, we obtain other parameters of Etingof-Freund-Ma functor as $N=p+q=5$, $p=1$ and $\mu=\frac{a-b}{N}=-\frac{1}{5}$.
		
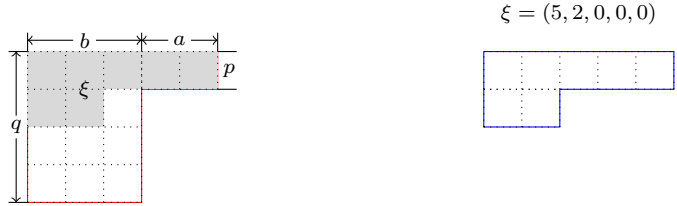
\begin{center}
			\begin{tikzpicture}
				\begin{scope}[scale=0.5, shift={(2,-12)}]
					\draw [blue] (0,4)--(5,4)--(5,3)--(2,3)--(2,2)--(0,2)--(0,4);
					\draw (2.5,5) node {\tiny{$\xi=(5,2,0,0,0)$}};
					\draw [dotted] (0,4) grid (5,3);
					\draw [dotted] (0,3) grid (2,2);
				\end{scope}
				\begin{scope}[scale=0.5, shift={(-10,-12)}]
					\draw [red] (0,4) rectangle (3,0);
					\draw [red] (3,4) rectangle (5,3);
					\draw [draw=none, fill=gray!30] (0,4) rectangle (5,3);
					\draw [draw=none, fill=gray!30] (0,3) rectangle (2,2);
					\draw (1.5,3) node {\tiny{$\xi$}};
					\draw [dotted] (0,4) grid (3,0);
					\draw [dotted] (3,4) grid (5,3);
					\draw (1.5,4.3) node {\tiny{$b$}};
					\draw (4,4.3) node {\tiny{$a$}};
					\draw (5,4)--(5,4.5);
					\draw (0,4)--(0,4.5);
					\draw (3,4)--(3,4.5);
					\draw[->] (1.2,4.3)--(0,4.3);
					\draw[->] (1.8,4.3)--(3,4.3);
					\draw[->] (3.7,4.3)--(3,4.3);
					\draw[->] (4.3,4.3)--(5,4.3);
					\draw (-0.3,2) node {\tiny{$q$}};
					\draw (5.3,3.5) node {\tiny{$p$}};
					\draw (0,4)--(-0.5,4);
					\draw (0,0)--(-0.5,0);
					\draw (5,4)--(5.5,4);
					\draw (5,3)--(5.5,3);
					\draw[->] (-0.3,2.3)--(-0.3,4);
					\draw[->] (-0.3,1.7)--(-0.3,0);
				\end{scope}
			\end{tikzpicture}
		\end{center}
	\end{eg}
	\textbf{Case 3b.} $j \leq 0$. Set two rectangles $$(a^p)=(1^1)$$ and $$(b^q)=((\nu_1-j+1)^{\ell(\nu)+1}).$$ Moreover, $\xi=(\xi_1, \cdots, \xi_{\ell(\nu)})$ with $\xi_1=\nu_1-j+2$ and $\xi_k=\beta_{k-1}-j+1$.\\
	\begin{eg}
		Let $L$ be an irreducible representation in $\co(H_7(1,-2))$ with a minimal weight $\zeta=[0,-2,-1,1,2,0,\textcolor{brown}{1}]$ such that $L_{\zeta} \neq 0$. There is only one corner $\zeta_{7}=1$. So $$\zeta_{r_1}=\zeta_7=1=\frac{|\kappa_2|}{2}.$$
		The standard tableau of $\zeta$ is as follows.\\
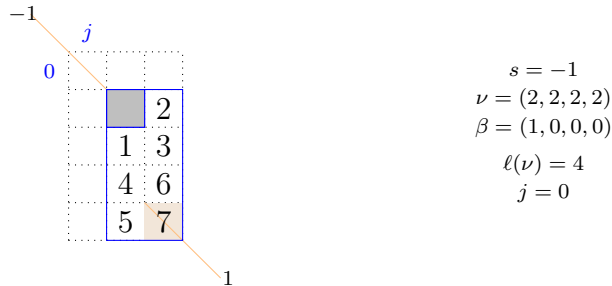
		\begin{center}
			\begin{tikzpicture}
				\begin{scope}[scale=0.5, shift={(-8,4)}]
					\draw [blue] (2,4)--(2,3)--(1,3);
					\draw [dotted] (0,5) grid (3,0);
					\draw [draw=none, fill=brown!20] (2,1) rectangle (3,0);
					\draw [orange!50] (2,1)--(4,-1);
					\draw (4.2,-1) node {\tiny{$1$}};
					\draw [orange!50] (-1,6)--(1,4);
					\draw (-1.2,6) node {\tiny{$-1$}};
					\draw [blue] (-0.5,4.5) node {\tiny{$0$}};
					\draw [blue] (0.5,5.5) node {\tiny{$j$}};
					\draw  (1.5,2.5) node {$1$};
					\draw  (1.5,1.5) node {$4$};
					\draw  (2.5,3.5) node {$2$};
					\draw  (2.5,2.5) node {$3$};
					\draw  (2.5,0.5) node {$7$};
					\draw  (1.5,0.5) node {$5$};
					\draw  (2.5,1.5) node {$6$};
					\draw [draw=blue, fill=gray!50] (1,4) rectangle (2,3);
					\draw [blue] (1,4) rectangle (3,0);
				\end{scope}
				\begin{scope}[scale=0.5, shift={(2,7)}]
					\draw [scale=0.5] (5,3) node {\tiny{$s=-1$}};
					\draw [scale=0.5] (5,1.5) node {\tiny{$\nu=(2,2,2,2)$}};
					\draw [scale=0.5] (5,0) node {\tiny{$\beta=(1,0,0,0)$}};
					\draw [scale=0.5] (5,-2) node {\tiny{$\ell(\nu)=4$}};
					\draw [scale=0.5] (5,-3.5) node {\tiny{$j=0$}};
				\end{scope}
			\end{tikzpicture}
		\end{center}
		The two rectangles are $(a^p)=(1^1)$ and $(b^q)=(3^5)$. Place the southeastern corner of $(b^q)$ at $T_{\zeta}(r_1)=T_{\zeta}(7)$ and the northwestern corner of $(a^p)$ at the cell $(0,\nu_1+1)$. The gray area forms $\xi$.\\
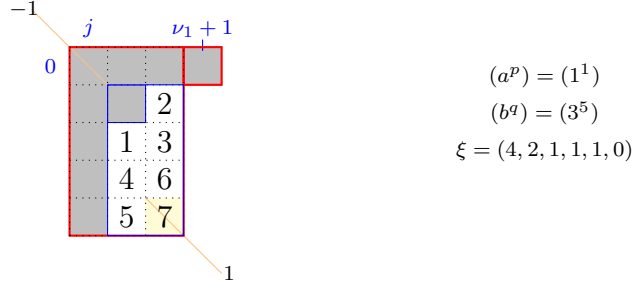
		\begin{center}
			\begin{tikzpicture}
				\begin{scope}[scale=0.5, shift={(-9,4)}]
					\draw [draw =none, fill=yellow!20] (2,1) rectangle (3,0);
					\draw [draw=blue, fill=gray!50] (1,4) rectangle (2,3);
					\draw [draw=none, fill=gray!50] (0,5) rectangle (1,0);
					\draw [draw=none, fill=gray!50] (1,5) rectangle (4,4);
					\draw [orange!50] (2,1)--(4,-1);
					\draw (4.2,-1) node {\tiny{$1$}};
					\draw [orange!50] (-1,6)--(1,4);
					\draw (-1.2,6) node {\tiny{$-1$}};
					\draw [blue] (-0.5,4.5) node {\tiny{$0$}};
					\draw [blue] (0.5,5.5) node {\tiny{$j$}};
					\draw [blue] (3.5,5.5) node {\tiny{$\nu_1+1$}};
					\draw [blue] (3.5,5.2)--(3.5,4.9);
					\draw [red,thick] (0,5) rectangle (3,0);
					\draw [red,thick] (3,5) rectangle (4,4);
					\draw [dotted] (0,5) grid (3,0);
					\draw [blue] (1,4) rectangle (3,0);
					\draw  (1.5,2.5) node {$1$};
					\draw  (1.5,1.5) node {$4$};
					\draw  (2.5,3.5) node {$2$};
					\draw  (2.5,2.5) node {$3$};
					\draw  (2.5,0.5) node {$7$};
					\draw  (1.5,0.5) node {$5$};
					\draw  (2.5,1.5) node {$6$};
				\end{scope}
				\begin{scope}[scale=0.5, shift={(1,5)}]
					\draw [scale=0.5] (5,6.5) node {\tiny{$(a^p)=(1^1)$}};
					\draw [scale=0.5] (5,4.5) node {\tiny{$(b^q)=(3^5)$}};
					\draw [scale=0.5] (5,2.5) node {\tiny{$\xi=(4,2,1,1,1,0)$}};
				\end{scope}
			\end{tikzpicture}
		\end{center}
		
		Furthermore, we obtain other parameters of Etingof-Freund-Ma functor as $N=p+q=6$, $p=1$ and $\mu=\frac{a-b}{N}=-\frac{1}{3}$.
		
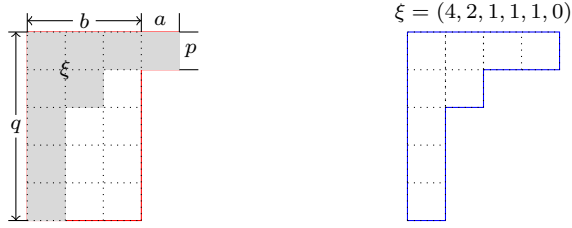
\begin{center}
			\begin{tikzpicture}
				\begin{scope}[scale=0.5, shift={(0,-12)}]
					\draw [blue] (0,5)--(4,5)--(4,4)--(2,4)--(2,3)--(1,3)--(1,0)--(0,0)--(0,5);
					\draw (2,5.5) node {\tiny{$\xi=(4,2,1,1,1,0)$}};
					\draw [dotted] (0,5) grid (1,0);
					\draw [dotted] (1,5) grid (2,3);
					\draw [dotted] (2,5) grid (4,4);
				\end{scope}
				\begin{scope}[scale=0.5, shift={(-10,-12)}]
					\draw [red] (0,5) rectangle (3,0);
					\draw [red] (3,5) rectangle (4,4);
					\draw [draw=none, fill=gray!30] (0,5) rectangle (4,4);
					\draw [draw=none, fill=gray!30] (0,4) rectangle (2,3);
					\draw [draw=none, fill=gray!30] (0,3) rectangle (1,0);
					\draw (1,4) node {\tiny{$\xi$}};
					\draw [dotted] (0,5) grid (3,0);
					\draw (1.5,5.3) node {\tiny{$b$}};
					\draw (3.5,5.3) node {\tiny{$a$}};
					\draw (0,5)--(0,5.5);
					\draw (3,5)--(3,5.5);
					\draw (4,5)--(4,5.5);
					\draw[->] (1.2,5.3)--(0,5.3);
					\draw[->] (1.8,5.3)--(3,5.3);
					\draw (-0.3,2.5) node {\tiny{$q$}};
					\draw (4.3,4.5) node {\tiny{$p$}};
					\draw (0,5)--(-0.5,5);
					\draw (0,0)--(-0.5,0);
					\draw (4,4)--(4.5,4);
					\draw (4,5)--(4.5,5);
					\draw[->] (-0.3,2.8)--(-0.3,5);
					\draw[->] (-0.3,2.2)--(-0.3,0);
				\end{scope}
			\end{tikzpicture}
		\end{center}
	\end{eg}

	\textbf{Case 4.} The corner $\zeta_{r_1}=-\frac{|\kappa_2|}{2}$ and there is no corner $\zeta_{r_2}$. Set $j=s-\frac{|\kappa_2|}{2}$. Then the cell $(0,j)$ is on the diagonal of eigenvalue $\frac{|\kappa_2|}{2}$. Let us discuss in two subcases.\\
	\textbf{Case 4a.} When $j \geq 1$. Set two rectangles $$(a^p)=(j^1)$$ and $$(b^q)=(\nu_1^{\ell(\nu)+1}).$$ Moreover, $\xi=(\xi_1, \cdots, \xi_{\ell(\nu)})$ with $\xi_1=\nu_1+j$ and $\xi_k=\beta_{k-1}$.\\
	\begin{eg}
		Let $L$ be an irreducible representation in $\co(H_7(1,-2))$ with a minimal weight $\zeta=[4,3,2,-2,1,0,\textcolor{brown}{-1}]$ such that $L_{\zeta} \neq 0$. There is only one corner $\zeta_{7}=-1$. So $$\zeta_{r_1}=\zeta_7=-1=-\frac{|\kappa_2|}{2}.$$
		The standard tableau of $\zeta$ is as follows.
		
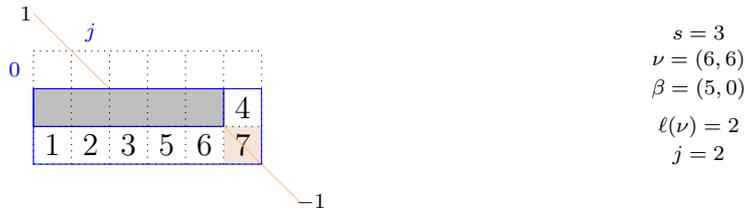
\begin{center}
			\begin{tikzpicture}
				\begin{scope}[scale=0.5, shift={(-10,3)}]
					\draw [blue] (5,2)--(5,1)--(0,1);
					\draw [draw=none, fill=brown!20] (5,1) rectangle (6,0);
					\draw  (0.5,0.5) node {$1$};
					\draw  (5.5,1.5) node {$4$};
					\draw  (1.5,0.5) node {$2$};
					\draw  (2.5,0.5) node {$3$};
					\draw  (5.5,0.5) node {$7$};
					\draw  (3.5,0.5) node {$5$};
					\draw  (4.5,0.5) node {$6$};
					\draw [draw=blue, fill=gray!50] (0,2) rectangle (5,1);
					\draw [dotted] (0,3) grid (6,0);
					\draw [orange!50] (5,1)--(7,-1);
					\draw (7.3,-1) node {\tiny{$-1$}};
					\draw [orange!50] (0,4)--(2,2);
					\draw (-0.2,4) node {\tiny{$1$}};
					\draw [blue] (1.5,3.5) node {\tiny{$j$}};
					\draw [blue] (-0.5,2.5) node {\tiny{$0$}};
					\draw [blue] (0,2) rectangle (6,0);
				\end{scope}
				\begin{scope}[scale=0.5, shift={(5,5)}]
					\draw [scale=0.5] (5,3) node {\tiny{$s=3$}};
					\draw [scale=0.5] (5,1.5) node {\tiny{$\nu=(6,6)$}};
					\draw [scale=0.5] (5,0) node {\tiny{$\beta=(5,0)$}};
					\draw [scale=0.5] (5,-2) node {\tiny{$\ell(\nu)=2$}};
					\draw [scale=0.5] (5,-3.5) node {\tiny{$j=2$}};
				\end{scope}
			\end{tikzpicture}
		\end{center}
		The two rectangles are $(a^p)=(2^1)$ and $(b^q)=(6^3)$. Place the southeastern corner of $(b^q)$ at $T_{\zeta}(r_1)=T_{\zeta}(7)$ and the northwestern corner of $(a^p)$ at cell $(0,\nu_1+1)=(0,7)$. The gray area forms $\xi$.
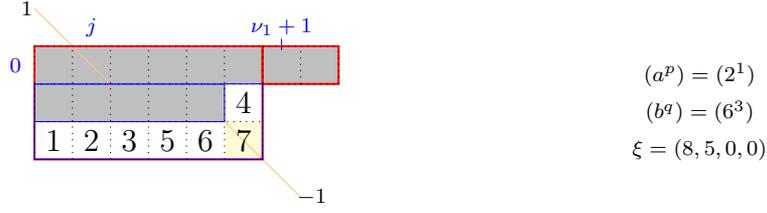
		\begin{center}
			\begin{tikzpicture}
				\begin{scope}[scale=0.5, shift={(-11,4)}]
					\draw [draw =none, fill=yellow!20] (5,1) rectangle (6,0);
					\draw [draw=blue, fill=gray!50] (0,2) rectangle (5,1);
					\draw [draw=none, fill=gray!50] (0,3) rectangle (8,2);
					\draw [orange!50] (5,1)--(7,-1);
					\draw (7.3,-1) node {\tiny{$-1$}};
					\draw [orange!50] (0,4)--(2,2);
					\draw (-0.2,4) node {\tiny{$1$}};
					\draw [blue] (1.5,3.5) node {\tiny{$j$}};
					\draw [blue] (6.5,3.5) node {\tiny{$\nu_1+1$}};
					\draw [blue] (6.5,3.2)--(6.5,2.9);
					\draw [blue] (-0.5,2.5) node {\tiny{$0$}};
					\draw [red,thick] (0,3) rectangle (6,0);
					\draw [red,thick] (6,3) rectangle (8,2);
					\draw [dotted] (0,3) grid (6,0);
					\draw [dotted] (6,3) grid (8,2);
					\draw [blue] (0,2) rectangle (6,0);
					\draw  (0.5,0.5) node {$1$};
					\draw  (5.5,1.5) node {$4$};
					\draw  (1.5,0.5) node {$2$};
					\draw  (2.5,0.5) node {$3$};
					\draw  (5.5,0.5) node {$7$};
					\draw  (3.5,0.5) node {$5$};
					\draw  (4.5,0.5) node {$6$};
				\end{scope}
				\begin{scope}[scale=0.5, shift={(4,3)}]
					\draw [scale=0.5] (5,6.5) node {\tiny{$(a^p)=(2^1)$}};
					\draw [scale=0.5] (5,4.5) node {\tiny{$(b^q)=(6^3)$}};
					\draw [scale=0.5] (5,2.5) node {\tiny{$\xi=(8,5,0,0)$}};
				\end{scope}
			\end{tikzpicture}
		\end{center}
		
		Furthermore, we obtain other parameters of Etingof-Freund-Ma functor as $N=q+p=4$, $q=3$ and $\mu=\frac{b-a}{N}=1$.
		
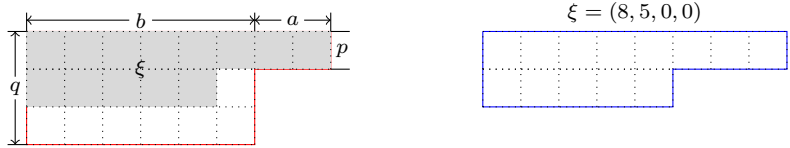
\begin{center}
			\begin{tikzpicture}
				\begin{scope}[scale=0.5, shift={(2,-12)}]
					\draw [blue] (0,3)--(8,3)--(8,2)--(5,2)--(5,1)--(0,1)--(0,3);
					\draw (4,3.5) node {\tiny{$\xi=(8,5,0,0)$}};
					\draw [dotted] (0,3) grid (8,2);
					\draw [dotted] (0,2) grid (5,1);
				\end{scope}
				\begin{scope}[scale=0.5, shift={(-10,-12)}]
					\draw [red] (0,3) rectangle (6,0);
					\draw [red] (6,3) rectangle (8,2);
					\draw [draw=none, fill=gray!30] (0,3) rectangle (8,2);
					\draw [draw=none, fill=gray!30] (0,2) rectangle (5,1);
					\draw (3,2) node {\tiny{$\xi$}};
					\draw [dotted] (0,3) grid (8,2);
					\draw [dotted] (0,2) grid (6,0);
					\draw (3,3.3) node {\tiny{$b$}};
					\draw (7,3.3) node {\tiny{$a$}};
					\draw (0,3)--(0,3.5);
					\draw (6,3)--(6,3.5);
					\draw (8,3)--(8,3.5);
					\draw[->] (2.7,3.3)--(0,3.3);
					\draw[->] (3.3,3.3)--(6,3.3);
					\draw[->] (6.7,3.3)--(6,3.3);
					\draw[->] (7.3,3.3)--(8,3.3);
					\draw (-0.3,1.5) node {\tiny{$q$}};
					\draw (8.3,2.5) node {\tiny{$p$}};
					\draw (0,3)--(-0.5,3);
					\draw (0,0)--(-0.5,0);
					\draw (8,3)--(8.5,3);
					\draw (8,2)--(8.5,2);
					\draw[->] (-0.3,1.8)--(-0.3,3);
					\draw[->] (-0.3,1.2)--(-0.3,0);
				\end{scope}
			\end{tikzpicture}
		\end{center}
	\end{eg}
	\textbf{Case 4b.} When $j \leq 0$. Set two rectangles $$(a^p)=(1^1)$$ and $$(b^q)=((\nu_1-j+1)^{\ell(\nu)+1}).$$ Moreover, $\xi=(\xi_1, \cdots, \xi_{\ell(\nu)})$ with $\xi_1=\nu_1-j+2$ and $\xi_k=\beta_{k-1}-j+1$.\\
	\begin{eg}
		Let $L$ be an irreducible representation in $\co(H_7(1,-2))$ with a minimal weight $\zeta=[0,-1,2,1,-2,0,\textcolor{brown}{-1}]$ such that $L_{\zeta} \neq 0$. There is only one corner $\zeta_{7}=-1$. So $$\zeta_{r_1}=\zeta_7=-1=-\frac{|\kappa_2|}{2}.$$
		The standard tableau of $\zeta$ is as follows.
		
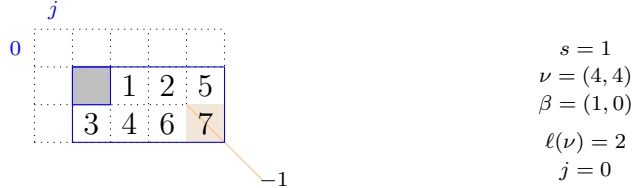
\begin{center}
			\begin{tikzpicture}
				\begin{scope}[scale=0.5, shift={(-10,4)}]
					\draw [blue] (2,2)--(2,1)--(1,1);
					\draw [dotted] (0,3) grid (5,0);
					\draw [blue] (-0.5, 2.5) node {\tiny{$0$}};
					\draw [blue] (0.5, 3.5) node {\tiny{$j$}};
					\draw [draw=none, fill=brown!20] (4,1) rectangle (5,0);
					\draw [orange!50] (4,1)--(6,-1);
					\draw (6.3,-1) node {\tiny{$-1$}};
					\draw  (2.5,1.5) node {$1$};
					\draw  (2.5,0.5) node {$4$};
					\draw  (3.5,1.5) node {$2$};
					\draw  (1.5,0.5) node {$3$};
					\draw  (4.5,0.5) node {$7$};
					\draw  (4.5,1.5) node {$5$};
					\draw  (3.5,0.5) node {$6$};
					\draw [draw=blue, fill=gray!50] (1,2) rectangle (2,1);
					\draw [blue] (1,2) rectangle (5,0);
				\end{scope}
				\begin{scope}[scale=0.5, shift={(2,5)}]
					\draw [scale=0.5] (5,3) node {\tiny{$s=1$}};
					\draw [scale=0.5] (5,1.5) node {\tiny{$\nu=(4,4)$}};
					\draw [scale=0.5] (5,0) node {\tiny{$\beta=(1,0)$}};
					\draw [scale=0.5] (5,-2) node {\tiny{$\ell(\nu)=2$}};
					\draw [scale=0.5] (5,-3.5) node {\tiny{$j=0$}};
				\end{scope}
			\end{tikzpicture}
		\end{center}
		The two rectangles are $(a^p)=(1^1)$ and $(b^q)=(5^3)$. Place the southeastern corner of $(b^q)$ at $T_{\zeta}(r_1)=T_{\zeta}(7)$ and the northwestern corner of $(a^p)$ at the cell $(0,\nu_1+1)$. The gray area forms $\xi$.
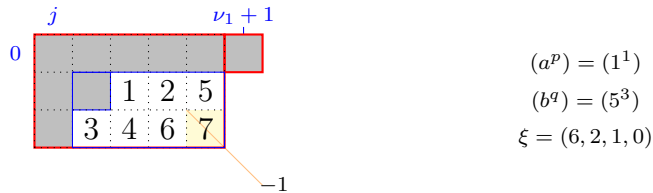
		\begin{center}
			\begin{tikzpicture}
				\begin{scope}[scale=0.5, shift={(-10,4)}]
					\draw [blue] (-0.5, 2.5) node {\tiny{$0$}};
					\draw [blue] (0.5, 3.5) node {\tiny{$j$}};
					\draw [blue] (5.5, 3.5) node {\tiny{$\nu_1+1$}};
					\draw [blue] (5.5,3.2)--(5.5,2.7);
					\draw [draw =none, fill=yellow!20] (4,1) rectangle (5,0);
					\draw [orange!50] (4,1)--(6,-1);
					\draw (6.3,-1) node {\tiny{$-1$}};
					\draw [draw=blue, fill=gray!50] (1,2) rectangle (2,1);
					\draw [draw=none, fill=gray!50] (0,3) rectangle (1,0);
					\draw [draw=none,fill=gray!50] (1,3) rectangle (6,2);
					\draw [red,thick] (0,3) rectangle (5,0);
					\draw [red,thick] (5,3) rectangle (6,2);
					\draw [dotted](0,3) grid (5,0);
					\draw [blue] (1,2) rectangle (5,0);
					\draw  (2.5,1.5) node {$1$};
					\draw  (2.5,0.5) node {$4$};
					\draw  (3.5,1.5) node {$2$};
					\draw  (1.5,0.5) node {$3$};
					\draw  (4.5,0.5) node {$7$};
					\draw  (4.5,1.5) node {$5$};
					\draw  (3.5,0.5) node {$6$};
				\end{scope}
				\begin{scope}[scale=0.5, shift={(2,3)}]
					\draw [scale=0.5] (5,6.5) node {\tiny{$(a^p)=(1^1)$}};
					\draw [scale=0.5] (5,4.5) node {\tiny{$(b^q)=(5^3)$}};
					\draw [scale=0.5] (5,2.5) node {\tiny{$\xi=(6,2,1,0)$}};
				\end{scope}
			\end{tikzpicture}
		\end{center}
		Furthermore, we obtain other parameters of Etingof-Freund-Ma functor as $N=q+p=4$, $q=3$ and $\mu=\frac{b-a}{N}=1$.
		
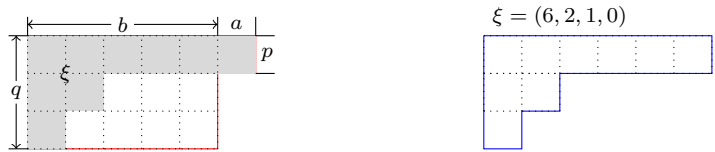
\begin{center}
			\begin{tikzpicture}
				\begin{scope}[scale=0.5, shift={(2,-12)}]
					\draw [blue] (0,3)--(6,3)--(6,2)--(2,2)--(2,1)--(1,1)--(1,0)--(0,0)--(0,3);
					\draw (2,3.5) node {\tiny{$\xi=(6,2,1,0)$}};
					\draw [dotted] (0,3) grid (6,2);
					\draw [dotted] (0,2) grid (2,1);
				\end{scope}
				\begin{scope}[scale=0.5, shift={(-10,-12)}]
					\draw [red] (0,3) rectangle (5,0);
					\draw [red] (5,3) rectangle (6,2);
					\draw [draw=none,fill=gray!30] (0,3) rectangle (6,2);
					\draw [draw=none,fill=gray!30] (0,2) rectangle (2,1);
					\draw [draw=none,fill=gray!30] (0,1) rectangle (1,0);
					\draw (1,2) node {\tiny{$\xi$}};
					\draw [dotted] (0,3) grid (5,0);
					\draw (2.5,3.3) node {\tiny{$b$}};
					\draw (5.5,3.3) node {\tiny{$a$}};
					\draw (0,3)--(0,3.5);
					\draw (5,3)--(5,3.5);
					\draw (6,3)--(6,3.5);
					\draw[->] (2.2,3.3)--(0,3.3);
					\draw[->] (2.8,3.3)--(5,3.3);
					\draw (-0.3,1.5) node {\tiny{$q$}};
					\draw (6.3,2.5) node {\tiny{$p$}};
					\draw (0,3)--(-0.5,3);
					\draw (0,0)--(-0.5,0);
					\draw (6,2)--(6.5,2);
					\draw (6,3)--(6.5,3);
					\draw[->] (-0.3,1.8)--(-0.3,3);
					\draw[->] (-0.3,1.2)--(-0.3,0);
				\end{scope}
			\end{tikzpicture}
		\end{center}
	\end{eg}

	\textbf{Case 5.} The corner $\zeta_{r_1}<-\frac{|\kappa_2|}{2}$. Let $j_1=\nu_{\ell(\nu)+\frac{|\kappa_2|}{2}}+\zeta_{r_1}$ and $j_2=\nu_{\ell(\nu)-\frac{|\kappa_2|}{2}}+\zeta_{r_1}$. Set two rectangles $$(a^p)=((\nu_1-j_1)^{\ell(\nu)})$$ and $$(b^q)=((\nu_1-j_2)^{\ell(\nu)}).$$
	\begin{claim}
		According to the setting above, the number $\nu_{\ell(\nu)}-j_1-j_2 \geq 0$
	\end{claim}
	
	\begin{proof}
		There exist a weight $\tilde{\zeta}$ such that $L_{\tilde{\zeta}} \neq 0$, $Im(T_{\tilde{\zeta}})=Im(T_{\zeta})$ and $T_{\tilde{\zeta}}(n)=(\ell(\nu), \nu_{\ell(\nu)})$. Let $v$ be a nonzero weight vector of weight $\tilde{\zeta}$. Since $\zeta_{r_1}<-\frac{|\kappa_2|}{2}$, we obtain a nonzero weight vector $\phi_n v$ of weight $\gamma_n \tilde{\zeta}$. Moreover, $$Im(T_{\gamma_n \tilde{\zeta}})=Im(T_{\zeta}) \setminus \{(\ell(\nu), \nu_{\ell(\nu)})\} \cup \{(\ell(\nu)+1, 2\ell(\nu)-\nu_{\ell(\nu)}+2s+1)\}.$$
		Since $Im(T_{\gamma_n \tilde{\zeta}})$ is a skew shape, it follows $2\ell(\nu)-\nu_{\ell(\nu)}+2s+1 \leq 1$. Applying $j_1=\nu_{\ell(\nu)}+\frac{|\kappa_2|}{2}+\zeta_{r_1}$ and $j_2=\nu_{\ell(\nu)}-\frac{|\kappa_2|}{2}+\zeta_{r_1}$, the statement $\nu_{\ell(\nu)}-j_1-j_2 \geq 0$ follows.
	\end{proof}
	Set $\xi^{(1)}=(\xi^{(1)}_1, \cdots, \xi^{(1)}_{\ell(\nu)})$ with $$\xi^{(1)}_k=\beta_k+\nu_1-j_1-j_2$$
	for $k=1, \cdots, \ell(\nu)$, $\xi^{(2)}=(\xi^{(2)}_1, \cdots, \xi^{(2)}_{\ell(\nu)})$ with
	$$\xi^{(2)}_k=\nu_1-\nu_{\ell(\nu)-k+1}$$
	for $k=1, \cdots, \ell(\nu)$ and $\xi=(\xi_1,\cdots, \xi_{2\ell(\nu)})$ with
	$$\xi_k=\xi^{(1)}_k$$ for $k=1,\cdots, \ell(\nu)$ and $$\xi_k=\xi^{(2)}_{k-\ell(\nu)}$$ for $k=\ell(\nu)+1, \cdots, 2\ell(\nu)$.
	\begin{rmk}
		Claim 10.22 implies the following two facts.
		\begin{enumerate}
			\item It follows $\nu_1-j_1-j_2 \geq 0$.
			\item The inequality $\nu_1-\nu_{\ell(\nu)}=\xi^{(2)}_1 \leq \xi^{(1)}_{\ell(\nu)}=\nu_1-j_1-j_2$ holds and hence $\xi$ is a well-defined Young diagram.
		\end{enumerate}
	\end{rmk}
	\begin{eg}
		Let $L$ be an irreducible representation in $\co(H_7(1,-2))$ with a minimal weight $\zeta=[-2,-1,-5,\textcolor{brown}{-6},-3,\textcolor{brown}{-4},\textcolor{brown}{-2}]$ such that $L_{\zeta} \neq 0$. The corners of $\zeta$ are $\zeta_4=-6$, $\zeta_6=-4$ and $\zeta_7=-2$. So $\zeta_{r_1}=\zeta_7=-2$. The standard tableau of $\zeta$ is as follows.
		
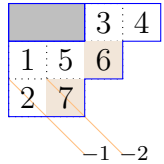
\begin{center}
			\begin{tikzpicture}
				\begin{scope}[scale=0.5, shift={(-10,4)}]
					\draw [dotted] (3,2) grid (5,0);
					\draw [dotted](5,3) grid (6,1);
					\draw [draw =none, fill=brown!20] (4,1) rectangle (5,0);
					\draw [draw =none, fill=brown!20] (5,2) rectangle (6,1);
					\draw [orange!50] (4,1)--(6,-1);
					\draw (6.3, -1) node {\tiny{$-2$}};
					\draw [orange!50] (3,1)--(5,-1);
					\draw (5.3, -1) node {\tiny{$-1$}};
					\draw  (3.5,1.5) node {$1$};
					\draw  (6.5,2.5) node {$4$};
					\draw  (3.5,0.5) node {$2$};
					\draw  (4.5,0.5) node {$7$};
					\draw  (5.5,1.5) node {$6$};
					\draw  (4.5,1.5) node {$5$};
					\draw  (5.5,2.5) node {$3$};
					\draw [draw=blue, fill=gray!50] (3,3) rectangle (5,2);
					\draw [blue] (3,2)--(5,2)--(5,3);
					\draw [blue] (3,3)--(3,0)--(5,0)--(5,1)--(6,1)--(6,2)--(7,2)--(7,3)--(3,3);
				\end{scope}
				\begin{scope}[scale=0.5, shift={(2,5)}]
					\draw [scale=0.5] (5,3) node {\tiny{$s=-3$}};
					\draw [scale=0.5] (5,1.5) node {\tiny{$\nu=(4,3,2)$}};
					\draw [scale=0.5] (5,0) node {\tiny{$\beta=(2,0,0)$}};
					\draw [scale=0.5] (5,-1.5) node {\tiny{$\ell(\nu)=3$}};
				\end{scope}
			\end{tikzpicture}
		\end{center}
		The two rectangles $(a^p)=(3^3)$ and $(b^q)=(5^3)$ follow. Place the northeastern corner of $(a^p)=(3^3)$ at the cell $(1, \nu_1)$ and the southeastern corner of $(b^q)=(5^3)$ at the cell $(\ell(\nu), \ell(\nu)+\frac{|\kappa_2|}{2}+s$. The gray area on the left forms $\xi^{(1)}$ and the gray area on the right forms \rotatebox[origin=c]{180}{$\xi^{(2)}$}.
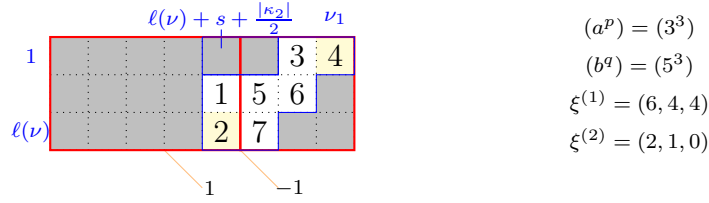
		\begin{center}
			\begin{tikzpicture}
				\begin{scope}[scale=0.5, shift={(-10,4)}]
					\draw [orange!50] (1,1)--(3,-1);
					\draw (3.2, -1) node {\tiny{$1$}};
					\draw [orange!50] (3,1)--(5,-1);
					\draw (5.3, -1) node {\tiny{$-1$}};
					\draw [draw =none, fill=yellow!20] (3,1) rectangle (4,0);
					\draw [draw =none, fill=yellow!20] (6,3) rectangle (7,2);
					\draw [draw=blue, fill=gray!50] (3,3) rectangle (5,2);
					\draw [draw=none, fill=gray!50] (-1,3) rectangle (3,0);
					\draw [draw=none,fill=gray!50] (5,1) rectangle (7,0);
					\draw [draw=none,fill=gray!50] (6,2) rectangle (7,1);
					\draw [blue] (-1.5,0.5) node {\tiny{$\ell(\nu)$}};
					\draw [blue] (3.5,3.5) node {\tiny{$\ell(\nu)+s+\frac{|\kappa_2|}{2}$}};
					\draw [blue] (-1.5,2.5) node {\tiny{$1$}};
					\draw [blue] (6.5,3.5) node {\tiny{$\nu_1$}};
					\draw [blue] (3.5,3.2)--(3.5,2.7);
					\draw [dotted] (-1,3) grid (7,0);
					\draw [red, thick] (-1,3) rectangle (4,0);
					\draw [red,thick] (4,3) rectangle (7,0);
					\draw [blue] (3,2)--(5,2)--(5,3);
					\draw [blue] (3,3)--(3,0)--(5,0)--(5,1)--(6,1)--(6,2)--(7,2)--(7,3)--(3,3);
					\draw  (3.5,1.5) node {$1$};
					\draw  (6.5,2.5) node {$4$};
					\draw  (3.5,0.5) node {$2$};
					\draw  (4.5,0.5) node {$7$};
					\draw  (5.5,1.5) node {$6$};
					\draw  (4.5,1.5) node {$5$};
					\draw  (5.5,2.5) node {$3$};
				\end{scope}
				\begin{scope}[scale=0.5, shift={(2,4)}]
					\draw [scale=0.5] (5,6.5) node {\tiny{$(a^p)=(3^3)$}};
					\draw [scale=0.5] (5,4.5) node {\tiny{$(b^q)=(5^3)$}};
					\draw [scale=0.5] (5,2.5) node {\tiny{$\xi^{(1)}=(6,4,4)$}};
					\draw [scale=0.5] (5,0.5) node {\tiny{$\xi^{(2)}=(2,1,0)$}};
				\end{scope}
			\end{tikzpicture}
		\end{center}
		So the three shapes $(a^p)$, $(b^q)$ and $\xi$ are set as follows. The other parameters of Etingof-Freund-Ma functor are set as $N=6$, $p=3$ and $\mu=1/3$.
		
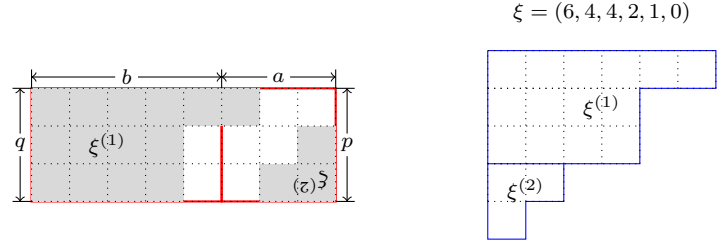
\begin{center}
			\begin{tikzpicture}
				\begin{scope}[scale=0.5,shift={(-10,3)}]
					\draw [red,thick] (-1,3) rectangle (4,0);
					\draw [red,thick] (4,3) rectangle (7,0);
					\draw [draw=none,fill=gray!30] (-1,3) rectangle (5,2);
					\draw [draw=none,fill=gray!30] (-1,2) rectangle (3,0);
					\draw (1,1.5) node {\tiny{$\xi^{(1)}$}};
					\draw [draw=none,fill=gray!30] (6,2) rectangle (7,1);
					\draw [draw=none,fill=gray!30] (5,1) rectangle (7,0);
					\draw (6.3,0.5) node {\tiny{\rotatebox{180}{$\xi^{(2)}$}}};
					\draw [dotted] (-1,3) grid (7,0);
					\draw (1.5,3.3) node {\tiny{$b$}};
					\draw (5.5,3.3) node {\tiny{$a$}};
					\draw (-1.3,1.5) node {\tiny{$q$}};
					\draw (7.3,1.5) node {\tiny{$p$}};
					\draw (-1,3)--(-1,3.5);
					\draw (4,3)--(4,3.5);
					\draw (7,3)--(7,3.5);
					\draw[->] (1.2,3.3)--(-1,3.3);
					\draw[->] (1.8,3.3)--(4,3.3);
					\draw[->] (5.2,3.3)--(4,3.3);
					\draw[->] (5.8,3.3)--(7,3.3);
					\draw (-1,3)--(-1.5,3);
					\draw (-1,0)--(-1.5,0);
					\draw (7,3)--(7.5,3);
					\draw (7,0)--(7.5,0);
					\draw[->] (-1.3,1.8)--(-1.3,3);
					\draw[->] (-1.3,1.2)--(-1.3,0);
					\draw[->] (7.3,1.8)--(7.3,3);
					\draw[->] (7.3,1.2)--(7.3,0);
				\end{scope}
				\begin{scope}[scale=0.5, shift={(2,4)}]
					\draw [blue] (-1,3)--(-1,0)--(3,0)--(3,2)--(5,2)--(5,3)--(-1,3);
					\draw [blue] (-1,0)--(-1,-2)--(0,-2)--(0,-1)--(1,-1)--(1,0);
					\draw (2,1.5) node {\tiny{$\xi^{(1)}$}};
					\draw (0,-0.7) node {\tiny{$\xi^{(2)}$}};
					\draw [dotted] (-1,3) grid (5,2);
					\draw [dotted] (-1,2) grid (3,0);
					\draw [dotted] (-1,0) grid (1,-1);
					\draw (2,4) node {\tiny{$\xi=(6,4,4,2,1,0)$}};
				\end{scope}
			\end{tikzpicture}
		\end{center}
	\end{eg}
	\begin{rmk}
		When we fix the number $n$, for different input $(\xi, N, p, \mu)$, we could actually get isomorphic $H_n$-modules. Consider the following example of representations of $H_3(1, -1)$. \\
		Let $\xi = (3,3,2)$, $N = 4$ , $ p = 1$ and $\mu = - \frac{1}{4}$.\\
		In this case,$ a=\mu q + \frac{|\xi| + n}{N} = 2$ and $b= - \mu p + \frac{|\xi| + n}{N} = 3$. Then the image $F=F_{3,1,- \frac{1}{4}}(V^{\xi})$ is an $H_3(1,-1)$-module with the following minimal shape $\varphi_{3,1,-\frac{1}{4}}^{\xi}=(5,3,3)/(3,3,2)$.\\
		\begin{center}
			\begin{tikzpicture}
				\begin{scope}[scale=0.4,shift={(0,0)}]
					\draw [red,thick] (0,3) rectangle (3,0);
					\draw [red,thick] (3,3) rectangle (5,2);
					\draw [draw=none, fill= gray!50] (0,3) rectangle (3,1);
					\draw (1.5,2) node {\tiny{$\xi$}};
					\draw [draw=none, fill= gray!50] (0,1) rectangle (2,0);
					\draw [dotted] (0,3) grid (3,0);
					\draw [dotted] (3,3) grid (5,2);
					\draw (1.5,3.3) node {\tiny{$b$}};
					\draw (4,3.3) node {\tiny{$a$}};
					\draw (-0.3,1.5) node {\tiny{$q$}};
					\draw (5.3,2.5) node {\tiny{$p$}};
					\draw (0,3)--(0,3.5);
					\draw (3,3)--(3,3.5);
					\draw (5,3)--(5,3.5);
					\draw[->] (1.2,3.3)--(0,3.3);
					\draw[->] (1.8,3.3)--(3,3.3);
					\draw[->] (4.2,3.3)--(5,3.3);
					\draw[->] (3.8,3.3)--(3,3.3);
					\draw (0,3)--(-0.5,3);
					\draw (0,0)--(-0.5,0);
					\draw (5,3)--(5.5,3);
					\draw (5,2)--(5.5,2);
					\draw[->] (-0.3,1.8)--(-0.3,3);
					\draw[->] (-0.3,1.2)--(-0.3,0);
					\draw [orange!50] (2,1)--(4,-1);
					\draw (4.2,-1) node {\tiny{$\frac{1}{2}$}};
				\end{scope}
			\end{tikzpicture}
		\end{center}
		Then the basis is indexed by the standard tableaux on the skew shapes:
		$(5,3,3)/\xi$, $(4,3,3,1)/\xi$ and $(3,3,3,2)/\xi$. There is a minimal weight $\zeta=[\frac{1}{2}, -\frac{5}{2}, -\frac{7}{2}]$ such that $F_{\zeta} \neq 0$.
		Now let us recover a functor $F_{n,p',\mu'}$ such that $F_{n,p',\mu'}(V^{\xi'})$ is an $H_3(1,-1)$-module with a minimal weight $\zeta=[\textcolor{brown}{\frac{1}{2}}, -\frac{5}{2}, \textcolor{brown}{-\frac{7}{2}}]$. According to Case 1, $(a'^{p'})=(3^1)$, $(b'^{q'})=(3^2)$, $\xi'=(4,2,0)$ and $\mu'=0$.\\
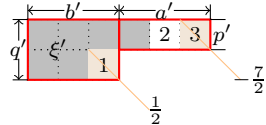
		\begin{center}
			\begin{tikzpicture}
				\begin{scope}[scale=0.4,shift={(0,0)}]
					\draw [draw=none, fill= gray!50] (0,2) rectangle (2,0);
					\draw [draw=none, fill= gray!50] (0,2) rectangle (4,1);
					\draw (1,1) node {\tiny{$\xi'$}};
					\draw [dotted] (0,2) grid (6,1);
					\draw [dotted] (0,1) grid (3,0);
					\draw [draw=none,fill=brown!20] (2,1) rectangle (3,0);
					\draw [draw=none,fill=brown!20] (5,2) rectangle (6,1);
					\draw [red,thick] (0,2) rectangle (3,0);
					\draw [red,thick] (3,2) rectangle (6,1);
					\draw (2.5,0.5) node {\tiny{$1$}};
					\draw (4.5,1.5) node {\tiny{$2$}};
					\draw (5.5,1.5) node {\tiny{$3$}};
					\draw (1.5,2.3) node {\tiny{$b'$}};
					\draw (4.5,2.3) node {\tiny{$a'$}};
					\draw (-0.3,1) node {\tiny{$q'$}};
					\draw (6.4,1.5) node {\tiny{$p'$}};
					\draw (0,2)--(0,2.5);
					\draw (3,2)--(3,2.5);
					\draw (6,2)--(6,2.5);
					\draw[->] (1.2,2.3)--(0,2.3);
					\draw[->] (1.8,2.3)--(3,2.3);
					\draw[->] (4.8,2.3)--(6,2.3);
					\draw[->] (4.2,2.3)--(3,2.3);
					\draw (0,2)--(-0.5,2);
					\draw (0,0)--(-0.5,0);
					\draw (6,2)--(6.5,2);
					\draw (6,1)--(6.5,1);
					\draw[->] (-0.3,1.3)--(-0.3,2);
					\draw[->] (-0.3,0.7)--(-0.3,0);
					\draw [orange!50] (2,1)--(4,-1);
					\draw (4.2,-1) node {\tiny{$\frac{1}{2}$}};
					\draw [orange!50] (5,2)--(7,0);
					\draw (7.3,0) node {\tiny{$-\frac{7}{2}$}};
				\end{scope}
			\end{tikzpicture}
		\end{center}
		
		
	\end{rmk}
	
	\subsection{Other $\Y$-semisimple representations}
	The image of the Etingof-Freund-Ma functor does not exhaust all the $\Y$-semisimple representations. The following are two examples of $\Y$-semisimple $H_n(1,\kappa_2)$ representation which are not in $\co(H_n(1,\kappa_2))$.
	\begin{eg}
		Obviously, the representation obtained under the Etingof-Freund-Ma does not contain a weight vector of weight $\zeta$ with $ -\frac{|\kappa_2|}{2}< \zeta_n < \frac{|\kappa_2|}{2}$.  
		
		Consider the representation of $H_3(1, -6)$ generated by the weight vector of weight $[1,2,-3]$. This representation has the following characters:\\
		\begin{center}
			\begin{tikzpicture}[scale=0.5]
				\begin{scope}[shift = {(8,0)}]
					\draw (0,0) node {\tiny{$[-3,-2,-1]$}};
					\draw [->] (0,-0.3)--(0, -1.7);
					\draw (-0.5,-1) node {\tiny{$\m_3$}};
					\draw (0,-2) node {\tiny{$[-3,-2,1]$}};
					\draw [->](0,-2.3)--(0, -3.7);
					\draw (-0.5,-3) node {\tiny{$\m_2$}};
					\draw (0,-4) node {\tiny{$[-3,1,-2]$}};
					\draw [->](0.3,-4.4)--(2.7, -5.7);
					\draw (2.0,-5) node {\tiny{$\m_3$}};
					\draw (4,-6) node {\tiny{$[-3,1,2]$}};
					\draw [->](-0.3,-4.4)--(-2.7, -5.7);
					\draw (-2,-5) node {\tiny{$\m_1$}};
					\draw (-4,-6) node {\tiny{$[1,-3,-2]$}};
					\draw [->](-2.7,-6.4)--(-0.3, -7.6);
					\draw (-2,-7) node {\tiny{$\m_3$}};
					\draw (0,-8) node {\tiny{$[1,-3,2]$}};
					\draw [->](2.7,-6.4)--(0.3, -7.6);
					\draw (2,-7) node {\tiny{$\m_1$}};
					\draw (0,-10) node {\tiny{$[1,2,-3]$}};
					\draw [->](0,-8.3)--(0, -9.7);
					\draw (-0.5,-9) node {\tiny{$\m_2$}};
				\end{scope}
			\end{tikzpicture}
		\end{center}
	\end{eg}

\bibliography{refs.bib}
\bibliographystyle{alpha}

\end{document}